\documentclass[10pt]{article}

\usepackage{amsthm,amsfonts,amsmath,amscd,amssymb,times,epsfig,verbatim}
\usepackage[all]{xy}
\usepackage{enumerate,array}

\usepackage{color}

\newtheorem{thm}{Theorem}[section]
\newtheorem{prop}[thm]{Proposition}
\newtheorem{lem}[thm]{Lemma}
\newtheorem{cor}[thm]{Corollary}

\theoremstyle{remark}
\newtheorem{rem}[thm]{Remark}
\theoremstyle{definition}
\theoremstyle{axiom}
\newtheorem{ax}{Axiom}
\newtheorem{defn}[thm]{Definition}
\newtheorem{ex}[thm]{Example}

\newcommand{\C}{\mathbb{C}}

\newcommand{\R}{\mathbb{ R}}

\newcommand{\rk}{\operatorname{rk}}

\newcommand{\Z}{\mathbb{ Z}}

\newcommand{\sg}{\mathrm{sign}}

\title{The foundations of $(2n,k)$-manifolds }
\author{Victor M.~Buchstaber and Svjetlana Terzi\'c}

\begin{document}

\maketitle
\begin{abstract}
In the focus of our paper is a system of axioms that  serves as a basis  for introducing  structural data for $(2n,k)$-manifolds $M^{2n}$, where $M^{2n}$ is a smooth, compact $2n$-dimensional manifold with a smooth effective action of the $k$-dimensional torus $T^k$. In terms of these data  a construction of the model space $\mathfrak{E}$ with an action of the torus $T^k$ is given, such that there exists a  $T^k$-equivariant homeomorphism $\mathfrak{E}\to M^{2n}$. This homeomorphism induces a homeomorphism $\mathfrak{E}/T^k\to M^{2n}/T^k$.
The number $d=n-k$ is called the complexity of an  $(2n,k)$-manifold. Our theory comprises   toric geometry and toric topology, where $d=0$. It is shown that  the class  of homogeneous spaces  $G/H$ of compact Lie groups, where $\rk G=\rk H$,  contains $(2n,k)$-manifolds that  have non zero complexity. The results are demonstrated on the complex Grassmann manifolds $G_{k+1,q}$ with an effective action of the torus $T^k$.
\end{abstract}

\tableofcontents

.


\section{Introduction}

The goal of this  paper is to extend the issues  from~\cite{BTO}, specify the axioms and present the new results of our theory of   $(2n,k)$ - manifolds. It is about  a wide class of $2n$ - dimensional smooth,  compact, oriented manifolds with an effective action of the  compact torus $T^k$ having only isolated fixed points. We propose the tools for an effective description of  the equivariant  structure of such manifolds as well as the structure of their orbit spaces.

This class contains toric and quasitoric manifolds $M^{2n}$ , whose effective description in toric topology (see~\cite{BP}) is given by the  combinatorial data  $(P^{n},\Lambda)$, where $P^{n}$ is a $n$-dimensional  simple polytope and $\Lambda$ is a characteristic function from the set of facets of the polytope $P^{n}$ to the lattice $\Z ^{n}$ that  satisfies Davis-Januszkiewicz (*) condition~\cite{DJ-91}. In this case,  a   $(2n,n)$-manifold is obtained, while  the orbit space $M^{2n}/T^n$ is homeomorphic to the polytope $P^{n}$. In this paper are described the key examples of the manifolds $M^{2n}$ with an effective action  of the torus $T^k$ which form a basis for the theory of $(2n,k)$-manifolds.

Any $(2n,k)$-manifold  is equipped with  the so called  almost moment map $\mu : M^{2n}\to \R ^k$, whose image is a convex polytope $P^k$.  The polytope  $P^k$ does not need to be simple  and the orbit space $M^{2n}/T^k$  for $k<n$ is not homeomorphic to $P^k$ in general.  For the  Grassmann manifolds $G_{k+1,q},$ the polytope  $P^k$ is the hypersimplex $\Delta _{k+1,q}$ and for the complex flag manifolds $F_{k+1}$,  the polytope $P^{k}$ is the permutahedron $P{e}^{k}$. Note that $\Delta _{k+1, q}$ is a simple polytope only for $q=1$ or $q=k$, while $P{e}^{k}$ is a simply polytope for any $k$.


One of the main tools, which  we introduce, is a family of admissible polytopes $P_{\sigma}$ which   are spanned by some subsets $\sigma$ of vertices of the polytope $P^k$. In the quasitoric case the family of admissible polytopes coincides with the family of the  faces of
the simple polytope $P^k$ including  the polytope $P^{k}$. In the case of Grassmann manifolds $G_{k+1,q}$ our admissible polytopes coincide with the admissible polytopes from the paper~\cite{GS} in which  a number of results about these polytopes are obtained .

To each admissible polytope $P_{\sigma}$ corresponds a $T^k$-invariant subspace $W_{\sigma}\subset M^{2n},$ called the stratum,   and a  subtorus $T_{\sigma}\subseteq T^k$ which acts {\it  trivially} on $W_{\sigma}$ such that the torus $T^{\sigma}=T^{k}/T_{\sigma}$,   $\dim T^{\sigma}=\dim P_{\sigma}$  acts freely on $W_{\sigma}$. The polytope $P^k$  is considered to be an admissible polytope and the corresponding  stratum $W$ is called the  main stratum. One of our axioms requires  that $\mu (W_{\sigma}) = \stackrel{\circ}{P_{\sigma}}$  and  that the  restriction of the induced almost moment map $\hat{\mu}$ to $W_{\sigma}/T^{\sigma}$ is a  projection of a fiber bundle with the base    $\stackrel{\circ}{P_{\sigma}}$. This implies that $W_{\sigma}/T^{\sigma} \cong \stackrel{\circ}{P_{\sigma}}\times F_{\sigma}$, for some topological space $F_{\sigma}$, which is called the  space of parameters of the stratum $W_{\sigma}$.  
In the case of Grassmann manifolds, our strata   $W_{\sigma}$ coincide with the  strata of  Gel'fand, MacPherson, Goresky and Serganova, which they introduced using the action  of the algebraic torus $(\C ^{*})^{k+1}$ on $G_{k+1,q}$. In this case, the subtorus $(\C ^{*})^{\sigma}$ acts freely on $W_{\sigma}$  and $F_{\sigma} = W_{\sigma}/(\C ^{*})^{\sigma}$. Note that in general case, an action of the compact torus  $T^k$ on a $(2n,k)$-manifold $M^{2n}$ does not extend to an effective action of the algebraic torus $(\C ^{*})^{k}$. 

The union of all admissible polytopes $P_{\sigma}$ is the polytope $P^k$  and to each point $x\in P^k$ we assign the cort\'ege $\sigma (x) = \{ P_{\sigma} : x\in \stackrel{\circ}{P_{\sigma}}\}$. We assume  that $\mu$ is a smooth map and obtain that  if  $\dim P_{\sigma}=k$ for any $P_{\sigma}\in \sigma (x)$ then  $x$ is a regular value for $\mu$.

Towards  our goal  to describe the  equivariant topology of an $(2n,k)$-manifold   $M^{2n}$ and its orbit space  $M^{2n}/T^k$,  we introduce  structural data and  obtain a  {\it model}  for the space  $M^{2n}/T^k$ in  terms of these  structural data. Our structural data consist of the  {\it virtual} spaces of parameters $\tilde{F_{\sigma}}$ together with the continuous projections $p_{\sigma} : \tilde{F}_{\sigma} \to F_{\sigma}$, and the  {\it universal} space of parameters $\mathcal{F}$  together with the  embeddings  $I_{\sigma} : \tilde{F}_{\sigma} \to \mathcal{F}$.    Note that  for the main stratum, the virtual space of parameters $\tilde{F}$ coincides with the space of parameters $F$. We use the fact that the  orbit space  of the main stratum $W/T^k \cong \stackrel{\circ}{P^{k}}\times F$ is  a dense set in $M^{2n}/T^k$.  It is 
required that the compactification of  $\stackrel{\circ}{P^{k}}\times F$, which corresponds to the compactification $\overline{W/T^k}= M^{2n}/T^k$, is realized by the topology of the {\it universal} space of parameters $\mathcal{F}$, which is a compactification   of the space of parameters $F$ of the main stratum such that $\mathcal{F}= \cup _{\sigma}I_{\sigma}(\tilde{F}_{\sigma})$.   We realized in detail this approach in~\cite{BT-1} for the Grassmann manifold $G_{4,2}$ and in~\cite{BTN} for the Grassmann manifold $G_{5,2}$. In general case, we obtain the orbit spaces  $M^{2n}/T^k$ as a   quotient space of the union  $\cup _{\sigma} (\stackrel{\circ}{P_{\sigma}}\times \tilde{F}_{\sigma})$ by an equivalence relation which is defined in terms of the structural maps $I_{\sigma}: \tilde{F}_{\sigma}\to \mathcal{F}$ and $p_{\sigma} : \tilde{F}_{\sigma}\to F_{\sigma}$.

The complexity of an $(2n,k)$-manifold is defined to be  the number $d=n-k$. Our definition of the complexity of an action generalizes the definition of the complexity of an  algebraic torus action $(\C ^{*})^{k}$  in algebraic geometry and symplectic geometry.    The  complexity $1$ torus actions,  under some appropriate assumptions,   are widely studied in algebraic  and symplectic geometry (see ~\cite{T, K}).  In the recent  paper~\cite{AA}, inspired by our work~\cite{BT-1, BTN},  the approach  is described  for  solving the classification problem of complexity $1$ torus actions in terms of equivariant topology. The  problem of the torus actions of complexity $2$ or more is still not well understood and, in the literature,   it is considered   to be quite difficult. Our theory leads  to the results in this direction. In the papers ~\cite{BT-1, BTN}  the orbit spaces $\C P^{5}/T^4$ and $G_{5,2}/T^5$ of the complexity $2$ actions on an  $(10, 3)$ and $(12, 4)$-manifolds, respectively  are described  in detail.

The axioms and methods  of this paper  rely on the well known  results on the algebraic torus action on Grassmann manifolds~\cite{GM,GS,GGMS,Kap}, as well as on  our results on   the  orbit spaces of  the compact torus action on Grassmann manifolds~\cite{BT-1, BTN}. We emphasize the Grassmann manifolds $G_{k+1,2}$ with an effective action of the torus $T^{k}$ of complexity $k-2$, $k\geq 2$, since   many papers have been devoted  to them in the recent time, due to their connection with the moduli spaces of curves (see, for example,~\cite{Kap, Kap-2, Keel}).


\section{The key examples of  manifolds with torus actions}
In this section we provide the key examples  and  the basic facts,  which   served as a  starting point for establishing  axioms  for the theory of  $(2n,k)$-manifolds.


\subsection{Quasitoric manifolds}\label{qt}
A quasitoric manifold is a topological analog of a non-singular projective toric variety from algebraic geometry.  We follow
~\cite{BP} and recall that a quasitoric manifold is a  smooth, closed manifold $M^{2n}$ equipped with a  smooth action of the torus $T^n$ such that:
\begin{itemize}
\item  the  action of the torus $T^n$ on $M^{2n}$ is locally standard;
\item the orbit space $M^{2n}/T^n$ is diffeomorphic to  a  simple polytope $P^n$ as a manifold with corners.
\end{itemize}
The second condition gives that there exists   a  smooth map $\mu : M^{2n}\to P^{n}$ that is constant on $T^n$-orbits and maps an $p$-dimensional orbit to an  interior point of some $p$-dimensional face of $P^{n}$.    In follows that    $\mu ^{-1}(\stackrel{\circ}{P^{n}}), $ is a  dense  set in $M^{2n}$, the   action of the torus $T^{n}$ is free on  $\mu ^{-1}(\stackrel{\circ}{P^{n}})$  and the vertices of  the  polytope $P^{n}$ correspond to the fixed points of $T^n$-action.  We recall the  notion of the characteristic map and the characteristic  matrix for  a  quasitoric manifold $M^{2n}$ which, together with the combinatorics of the  polytope $P^{n}$, determine  an equivariant topology and cohomology of a  manifold $M^{2n}$.

Let $\{F_1, \ldots , F_{m}\}$ be the  set of all facets of $P^{n}$. The stationary subgroups $T(F_i)$ for the faces $F_i$ are  one-dimensional connected subgroups  in $T^n$ and they  can be written as   $T(F_i) = (e^{2\pi\sqrt{-1}\lambda_{1i}\varphi},\ldots ,e^{2pi \sqrt{-1}\lambda _{ni}\varphi})$, where $\varphi \in \R$ and $\lambda _{i} = (\lambda _{1i}, \ldots , \lambda _{ni})\in \Z ^{n}$.  Denote by  $S(T^n)$  the  set of all connected subgroups of the torus $T^n$.  The characteristic map $l: \{ F_{i}\} \to S(T^n)$,  which is defined by  $l: F_i\to T(F_i)$,  can be described using  the  characteristic matrix  $\Lambda$ whose  columns are the integer vectors $\lambda _{i}$, $1\leq i\leq m$ which satisfy the following condition: if the intersection $F= F_{i_1}\cap \ldots \cap F_{i_n}$  is a vertex of the polytope $P^n$, then the vectors $\lambda _{i_1}, \ldots , \lambda _{i_n}$ form a  basis for $\Z ^{n}$.   Due to Davis-Januszkiewicz theorem (see~\cite{DJ-91}), the matrix $\Lambda$ and  the combinatorics of the  polytope $P^{n}$ determine together  the cohomology of  $M^{2n}$.

Let $\bf{F}$ be the partially ordered  set of all faces for $P^{n}$.   The  points from $\mu ^{-1}(F)\subset M^{2n}$ have the same stabilizer  for any $F\in \bf{F}$, so the characteristic map  extends to the map   ${\bf F}\to S(T^n)$, which to each face  $F$ assigns  the stationary subgroup of the set  $\mu ^{-1}(F)$. More precisely, the face  $F= F_{i_1}\cap \ldots \cap F_{i_k}$  maps to the image of the subgroup  $T(F_{i_1})\times \cdots \times T(F_{i_k})\subset  T^n$.  The  map    ${\bf F}\to S(T^n)$ is completely  determined by the matrix $\Lambda$ and it is denoted  by $\Lambda$ as well.  A quasitoric manifold  $M^{2n}$ can be recovered, up to diffeomorphism, using   the  characteristic pair $(P^n, \Lambda )$. In other words, it can be constructed  a  model for $M^{2n}$  by:
\begin{equation}\label{recov}
M \cong (T^n \times P^{n})/\approx,\;\; (t_1,p_1)\approx (t_2, p_2)\; \text{if and only if }\; p_1=p_2,\; t_1t_2^{-1}\in \Lambda (F(p_1)),
\end{equation}
where $F(p_1)$ is the smallest face of the polytope  $P^{n}$ that  contains $p_1$.

In this case,   for any point $p\in P^{n}$,  the cort\'ege $\sigma (p) = \{ P_{\sigma} : p\in \stackrel{\circ}{P}_{\sigma}\}$ consists
of one polytope, that is the face $F(p)$.


 \subsection{Complex Grassmann manifolds $G_{k+1,q}$}\label{gr}

The complex Grassmann manifold $G_{k+1,q}$ consists of all  $q$-dimensional complex subspaces in the complex vector space $\C^ {k+1}$. The canonical action of the torus  $T^{k+1}$ on  $\C ^{k+1}$,  considered  in the canonical basis,  induces the action of the torus  $T^{k+1}$ on the manifold $G_{k+1,q}$.
This action is not effective,  as the diagonal subgroup  $\Delta  =\{(t,\ldots ,t), t\in S^1\}$ acts trivially on $G_{k+1,q}$. The torus $T^k = T^{k+1}/H$ acts effectively on $G_{k+1,q}$.

We recall some classical constructions on the complex Grassmann manifolds. After fixing  a basis in an   $q$-dimensional subspace $L\subset \C ^{k+1}$, this subspace  can be    represented by the $q\times (k+1)$ matrix $A(L)$ such that   $\text{rank} A(L)=q$. For any subset $J\subset \{1,\ldots ,k+1\}$ consisting of $q$ elements, $|J|=q$,  denote by  $A_{J}(L)$  the  matrix of dimension  $q\times q$ given by the columns of  the matrix $A_{L}$ that  are indexed by $J$. We will assume that  the set of all subsets $J = \{ j_1<j_2\ldots <j_q\}\subset \{1,\ldots ,k+1\}$ is  ordered lexicographically. Using this ordering  define the vector
\[
P(A(L))= (P^{J}(A(L))) = (\det A_{J}(L)),
\]

whose coordinates are called the Pl\"ucker coordinates of a point  $L\in G_{k+1,q}$.

The Pl\"ucker coordinates depend on a fixed basis for $L$ and they are, up to constant, uniquely defined. More precisely,    two bases $f_1,\ldots, f_q$ and $e_1,\ldots ,e_q$  for a subspace $L$ are related by $f_j=\sum _{i=1}^{q}\alpha _j^ie_i$, $1\leq j\leq q$. It implies that   the Pl\"ucker coordinates for $L$ in these two bases are related by
\begin{equation}\label{plucker}
P^{J}_{f}(L) = (\sum_{\sigma \in S_k}\sg (\sigma) \alpha _{\sigma (1)}^1\cdots \alpha _{\sigma (k)}^k)P^{J}_{e}(L) = \text{det}(\alpha),
\end{equation}
where $\alpha = (\alpha _{j}^{i})$ is a  transition  matrix  between these two bases.
In this way the Pl\"ucker coordinates produce an embedding of the Grassmann manifold $G_{k+1,q}$ into $\C P^{N-1}$, where $N={k+1\choose q}$.

The Pl\"ucker coordinates define the smooth atlas $\{(M_{I}, u_{I})\}$ on $G_{k+1, q}$, where $I$ runs through all $q$-element subsets of the set $\{1,\ldots , k+1\}$, as follows. Here  $M_{I} = \{ L\in G_{k+1, q} : P^{I}(L)\neq 0\}$ and the coordinate map $u_{I} : M_{I}\to \C ^{q(k+1-q)}$ is defined by the Pl\"ucker coordinates $P^{J}(L)$, $J=(I\setminus \{i_p\})\cup \{j_s\}$, $i_p\in I, j_s\in \{1, \ldots, k+1\}\setminus I$, in  such a basis of a subspace  $L$ that the $(q\times q)$-dimensional sub-matrix of the matrix $A(L)$ whose columns are indexed by $I$ is an identity matrix.

 Let us consider the action of $T^{k+1}$ on  $\C P^{N-1}$,  which is given by the composition of  the $q$-th exterior power representation $T^{k+1}\to T^{N}$ and the standard action of  $T^{N}$ on $\C P^{N-1}$. The  standard moment map  $\C P^{N-1}\to \R^{k+1}$  for  such an  action of the torus   $T^{k+1}$  induces  the moment map $\mu : G_{k+1,q} \to \R ^{k+1}$ ( see~\cite{Kir}), which is defined by
\begin{equation}\label{gr-moment}
\mu (L) = \frac{\sum _{J}|P^{J}(L)|^2\delta _{J}}{\sum _{J}|P^{J}(L)|^2},
\end{equation}
where $\delta _{J}\in \R ^{k+1}$ are the vectors whose coordinates are given by
\[
(\delta _{J})_{i}=1,\;\; i\in J,\;\; (\delta _{J})_{i}=0,\;\;  i\notin J,
\]
and $J$ runs through the $q$-element  subsets of $\{1,\ldots, k+1\}$.

The map $\mu$ is $T^{k+1}$-invariant and  the  image of $\mu$ is, by its definition,  a convex hull over the points $\delta _{J}$.    The convex polytope obtained in this way is known as the hypersimplex  $\Delta _{k+1,q}$, see~\cite{Z}. In particular, $\Delta _{4,2}$ is the  octahedron.

Recall that the Grassmann manifold $G_{k+1, q}$ admits as well an action of the algebraic torus $(\C ^{*})^{k+1}$, which is induced by the coordinate wise action of $(\C ^{*})^{k+1}$ on $\C ^{k+1}$.  The orbit $(\C ^{*})^{k+1}\cdot L$ is a smooth submanifold in $G_{k+1, q}$  for any point  $L\in G_{k+1, q}$. The stationary subgroup $(\C ^{*})^{k+1}_{L}$ of a point  $L$  is a toral subgroup in  $(\C ^{*})^{k+1}$ and the algebraic torus $(\C ^{*})^{L}=(\C ^{*})^{k+1}/(\C ^{*})^{k+1}_{L}$ acts freely on  the orbit $(\C ^{*})^{k+1}\cdot L$. Moreover, $\mu ((\C ^{*})^{k+1}\cdot L) = \stackrel{\circ}{P_L}$, where $P_{L}$ is a convex polytope spanned by the  vertices $\delta _{J}$ of the hypersimplex  $\Delta _{n,k}$ indexed by those $J$  such that $P^{J}(L)\neq 0$.


\section{Definition of  $(2n,k)$-manifolds}\label{jedan}

We assume the following to be given:
\begin{itemize}
\item  a smooth, closed, oriented manifold $M^{2n}$;
\item  a smooth, effective action $\theta$ of the torus $T^{k}$ on $M^{2n}$, where $1\leq k \leq n$, such that the stabilizer of any point is a connected subgroup of $T^k$;
\item   a smooth,  $\theta$-equivariant  map $\mu : M^{2n}\to \R^k$,  whose image is a $k$-dimensional convex polytope $P^k$, where $\R ^k$  is considered with the trivial $T^k$ - action. We assume $\mu : M^{2n}\to \R^{k}$ to be an open map. 
\end{itemize}
The map $\mu $ we call an {\it almost moment map} for the given $T^k$-action on $M^{2n}$.

We say that the  triple  $(M^{2n}, \theta ,\mu )$ is an $(2n,k)$-manifold if it satisfies the six axioms which we formulate below.


\subsection{A smooth manifold structure.}

\begin{ax}\label{atlas}
There exists  a smooth atlas $\mathfrak{M} = \{(M_{i}, \varphi _i)\}_{i\in I}$,  where  $M_{i}$ are an open subsets in $M^{2n}$ and $\varphi _i : M_{i}\to \R ^{2n}$   are coordinate homeomorphisms.  Any  chart  $M_{i}$ is $T^k$-invariant, contains  exactly one fixed point $x_{i}$ with $\varphi _{i}(x_i) = (0,\ldots ,0)$, such that   $x_i\neq x_j$ for $i\neq j$, The   closure of any chart $M_{i}$   is  the whole  manifold $M^{2n}$.
\end{ax}

 Any atlas that  satisfies  Axiom~\ref{atlas} has finitely many charts, 
since  $M^{2n}$ is  a compact manifold.  It implies:

\begin{cor}
The action of $T^{k}$ on $M^{2n}$ has finitely many isolated fixed points.
\end{cor}

By $m$ denote the  number of fixed points for $T^k$-action on $M^{2n}$. We enumerate as $(M_1,\varphi _1),\ldots ,(M_{m},\varphi _{m})$ the charts given by Axiom~\ref{atlas}   The sets $Y_{i} = M^{2n}\setminus  M_{i}$ are closed in $M^{2n}$ and $T^k$-invariant, by their definition.

 There is a standard concept of the  boundary for  a subset $Y$ of a  topological space $X$ in general topology. There is also a concept of the   boundary  of a manifold, manifold with corners and, in that context the boundary of a convex polytope, used in algebraic topology and differential geometry.  In all these cases $\partial$ is a standard notation for the  boundary. It is not difficult to realize that all these concepts are  {\it not always}  appropriate  for some purposes  in the theory of $(2n,k)$-manifolds.  Therefore,   we introduce  the new notion of the  boundary and the  corresponding symbol. The boundary $\bar{\partial}$ of a  subset $Y$ in a  topological space $X$  is the set  $\bar{\partial}Y  = \bar{Y}\setminus Y$, where $\bar{Y}$ is the closure of a  set $Y$ in a  space $X$. 
Note that if $Y$ is an open set   in $X$  , then $\bar{\partial}Y = \partial Y$, where $\partial Y = \bar{Y}\cap\overline{X\setminus Y}$ is the standard boundary as defined in general topology.  In the sequel,  set $\bar{\partial}Y\subset X$  is  called    the  $\bar{\partial}-$boundary  of a  set $Y$ as well.


Therefore,  the fact that  the above defined  sets $M_{i}$ are   dense, open sets  in $M^{2n}$,    implies that $Y_{i} = \partial M_{i}= \bar{\partial} M_{i}$.

For any $\sigma = \{i_1,\ldots ,i_l\}\subseteq [1,m]$  let us consider the set   :
\[
W_{\sigma} = M_{i_1}\cap \cdots \cap M_{i_l}\cap Y_{i_{l+1}}\cap \cdots \cap Y_{i_{m}},
\]
where $\{i_{l+1},\ldots ,i_{m}\} = [1,m] \setminus \{i_1,\ldots ,i_l\}$.

\begin{defn}
The non-empty set  $W_{\sigma}$  is said to be a stratum. The  index set $\sigma \subset [1,m]$  of a  stratum  $W_{\sigma7}$ is   said to be an  admissible set.
\end{defn}

\begin{lem}
The strata  $W_{\sigma}$ are $T^k$-invariant, pairwise disjoint and their union is the whole manifold $M^{2n}$.
\end{lem}

\begin{proof}
Any stratum  $W_{\sigma}$ is $T^k$-invariant since the sets  $M_{i}$ and $Y_{i}$ are $T^k$- invariant.  Further,
if  $W_{\sigma _1}\neq W_{\sigma _2}$ then   $\sigma _1\neq \sigma _2$ and, thus, we can assume that  there exists $i\in \sigma _1$ such that $i\notin \sigma _2$, $1\leq i\leq m$. Therefore, for $x\in W_{\sigma _1}$ it follows  that $x\in M_{i}$ and, thus,  $x\notin Y_{i}$,  which  gives that $x\notin W_{\sigma _2}$. Similarly,   if $x\in W_{\sigma _2}$ then $x\in Y_{i}$ and thus  $x\notin M_{i}$,  which implies $x\notin W_{\sigma _1}$.    The union of all strata  is  the manifold $M^{2n}$,  since the  charts cover  $M^{2n}$.  Note that  it is uniquely defined an admissible set $\sigma =\sigma (x) = \{i\in [1,m] | x\in M_{i}\}$  for any point $x\in M^{2n}$.
\end{proof}

This Lemma together with the fact that the $T^k$-action on $M^{2n}$ is continuous implies that the  closure $\overline{W_{\sigma}}$ is a  $T^k$-invariant set for any stratum $W_{\sigma}$.

\begin{ex}\label{main}
The set $W_{[1,m]} = M_1\cap \cdots \cap M_{m}$ is non-empty and, thus, it is  a stratum  and the set $\sigma = [1,m]$ is an admissible set. The set  $W_{[1,m]}$  is an  open dense set in $M^{2n}$ since, by Axiom~\ref{atlas} any  chart $M_i$, $1\leq i\leq m$ is an  open,  dense set in $M^{2n}$. The stratum $W_{[1,m]}$ is called the  main stratum  and it is  further denoted by  $W= W_{[1,m]}$.
\end{ex}
\begin{ex}\label{one}
The set  $W_{\{i\}} = M_{i}\cap \left ( \bigcap _{j\neq i} Y_{j} \right )$ is a stratum  and the set $\sigma = \{i\}$ is an admissible set  for any $1\leq i\leq m$. It follows from the observation that the set  $W_{i}$ is non-empty since, by Axiom~\ref{atlas}  the fixed point $x_i$ which belongs to
$ M_{i}$ also belongs  to all $Y_j$, $j\neq i$, $1\leq j\leq m$.
\end{ex}

\begin{rem}
Since $M^{2n}=M_1\cup \cdots \cup M_{m}$, we see  that $W_{\emptyset}=Y_1\cap \cdots \cap Y_{m}=\emptyset$, which implies that $W_{\emptyset}$ is not a stratum and $\emptyset$ is not an admissible set.
\end{rem}
\begin{rem}\label{open}
A stratum $W_{\sigma}$ different from the main stratum $W$ is not  an open set in $M^{2n}$. It follows from the observations that in this case there exists a chart $M_{i}$  such that $M_{i}\cap W_{\sigma} =\emptyset$ and that $M_{i}$ is a dense set in $M^{2n}$.
\end{rem}

\begin{lem}\label{bound}
The boundary $\bar{\partial} W_{\sigma}$ of a  stratum  $W_{\sigma}$ is contained in the union of the strata  $W_{\tilde{\sigma}}$, where $\tilde{\sigma}$ are the admissible sets such that $\tilde{\sigma}\subset \sigma$.
\end{lem}

\begin{proof}
Let $W_{\sigma} = M_{i_1}\cap \cdots \cap M_{i_l}\cap Y_{i_{l+1}}\cap \ldots \cap Y_{i_m}$. Then the boundary $\partial W_{\sigma}$ is contained in the union of the sets $
 \left (\bigcap \limits_{q=1}^{p}\bar{\partial} M_{i_{j_q}}\right )\cap \left (\bigcap \limits_{s\neq j_1,\ldots , j_p}M_{i_s}\right )\cap \left (\bigcap \limits _{j=l+1}^{m}Y_{i_{j}}\right )=
 \left (\bigcap \limits _{q=1}^{p} Y_{i_{j_q}}\right )\cap \left (\bigcap \limits_{s\neq j_1,\ldots , j_p}M_{i_s}\right )\cap \left (\bigcap \limits_{ j=l+1}^{m}Y_{i_{j}}\right )$, where $j_1<\ldots <j_{p}$ and $1\leq p\leq l$.
Such nonempty sets give $W_{\tilde{\sigma}}$, where $\tilde{\sigma} = \sigma \setminus \{i_{j_1}, \ldots i_{j_p}\}$.
Hence, $\bar{\partial} W_{\sigma} \subseteq  \cup W_{\tilde{\sigma}}$, where ${\tilde{\sigma}}$ goes through all  proper admissible subsets  of $\sigma$.
\end{proof}


\section{Almost moment map,  strata and admissible  polytopes}

\begin{ax}\label{bij}
The map $\mu$ is a  bijection between  the set of fixed points   and the set of vertices of the polytope $P^k$.
\end{ax}

Since  an  $k$-dimensional polytope has at least $k+1$ vertices,  it  follows:

\begin{cor}
The number of fixed points for $T^k$-action on $M^{2n}$ is not less then $k+1$.
\end{cor}

Let $S(P^k)$ be a family of all convex polytopes that are spanned by the vertices of the polytope $P^k$. By $\mathfrak{S}$ denote the set of all admissible sets.
Using the almost moment map $\mu$,  we define the  map
$s : \mathfrak{S}\to  S(P^k)$ as follows. Put   $v_i=\mu (x_i)$, $1\leq i\leq m$ and  let  $\sigma  = \{i_1,\ldots ,i_l\}\subseteq [1,m]$.  By Axiom~\ref{atlas}, for any  $i_{j}\in \sigma$  there exists a unique fixed point    $x_{i_j}\in M_{i_j}$ .
We put
\[
s(\sigma) = P_{\sigma},\; \text{where}\; P_{\sigma} =\text{convhull}(v_{i_1},\ldots ,v_{i_l}).
\]

\begin{defn}
The   polytope $P_{\sigma} \in S(P^k)$ is said to be an admissible polytope if  it is in the  image of the map
$s : \mathfrak{S}\to  S(P^k)$.
\end{defn}

\begin{rem}
Since $W_{\emptyset} = \emptyset$,  it follows that  $\emptyset$ is not an admissible polytope.
\end{rem}

\begin{ex}\label{adm_P}
The polytope $P^{k}$ is an admissible polytope. The set $\sigma = \{1,\ldots ,m\}$ is an admissible set as it is shown in Example~\ref{main}. In addition, by Axiom~\ref{bij}, $s(\sigma)$ is  a convex hull of all vertices of $P^k$, which  implies that  $s(\sigma) = P^{k}$.
\end{ex}

\begin{ex}\label{adm_v}
Any vertex $v_i$ of $P^{k}$ is an admissible polytope.   To see that, by Axiom~\ref{bij}, take  the fixed point $x_i$ such that $\mu (x_i) = v_i$. Then the set $\sigma = \{i\}$ is an admissible set as it is shown in Example~\ref{one} and $s(\sigma ) = v_i$.
\end{ex}

\begin{defn}
 The set of all admissible polytopes is said to be pure if any admissible polytope of the dimension $\leq k-1$ is a face of some admissible polytope of the dimension $k$.
\end{defn}

\begin{ex}
The set of admissible polytopes of a quasitoric manifold is a pure set, Moreover it follows from~\cite{BT-1}  and~\cite{BTN} (Proposition 9, Propositions 11-15 and Corollary 18)  that the set of admissible polytopes for the Grassmann manifolds $G_{4,2}$ and $G_{5,2}$ are  pure sets as well.
\end{ex}

\begin{rem}
For a general $(2n,k)$-manifold,  two  admissible polytopes $P_{\sigma _1}$ and $P_{\sigma _2}$ may have nonempty intersection 
$ \stackrel{\circ}{P}_{\sigma _1}\cap \stackrel{\circ}{P}_{\sigma _2}$. For example, one can verify this   in the case  of complex Grassmann manifold $G_{4,2}$ which is an  $(8,3)$-manifold, see~\cite{BT-1}.
\end{rem}

\begin{defn}
The point $p\in P^k$ is said to be an exceptional point   if $p\in \stackrel{\circ}{P}_{\sigma _1}\cap \stackrel{\circ}{P}_{\sigma _2}$ for some different  admissible polytopes  $P_{\sigma _1}, P_{\sigma _2}$.  Otherwise, it is said to be simple.
\end{defn}

In this way, the set of exceptional  points $S\subseteq P^{k}$  is defined   .


\section{Stabilizers for the torus action on the strata}

By $S(T^k)$ denote, as above, the set  of all connected subgroups of the torus $T^k$. Note that a connected subgroup of the torus  $T^k$ is a torus.
Let us consider a function $\chi : M^{2n}\to S(T^k)$ which to any point $x$ assigns its stabilizer  $\chi (x)$  regarded to  the given $T^k$-action
on $M^{2n}$. It follows from the  set-up assumptions that    $\chi (x)$ is  a connected subgroup of the torus  $T^k$. We assume the following to be satisfied:

\begin{ax}\label{stab}
 The characteristic function $\chi$ is constant on any stratum  $W_{\sigma}$.
\end{ax}

Using Axiom~\ref{stab}, the torus $T^{\sigma} = T^{k}/\chi (W_{\sigma})$ can be defined  for any stratum  $W_{\sigma}$.

\begin{cor}
The torus $T^{\sigma}$ acts freely on $W_{\sigma}$, which  gives the  principal bundle
\begin{equation}
T^{\sigma}\to W_{\sigma}\to W_{\sigma}/T^{\sigma}.
\end{equation}
\end{cor}

It is shown in~\cite{BTN} (Remark 3)  that,   in the case of  Grassmann manifolds,  the notion of the strata  as defined in~\cite{GS}, coincides with our notion of the strata. In addition,  one verifies~\cite{BTN} (Proposition 4) that the characteristic function is constant on any  stratum of the Grassmann manifolds.


\section{Orbit spaces of the strata}
By its definition the  almost moment map $\mu : M^{2n}\to P^k$ is  $T^k$ - invariant . Therefore,  it   induces the map  $\widehat{\mu} : M^{2n}/T^k \to P^k$.  

\begin{ax}~\label{fiber}
The almost moment map $\mu$:
\begin{itemize}
\item [a)] maps a   stratum  $W_{\sigma}$ onto $\stackrel{\circ}{P_{\sigma}}$,
\item [b)] induces  the   fiber bundle  $ \widehat{\mu}_{\sigma} : W_{\sigma}/T^{\sigma}\to \stackrel{\circ}{P_{\sigma}}$,
\item [c)] $\dim P_{\sigma} = \dim T^{\sigma}$.
\end{itemize}
\end{ax}

An immediate consequence of this Axiom is:
\begin{cor}
\begin{itemize} It holds
\item
$\widehat{\mu}(W/T^k) = \stackrel{\circ}{P^{k}}$  for the  main stratum  $W = M_1\cap\cdots \cap M_{m}$,
\item  $\mu (W_{\{i\}}) = \{v_i\}$, where  $v_{i}$  is a vertex and $W_{\{i\}} = M_{i}\cap \left (\bigcap _{j\neq i} \right ) Y_{j}$.

\end{itemize}
\end{cor}

\begin{rem}
Note that  Axiom~\ref{fiber} does not require a   fiber bundle   $ \widehat{\mu}_{\sigma} : W_{\sigma}/T^{\sigma}\to \stackrel{\circ}{P_{\sigma}}$ to be smooth, since there is no argument to claim that,  in general,  a stratum  $W_{\sigma}$  is   a smooth submanifold in $M^{2n}$.  The main stratum being open is of course a smooth submanifold,  but for the other strata it does not have to be the case. Even for the Grassmann manifolds, the   differential geometry of the strata can be very complicated, see~\cite{eric}.
\end{rem}

By $[F_{\sigma}]$ denote the  homeomorphic type of a fiber for  the fiber bundle $ \widehat{\mu}_{\sigma} : W_{\sigma}/T^{\sigma}\to \stackrel{\circ}{P_{\sigma}}$.
\begin{defn}\label{param}
The space $F_{\sigma}$ of a homeomorphic type $[F_{\sigma}]$ is called  the space of parameters  of a stratum $W_{\sigma}$.
\end{defn}

Since $\stackrel{\circ}{P_{\sigma}}$ is contractible, for the fiber bundle $ W_{\sigma}/T^{\sigma}\to  \stackrel{\circ}{P_{\sigma}}$ we conclude that  the following holds:   

\begin{cor}\label{homeom}
The  fiber bundle $\widehat{\mu}_{\sigma} : W_{\sigma}/T^{\sigma}\to  \stackrel{\circ}{P_{\sigma}}$  is isomorphic to the trivial bundle. That is $W_{\sigma}/T^{\sigma}$ is homeomorphic to $\stackrel{\circ}{P_{\sigma}} \times F_{\sigma}$ by the fiber wise homeomorphism
\begin{equation*}\begin{CD}
 W_{\sigma}/T^{\sigma} @>>>  \stackrel{\circ}{P_{\sigma}} \times F_{\sigma}\\
  @VV{\widehat{\mu}}V  @VV{}V\\
\stackrel{\circ}{P_{\sigma}} @= \stackrel{\circ}{P_{\sigma}}.
\end{CD}\end{equation*}
\end{cor}

\begin{defn}
For any $\sigma \in\mathfrak{S}$,  we fix  the space $F_{\sigma}$ and the trivialization  $h_{\sigma} : W_{\sigma}/T^{\sigma} \to \stackrel{\circ}{P}_{\sigma}\times F_{\sigma}$ as  structural data of $(2n,k)$-manifolds.
\end{defn}

Let $\overline{W_{\sigma}/T^{\sigma}}$ denote the closure of $W_{\sigma}/T^{\sigma}$. It is a compact subset in $M^{2n}/T^k$ since we assume $M^{2n}$ to be a compact manifold. We obtain:

\begin{cor}\label{closure}
$\widehat{\mu}(\overline{W_{\sigma}/T^{\sigma}}) = P_{\sigma}$.
\end{cor}
\begin{proof}
It holds  $\stackrel{\circ}{P_{\sigma}} \subset \widehat{\mu}(\overline{W_{\sigma}/T^{\sigma}}) \subseteq P_{\sigma}$ since $\widehat{\mu} : M^{2n}/T^k \to P^{k}$ is continuous and $\widehat{\mu}(W_{\sigma}/T^{\sigma}) = \stackrel{\circ}{P_{\sigma}}$. Furthermore,  since $\widehat{\mu}(\overline{W_{\sigma}/T^k})$ is a compact set, it follows  $\widehat{\mu}(\overline{W_{\sigma}/T^{\sigma}}) = P_{\sigma}$.
\end{proof}

The trivialization $h_{\sigma} : W_{\sigma}/T^{\sigma} \to \stackrel{\circ}{P}_{\sigma}\times F_{\sigma}$ induces the projection  $\xi _{\sigma} :  W_{\sigma}/T^{\sigma} \to F_{\sigma}$. For any   point $c_{\sigma}\in F_{\sigma}$,  define a  subspace $W_{\sigma}[\xi _{\sigma},c_{\sigma}]$ of  $W_{\sigma}$ by
\begin{equation}\label{f-leaf}
W_{\sigma}[\xi _{\sigma},c_{\sigma}] = (\pi _{\sigma} ^{-1}\circ \xi _{\sigma}^{-1})(c_{\sigma}),
\end{equation}
where $\pi _{\sigma} : W_{\sigma}\to W_{\sigma}/T^{\sigma}$ is a projection.

\begin{defn}
The space $W_{\sigma}[\xi _{\sigma},c_{\sigma}]$ is said to be the  {\it leaf} of a stratum $W_{\sigma}$ given by the trivialization $h_{\sigma}$.
\end{defn}

Note that, by its definition, a leaf $W_{\sigma}[\xi _{\sigma},c_{\sigma}]$ is invariant under the action of the torus $T^k$ and  $W_{\sigma} = \cup _{c_{\sigma}\in F_{\sigma}}W_{\sigma}[\xi _{\sigma},c_{\sigma}]$.

The definition of a  leaf also  implies:
\begin{lem}\label{l-leaf}
Let $\widehat{\mu}_{\xi _{\sigma},c_{\sigma}}$ denote the   restriction of the map $\widehat{\mu} : M^{2n}/T^k\to P^k$  to $W_{\sigma}[\xi _{\sigma},c_{\sigma}]/T^{\sigma}$.
Then the map  $\widehat{\mu}_{\xi _{\sigma},c_{\sigma}} : W_{\sigma}[\xi _{\sigma},c_{\sigma}]/T^{\sigma} \to \stackrel{\circ}{P_{\sigma}}$ is a homeomorphism for any $c_{\sigma}\in F_{\sigma}$.
\end{lem}

Moreover, we obtain:
\begin{lem}\label{leaf-bound}
$\widehat{\mu} (\overline{W_{\sigma}[\xi _{\sigma}, c_{\sigma}]/T^{\sigma}}) = P_{\sigma}$.
\end{lem}

\begin{proof}
It follows from Lemma~\ref{l-leaf}  that $\widehat{\mu} (\overline{W_{\sigma}[\xi _{\sigma}, c_{\sigma}]/T^{\sigma}}) \subseteq P_{\sigma}$, since $\widehat{\mu}$ is a continuous map.  On the other hand, $\overline{W_{\sigma}[\xi _{\sigma}, c_{\sigma}]/T^{\sigma}}$ is a closed subset in  the compact space $M^{2n}/T^k$ and,  thus,  it is a compact set as well. It implies that $\widehat{\mu} (\overline{W_{\sigma}[\xi _{\sigma}, c_{\sigma}]/T^{\sigma}})$ is a  compact set in $P_{\sigma}$ which   contains the interior of $P_{\sigma}$,  which further  implies  that $\widehat{\mu} (\overline{W_{\sigma}[\xi _{\sigma}, c_{\sigma}]/T^{\sigma}}) = P_{\sigma}$.
\end{proof}

Suppose   it is given an $2n$-dimensional manifold $M^{2n}$ with an effective action of the  algebraic torus $(\C ^{*})^{k}$. Assume that  the  induced   action of the compact torus $T^{k}\subset (\C ^{*})^{k}$ on $M^{2n}$ satisfies  Axioms 1-4  and the following folds:
\begin{enumerate}
\item [(1)] Any stratum $W_{\sigma}$ is $(\C ^{*})^{k}$- invariant;
\item [(2)] The free action of the torus $T^{\sigma}$ on $W_{\sigma}$ extends to a free action of the  algebraic torus $(\C ^{*})^{\sigma}$  on $W_{\sigma}$;
\item [(3)] The projections $\hat{\mu} : W_{\sigma}/T^{\sigma} \to \stackrel{\circ}{P}_{\sigma}$ and $\hat{\pi} : W_{\sigma}/T^{\sigma}\to W_{\sigma}/(\C ^{*})^{\sigma}$ define the homeomorphism $h_{\sigma} = (\hat{\mu}, \hat{\pi}) :  W_{\sigma}/T^{\sigma}\to \stackrel{\circ}{P}_{\sigma}\times W_{\sigma}/(\C ^{*})^{\sigma}$.
\end{enumerate}
Then, the space of parameters $F_{\sigma}$  of a stratum $W_{\sigma}$ can be identified with  $F_{\sigma} \cong W_{\sigma}/(\C ^{*})^{\sigma}$.

\begin{defn}
We say than an action of the compact torus $T^k$ on an $(2n,k)$-manifold $M^{2n}$ extends to the compatible action of the  algebraic torus $(\C ^{*})^{k}$  if it is defined an  action of the  algebraic torus $(\C ^{*})^{k}$  on $M^{2n}$, which   satisfies the conditions 1-3  given above.
\end{defn}

We use the description of the strata  as  spaces consisting of leafs to formulate the properties which allow to describe the gluing of the strata. Recall that we have  fixed  the projection  $\xi _{\sigma} :  W_{\sigma}/T^{\sigma} \to F_{\sigma}$ as a part of our structural data.
\begin{ax}\label{leaf}
For any  leaf $W_{\sigma}[\xi _{\sigma}, c_{\sigma}]$ of $W_{\sigma}$  it holds:
\begin{itemize}
\item [a)]  it is a smooth submanifold in $M^{2n}$ and  the induced map $\mu _{\xi _{\sigma}, c_{\sigma}} : W_{\sigma}[\xi _{\sigma}, c_{\sigma}] \to \stackrel{\circ}{P}_{\sigma}$ is a smooth fiber bundle,
\item
 [b)] its boundary $\bar{\partial} W_{\sigma}[\xi _{\sigma}, c_{\sigma}]$ is  the union of  leafs
$W_{\bar{\sigma}}[\xi _{\bar{\sigma}}, c_{\bar{\sigma}}]$ for exactly one $c_{\bar{\sigma}}\in F_{\bar{\sigma}} $, where  $P_{\bar{\sigma}}$ runs through  some admissible faces for  $P_{\sigma}$ and  $\sigma\in \mathfrak{S}$,
\item [c)] the map  $\eta _{\sigma, \bar{\sigma}} : F_{\sigma} \to F_{\bar{\sigma}}$, $\eta _{\sigma, \bar{\sigma}} (c_{\sigma}) =c_{\bar{\sigma}}$  given by b)   is a continuous map.
\end{itemize}
\end{ax}

\begin{rem}
Axiom~\ref{leaf}, as it will be seen  in Section~\ref{gr}, is motivated by the results of Atiyah, Guillemin-Sternberg and Gel'fand-MacPherson about $(\C ^{*})^{n}$-action on the complex Grassmann manifolds $G_{n,k}$.
\end{rem}

We deduce the following important consequence of  the statement {\it b)} of Axiom~\ref{leaf}.

\begin{prop}\label{face}
A face of any admissible polytope is an admissible polytope.
\end{prop}
\begin{proof}
Let us fix some  face $P_{\tilde{\sigma}}$ of an admissible polytopes $P_{\sigma}$ and let us consider a point $p\in \stackrel{\circ}{P_{\tilde{\sigma}}}$. By  Lemma~\ref{leaf-bound},  we see  that $\widehat{\mu}^{-1}(p)\cap \bar{\partial} W_{\sigma}[\xi _{\sigma}, c_{\sigma}]/T^{k}\neq \emptyset$.  Therefore,  there exists  a point $x$ from $\bar{\partial}$-boundary of the leaf  $W_{\sigma}[\xi _{\sigma}, c_{\sigma}]$ such that $\pi (x) \in \widehat{\mu}^{-1}(p)$. It implies that $\mu (x) = p$. By Axiom~\ref{leaf}, we have that  $x$ belongs to some leaf   $W_{\bar{\sigma}}[\xi _{\bar{\sigma}}, c_{\bar{\sigma}}]$, where $P_{\bar{\sigma}}$ is a face of $P_{\sigma}$.  This implies that $P_{\bar{\sigma}}$ is an admissible polytope and that  $\mu (x)=p \in  \stackrel{\circ}{P_{\bar{\sigma}}}$. Therefore,  $p\in  \stackrel{\circ}{P_{\tilde{\sigma}}}\cap \stackrel{\circ}{P_{\bar{\sigma}}}$, and,  since $P_{\bar{\sigma}}$ and $P_{\tilde{\sigma}}$ are  faces of the same polytope $P_{\sigma}$,  it  follows that $P_{\tilde{\sigma}} = P_{\bar{\sigma}}$. Therefore, $P_{\tilde{\sigma}}$ is an admissible polytope.
\end{proof}

\begin{rem}
In the case of the canonical  action of the  algebraic torus $(\C ^{*})^{k+1}$on $G_{k+1,q}$, this result is obtained in~\cite{AT, GS}.
\end{rem}
\begin{rem}
Note that combining Axiom~\ref{leaf}  and  the proof of  Proposition~\ref{face} we obtain  that  if   $P_{\bar{\sigma}}$  is a face of $P_{\sigma}$ then  there exists a  leaf $W_{\bar{\sigma}}[\xi _{\bar{\sigma}}, c_{\tilde{\sigma}}]$ which is contained in the $\bar{\partial}$-boundary of the leaf $W_{\sigma}[\xi _{\sigma}, c_{\sigma}]$.
\end{rem}

It follows that the condition {\it b)} of Axiom~\ref{leaf} can be strengthen:

\begin{cor}\label{all}
For any $c_{\sigma}\in F_{\sigma}$, the boundary $\bar{\partial} W_{\sigma}[\xi _{\sigma}, c_{\sigma}]$ of a leaf $W_{\sigma}[\xi _{\sigma},c_{\sigma}]$ of a  stratum  $W_{\sigma}$  is  the union of  leafs
$W_{\bar{\sigma}}[\xi _{\bar{\sigma}}, c_{\bar{\sigma}}]$ for exactly one $c_{\bar{\sigma}}\in F_{\bar{\sigma}} $, where  $P_{\bar{\sigma}}$ runs through  {\bf all} faces for  $P_{\sigma}$.
\end{cor}

The statement  {\it a)} of Axiom~\ref{leaf} combining  with Corollary~\ref{all} directly implies that Lemma~\ref{leaf-bound} can be strengthen:
\begin{cor}\label{cl-leaf}
The map $\widehat{\mu} : \overline{W_{\sigma}[\xi _{\sigma}, c_{\sigma}]/T^{\sigma}} \to P_{\sigma}$ is a homeomorphism.
\end{cor}

Let $P_{\bar{\sigma}}$ be a face of $P_{\sigma}$. Then  Corollary~\ref{all}  implies an  existence of the map
\begin{equation}\label{eta}
\eta _{\sigma, \bar{\sigma}} : F_{\sigma} \to F_{\bar{\sigma}}
\end{equation}
defined by: $\eta _{\sigma, \bar{\sigma}}(c_{\sigma})$ is a unique point $c_{\bar{\sigma}}$ from $F_{\bar{\sigma}}$ such that $W_{\bar{\sigma}}[\xi _{\bar{\sigma}}, c_{\bar{\sigma}}] \subset \bar{\partial} W_{\sigma}[\xi _{\sigma}, c_{\sigma}]$. The statement {\it c)} of  Axiom~\ref{leaf}  states that the map $\eta _{\sigma , \bar{\sigma}}$ is a continuous map.

Let now $P_{\bar{\bar{\sigma}}}$ be a face of $P_{\bar{\sigma}}$ and $P_{\bar{\sigma}}$ be a face of $P_{\sigma}$. Then Axiom~\ref{leaf} states an existence  of  the  projections $\xi _{\bar{\bar{\sigma}}} : W_{\bar{\bar{\sigma}}}/T^k\to F_{\bar{\bar{\sigma}}}$ and  $\xi _{\sigma} : W_{\sigma}/T^k\to F_{\sigma}$ such that
\[
W_{\bar{\bar{\sigma}}}[\xi _{\bar{\bar{\sigma}}}, c_{\bar{\bar{\sigma}}}] \subset \bar{\partial} W_{\bar{\sigma}}[\xi _{\bar{\sigma}}, c_{\bar{\sigma}}]\subset   \overline{W_{\bar{\sigma}}[\xi _{\bar{\sigma}}, c_{\bar{\sigma}}]} , \;\; W_{\bar{\sigma}}[\xi _{\bar{\sigma}}, c_{\bar{\sigma}}]\subset \bar{\partial} W_{\sigma}[\xi _{\sigma}, c_{\sigma}]\subset   \overline{W_{\sigma}[\xi _{\sigma}, c_{\sigma}]}
\]

Since the leafs are disjoint,  it follows that
\[
W_{\bar{\bar{\sigma}}}[\xi _{\bar{\bar{\sigma}}}, c_{\bar{\bar{\sigma}}}]\subset  \overline{W_{\sigma}[\xi _{\sigma}, c_{\sigma}]}\setminus  W_{\sigma}[\xi _{\sigma}, c_{\sigma}] = \bar{\partial}W_{\sigma}[\xi _{\sigma}, c_{\sigma}] .\]

Altogether this implies:
\begin{cor}
For any pair $P_{\bar{\sigma}}\subset P_{\sigma }$ there exists the map  $\eta _{\sigma ,\sigma ^{'}} : F_{\sigma}\to F_{\bar{\sigma}}$ such that if $P_{\bar{\bar{\sigma}}}\subset P_{\bar{\sigma}}\subset P_{\sigma}$ then
$\eta _{\bar{\sigma} ,\bar{\bar{\sigma}}}\circ \eta _{\sigma ,\bar{\sigma}} = \eta _{\sigma ,\bar{\bar{\sigma}}}$.
\end{cor}


\section{On singular and regular values \\ of the almost moment map}
We characterize the  singular and regular values of the almost moment map $\mu : M^{2n} \to P^k$.
\begin{defn}
The  cort\'ege $\sigma (x)$ of a  point $x\in P^{k}$ is a set of admissible polytopes  defined by:
\[
\sigma(x) = \{ P_{\sigma} \in P_{\mathfrak{S}}: x\in \stackrel{\circ}{P}_{\sigma}\}.
\]
\end{defn}

Obviously,  $\sigma(x)\neq \emptyset$ for any point  $x\in P^k$,  since $\cup _{\sigma\in \mathfrak{S}}\stackrel{\circ}{P}_{\sigma}=P^k$.
Moreover,  $P^k\in \sigma(x)$ for any point $x\in \stackrel{\circ}{P^{k}}$ .
\begin{defn}
The  point $x\in P^k$ is said to be  a regular point   if $\dim P_{\sigma} =k$ for all  $P_{\sigma}\in \sigma (x)$.
\end{defn}

Note that if  $x\in P^k$ is a regular point  then $x\in \stackrel{\circ}{P^{k}}$. By $P^{k}_{r}$ we denote the set of regular points in $P^{k}$.

\begin{rem}\label{regularpoints}
Note that the set of regular points in $\stackrel{\circ}{P^{k}}$  is a  non-empty set and moreover it is a dense set in $\stackrel{\circ}{P^{k}}$. This  follows from the fact that there are finitely many admissible polytopes, so the union of  admissible polytopes of the dimension less then $k$ has the dimension less then $k$.
\end{rem}

For a given $(2n,k)$-manifold $M^{2n}$,  let $Z$ denote the union of all open  admissible polytopes $\stackrel{\circ}{P}_{\sigma}$ whose dimension is $<k$. The following immediately holds:
 \begin{lem}
The set   $\stackrel{\circ}{P^{k}}\setminus (Z\cap \stackrel{\circ}{P^{k}})$ coincides with $P^{k}_{r}$. In particular $P^{k}_{r}$ is an open set in $P^{k}$ which  has finitely many connected components.
\end{lem}

\begin{ex}
 For the Grassmann manifold $G_{4,2}$,  the set  $P^{3}_{r}\subset \stackrel{\circ}{\Delta}_{4,2}$  is  the   complement to the union of   three open  diagonal squares $\stackrel{\circ}{P}_{12, 34}$, $\stackrel{\circ}{P}_{13, 24}$, $\stackrel{\circ}{P}_{14,23}$ in  $\stackrel{\circ}{\Delta}_{4,2}$ (see~\cite{BT-1}).   

For the Grassmann manifold $G_{5,2}$,  the set  $P^{4}_{r}\subset \stackrel{\circ}{\Delta}_{5,2}$ is the   complement to the union of of  ten  open prisms $\stackrel{\circ}{P}_{i}$, $1\leq i\leq 10$ in $\stackrel{\circ}{\Delta}_{5,2}$ (see~\cite{BTN},Proposition 9).
\end{ex}

 Recall that  $x$  is a regular value of the almost moment map $\mu :  M^{2n}\to P^k$  if and only if any $y\in \mu ^{-1}(x)$ is a regular point,  that is the  differential of $\mu$ at $y$ has  rank equal to $k$.

\begin{thm}\label{rpv}
If  $x\in P^k_{r}$     then $x$  is  a  regular   value for the almost  moment map $\mu : M^{2n}\to P^k$.
\end{thm}
\begin{proof}
Let  $x\in P^k$ be  a regular point. It holds  that  $ \mu ^{-1}(x)\subset \cup _{\sigma} W_{\sigma}$, where  the union goes over all admissible sets $\sigma$ such that $P_{\sigma}\in \sigma (x)$.   Now, if $y\in \mu ^{-1}(x)\cap W_{\sigma}$,  then $y$ belongs to a  unique leaf $W_{\sigma}[\xi _{\sigma}, c_{\sigma}]$. The differential of the map $\mu _{\xi _{\sigma}, c_{\sigma}} : W_{\sigma}[\xi _{\sigma} , c_{\sigma}] \to \stackrel{\circ}{P}_{\sigma}$ is  an epimorphism, according to Axiom~\ref{leaf}. Since $P_{\sigma}\in \sigma (x)$,  it follows that the rank of the differential of $\mu _{\xi _{\sigma}, c_{\sigma}}$  at  $y$ is equal to $\dim P_{\sigma} =k$, which proves the statement.
\end{proof}

Note that the above proof implies that any point of the main stratum $W$ is a regular point of the almost moment map.
\begin{cor}\label{submanif}
For any point  $x\in \stackrel{\circ}{P^{k}}$ the preimage $M_{x}^{2n-k}=\mu ^{-1}(x)$ is a closed submanifold in $M^{2n}$ of  the  dimension $2n-k$. The torus $T^k$ acts freely on the manifold  $M_{x}^{2n-k}$  and the orbit space $M_{x}^{2n-k}/T^k$  is a smooth  manifold of the dimension $2n-2k$, which  can be identified with some  compactification of the space of parameters $F$ of the main stratum.
\end{cor}

From Remark~\ref{regularpoints}, it follows  the well-known fact:
\begin{cor}
The set of regular values  of the almost  moment map $\mu : M^{2n}\to P^{k}$ is a dense set in $P^k$.
\end{cor}

\begin{rem}
From our axioms, in general case,  does not follow that the set of regular values of the almost moment map $\mu$ coincides with the set $P^{k}_{r}$. Nevertheless, in the case of Grassmann manifolds $G_{k+1, q}$,  it is proved in~\cite{BTN} that these two sets coincide.
\end{rem}

We can push up this further.
\begin{lem}
The set  $M^{2n}_{r}= \mu ^{-1}(P^{k}_{r})$ is a dense submanifold in $M^{2n}$  with a  free action of the torus $T^k$. The orbit space  $M^{2n}_{r}/T^{k}$ is a  smooth manifold which  is dense in   $M^{2n}/T^{k}$. Moreover,  the maps $\mu : M^{2n}_{r} \to P^{k}_{r}$ and $\hat{\mu} : M^{2n}_{r}\to P^{k}_{r}$ are  projections of the smooth fiber bundles.
\end{lem}
\begin{proof}
The  first statement  follows   from the observations that $P^{k}_{r}$ is an open set in $P^{k}$ and  that, by Axiom~\ref{fiber},  $T^{k}$ acts freely on $M^{2n}_{r}$.

The second statement follows from the observation that  the  almost moment   map $\mu : M^{2n}_{r} \to P^{k}_{r}$ is,  by Theorem~\ref{rpv}, a  submersion and moreover, being open,  $\mu$ is a  proper map. Since $\mu$ is $T^k$-invariant for the free $T^k$-action on $M_{r}^{2n}$,  it follows that  the same holds for the induced map $\hat{\mu} : M^{2n}_{r}/T^k \to P^{k}_{r}$.
Then,   classical  Ehressmann's fibration theorem implies  that the maps $\mu$ and $\hat{\mu}$ produce  the stated fiber bundles.
\end{proof}
    Let $P^{k}_{r,1}, \ldots , P^{k}_{r,s}$ be the connected components of  the open set $P^{k}_{r}$ and $M^{2n}_{r, i} = \mu ^{-1}(P^{k}_{r,i})$, $1\leq i\leq s$. Then, the maps  $\mu : M^{2n}_{r,i} \to P^{k}_{r,i}$ define  smooth fiber bundles with connected base. In this way, we obtain:
\begin{cor}\label{diff}
The manifolds  $M_{x}^{2n-k}$ and  $M_{y}^{2n-k}$ are diffeomorphic for any $x, y\in P^{k}_{r,i}$, $1\leq i\leq s$.
\end{cor}


\subsection{The main stratum and its  orbit space}\label{main-sp-orb}
 As noted in Example~\ref{main} the main stratum $W$ is a  dense set  in $M^{2n}$,  which  implies that $W/T^k$ is  a  dense set  in the orbit space $M^{2n}/T^k$. Moreover,  $\widehat{\mu}(W/T^k) = \stackrel{\circ}{P^{k}}$ and $W/T^k \cong \stackrel{\circ}{P^{k}}\times F$, where  $F$ denotes the space of parameters of the main stratum  $W$, as introduced  by Definition~\ref{param}. It follows that there exists  such  a  compactification of the product $\stackrel{\circ}{P^{k}}\times F$ that   is homeomorphic to $M^{2n}/T^k$.

The first result in this direction is the following.

\begin{prop}\label{point}
If the space $F$ is a point then the orbit space $M^{2n}/T^k$ is homeomorphic   to the polytope  $P^k$.
\end{prop}
\begin{proof}
If  $F$ is a point, then  $W$ consists of  one leaf. Thus,  Corollary~\ref{cl-leaf} implies that the closure $\overline{W/T^k}$ is homeomorphic   to the polytope $P^k$, which further  implies that $M^{2n}/T^{k}$ is homeomorphic   to the polytope $P^k$.
\end{proof}

\begin{ex}
If $M^{2n}$ is a quasitoric manifold,  then  $F$ is a point.
\end{ex}

Assume now that the space of parameters $F$ of the main stratum  is not a point. Let $x\in  \stackrel{\circ}{P^{k}}$ and  put  $\mathfrak{S}(x) = \{\sigma \in \mathfrak{S} | x\in \stackrel{\circ}{P_{\sigma}}\}$, that is $\sigma \in \mathfrak{S}(x)$ if and only if $P_{\sigma}\in \sigma (x)$.  Then $\widehat{\mu}^{-1}(x)$ is a closed  set in $M^{2n}/T^k$  and
\[
\widehat{\mu}^{-1}(x) =  \cup _{\sigma \in \mathfrak{S}(x)} \{y\in W_{\sigma}/T^k : \widehat{\mu}(y) = x\}.
\]
Since  the subspace  $\{y\in W_{\sigma}/T^k : \widehat{\mu}(y) = x\}$ is homeomorphic to $F_{\sigma}$,  we  introduce a topology on the union $ \cup _{\sigma \in \mathfrak{S}(x)} F_{\sigma}$ such that the space obtained as this union becomes homeomorphic to $\widehat{\mu}^{-1}(x)$.
Let $\overline{F_{x}}$ denote the closure of the space  $F_{x} = \widehat{\mu}^{-1}(x)\cap W/T^k \cong F$. Since $\overline{W/T^k} = M^{2n}/T^k$, it follows that 
\begin{equation}\label{param-cl}
 \overline{F_{x}} = \widehat{\mu}^{-1}(x)\cong \cup _{\sigma \in \mathfrak{S}(x)} F_{\sigma},
 \end{equation}
for any $x\in \stackrel{\circ}{P^{k}}$.  
From Corollary~\ref{submanif} and Corollary~\ref{diff} it follows:
\begin{cor}\label{finite}
If  $x\in P^{k}_{r}$ then $\overline{F_{x}} = M_{x}^{2n-k}/T^k$  is a smooth manifold in $M^{2n}/T^k$.
 Moreover,  the manifolds $\overline{F_{x}}$ and $\overline{F_{y}}$ are homeomorphic for any $x, y\in P^{k}_{r,i}$, $1\leq i\leq s$.
\end{cor}

\begin{rem}
We want to note that although a space  $F_{x}$ is homeomorphic to the space of parameters $F$ for any $x\in \stackrel{\circ}{P^{k}}$, there is no argument to claim that, in general,  their  compactifications  $\overline{F_{x}}$ are homeomorphic.  But,  Corollary~\ref{finite} implies that the set of manifolds   $\overline{F_{x}}$, $x\in P^{k}_{r}$, is finite, up to diffeomorphism.
\end{rem}

We first point the following:
\begin{lem}
The space of parameters $F$ of the main stratum  is a compact space if and only if all points from  $\stackrel{\circ}{P^{k}}$ are simple.
\end{lem}
\begin{proof}
If all  points from  $\stackrel{\circ}{P^{k}}$ are simple, then   the condition $P_{\sigma}\cap \stackrel{\circ}{P^{k}}\neq \emptyset$ implies that $P_{\sigma}=P^{k}$. This  means that $\widehat{\mu}^{-1}(x) \cong F$ for any point $x\in \stackrel{\circ}{P^{k}}$.  Therefore, $F$ is a compact space. In opposite direction, if   $F$ is a compact space, from  the fact that $F\cong F_{x} \subseteq \widehat{\mu}^{-1}(x)$ for any point $x\in \stackrel{\circ}{P^k}$, it follows that     $F\cong  \widehat{\mu}^{-1}(x)$,  which implies  that all points from $\stackrel{\circ}{P^{k}}$ are simple.
\end{proof}

In the case when the  spaces $\overline{F_{x}}$, $x\in \stackrel{\circ}{P^{k}}$ are all  homeomorphic and all points from the boundary of the polytope  $P^k$ are simple,  we provide an explicit topological description of the orbit space $M^{2n}/T^k$.

Let $P_{\sigma_{0}}$ be a face of the polytope $P^{k}$ and by $\mathfrak{S}(\sigma_{0})$ denote the set of those admissible sets $\sigma \in \mathfrak{S}$  for which  $P_{\sigma _{0}}$ is a face of the polytope $P_{\sigma}$ and $P_{\sigma}$ is not a face of the polytope  $P^{k}$.
\begin{prop}\label{paramet-bound}
Assume that all points from $\partial P^k$ are simple,  that a space $\overline{F_{x}}$  is  homeomorphic to the space $\bar{F}$ for all  $x \in \stackrel{\circ}{P^{k}}$.  Then for any $\sigma _{0}\in \mathfrak{S}$, such that  $P_{\sigma_{0}}$ is  a face of the polytope $P^{k},$  there exists a point $x\in \stackrel{\circ}{P^k}$ such that
$
\widehat{\mu} ^{-1}(x) \cong \cup_{\sigma \in \mathfrak{S}(\sigma_{0})} F_{\sigma}$ and, thus
\[
\cup_{\sigma \in \mathfrak{S}(\sigma_{0})} F_{\sigma} \cong \overline{F}.
\]
\end{prop}

\begin{proof}
We first note  that the assumption that all  points from $\partial P^k$ are simple implies that  if $P_{\sigma _{0}}$ is not a face of some admissible polytope $P_{\sigma}$ then $\partial P_{\sigma}\cap P_{\sigma_{0}} =\emptyset$. Otherwise we would have a point $z\in \partial P_{\sigma}\cap P_{\sigma _{0}}$ meaning that $z$ belongs   to some face $P_{\bar{\sigma}}$ of $P_{\sigma}$,  which is an  admissible polytope.  Since $z$ is a simple point,  it follows  that  $P_{\bar{\sigma}} = P_{\sigma _{0}}$. Therefore, since there are only finitely many admissible polytopes, we see that   there exists a neighborhood $U$ of $P_{\sigma _{0}}$ in $P^k$ such that $U\cap P_{\sigma}=\emptyset$ for any admissible polytope $P_{\sigma}$ for which  $P_{\sigma _{0}}$ is not a face of  $P_{\sigma}$. Further $U\cap (\cap _{\sigma \in \mathfrak{S}(\sigma _{0})}P_{\sigma}) \neq \emptyset$ as $P_{\sigma _{0}}$ belongs to all polytopes from $\mathfrak{S}(\sigma _{0})$. For a  point $x\in  U\cap (\cap _{\sigma \in \mathfrak{S}(\sigma _{0})}P_{\sigma})$,  we obtain  $\overline{F}\cong \widehat{\mu}^{-1}(x) \cong \cup_{\sigma \in \mathfrak{S}(\sigma_{0})} F_{\sigma}$,  since $x\in \stackrel{\circ}{P^k}$.
\end{proof}

 By $\stackrel{\circ}{P_{\mathfrak{S}}}$ denote  the set of interiors of all admissible polytopes that  are not  proper faces of the polytope $P^{k}$.
\begin{thm}\label{main-main}
Assume that all points from $\partial P^k$ are simple and that   a space $\widehat{\mu}^{-1}(x) \subset M^{2n}/T^k$ is  homeomorphic to the space $\bar{F}$ for all $x\in \stackrel{\circ}{P^{k}}$. Assume also that  the set $\stackrel{\circ}{P_{\mathfrak{S}}}$ can be divided into  subsets $\stackrel{\circ}{P_{\mathfrak{S}_{1}}}, \ldots , \stackrel{\circ}{P_{\mathfrak{S}_{l}}}$ such that the areas 
 from  $\stackrel{\circ}{P_{\mathfrak{S}_{i}}}$ give the polytopal decomposition  for  the area $\stackrel{\circ}{P^k}$ and $F_{\sigma }\cong  F_{\mathfrak{S}_{i}}$ for any $\sigma \in \mathfrak{S}_{i}$, where $1\leq i\leq l$. Then the orbit space $M^{2n}/T^k$ is homeomorphic to the quotient space 
$
P^{k}\times \bar{F}/\approx ,
$ where
\[
(x,f_1)\approx (y,f_2) \Leftrightarrow x=y\in P_{\sigma _{0}}\subset \partial P^k\; \text{and}\; \eta _{\sigma _{1} ,\sigma _{0}}(f_1) = \eta _{\sigma _{2} ,\sigma _{0}}(f_2).
\]
Here $\sigma _1, \sigma _{2} \in \mathfrak{S}(\sigma _{0})$ are such that $f_1\in F_{\sigma _{1}}$, $f_{2}\in F_{\sigma _{2}}$ and the map $\eta _{\sigma _{i}, \sigma _{0}}$ is defined by the formula~\eqref{eta}.
\end{thm}
\begin{proof}
Put  $W_{\mathfrak{S}_{i}}/T^k = \widehat{\mu}^{-1}(\stackrel{\circ}{P^k}) \cap (\cup _{\sigma \in \mathfrak{S}_{i}}W_{\sigma}/T^k)$. Since $\widehat{\mu}_{\sigma} : W_{\sigma}/T^k \to \stackrel{\circ}{P}_{\sigma}$ is a fiber bundle and,  by the assumption,  all $F_{\sigma}$ for $\sigma \in \mathfrak{S}_{i}$ are homeomorphic to the space $F_{\mathfrak{S}_{i}}$, it follows that $\widehat{\mu} : W_{\mathfrak{S}_{i}}/T^k \to \stackrel{\circ}{P^k}$ is a fiber bundle with a fiber $F_{\mathfrak{S}_{i}}$. Since $\stackrel{\circ}{P^k}$ is contractible, it follows that  $W_{\mathfrak{S}_{i}}/T^k\cong  \stackrel{\circ}{P^k}\times F_{\mathfrak{S}_{i}}$.  Since  $\overline{F}\cong  \widehat{\mu}^{-1}(x)$, for   $x\in \stackrel{\circ}{P^{k}}$, it follows that $\widehat{\mu}^{-1}(\stackrel{\circ}{P^k})  = \cup_{i=1}^{l} W_{\mathfrak{S}_{i}} \cong \stackrel{\circ}{P^k}\times (\cup_{i=1}^{l} F_{\mathfrak{S}_{i}}) \cong \stackrel{\circ}{P^k}\times \bar{F}$, where the last homeomorphism holds according to~\eqref{param-cl}.

Let $P_{\sigma _{0}}$ be a face of the polytope $P^{k}$. It is an admissible polytope.  Since any point from $\partial P^{k}$ is simple, it follows that
$W_{\sigma_{0}}/T^k = \widehat{\mu}^{-1}(\stackrel{\circ}{P}_{\sigma _{0}})$ and, thus, the projection $\widehat{\mu}^{-1} : \widehat{\mu}^{-1}(\stackrel{\circ}{P}_{\sigma _{0}})\to \stackrel{\circ}{P}_{\sigma _{0}}$ is a fiber bundle with the  fiber $F_{\sigma _{0}}$. It implies that the space $\widehat{\mu}^{-1}(\stackrel{\circ}{P}_{\sigma _{0}})$ is homeomorphic to the direct product $\stackrel{\circ}{P}_{\sigma _{0}}\times F_{\sigma _{0}}$.  We also have that $F_{\sigma _{0}}\subseteq \overline{F}$ since $\widehat{\mu}^{-1}(\stackrel{\circ}{P^{k}})$ is an  everywhere dense set in $M^{2n}/T^k$ and $\overline{F}$ is a closed set. Combining this and   Proposition~\ref{paramet-bound}, we obtain
$F_{\sigma _{0}} \subseteq \cup_{\sigma \in \mathfrak{S}(\sigma_{0})} F_{\sigma}$. Therefore, the family of maps $\eta _{\sigma, \sigma_{0}} : F_{\sigma}\to F_{\sigma _{0}}$ produces  the map $\eta : \overline{F}\to F_{\sigma _{0}}$ that  is surjective. It implies that $\stackrel{\circ}{P}_{\sigma _{0}}\times F_{\sigma _{0}}$ is homeomorphic to a quotient of the space $\stackrel{\circ}{P}_{\sigma _{0}}\times \overline{F}$ by the equivalence  relation $(x, c_1)\approx (x, c_2)$ if and only if
$\eta (c_1) = \eta (c_2)$, that is  $\eta _{\sigma _{1}, \sigma _{0}}(c_1) = \eta _{\sigma _{2}, \sigma _{0}}(c_2)$, where $c_1\in F_{\sigma _{1}}, c_2\in F_{\sigma_{2}}$ and $\sigma _2, \sigma _2\in \mathfrak{S}$.
As the faces of $P^k$ are the only admissible polytopes from $\partial P^k$,  this proves the statement.
\end{proof}

\begin{thm}\label{join}
Assume that 
\begin{itemize}
\item [1)] all points from  the boundary  $\partial P^k$ are simple;
\item [2)]  a space $\widehat{\mu}^{-1}(x) \subset M^{2n}/T^k$ is  homeomorphic to the space $\bar{F}$  for any $x\in \stackrel{\circ}{P^{k}}$;
\item [3)]  the set $\stackrel{\circ}{P_{\mathfrak{S}}}$ can be divided into  subsets $\stackrel{\circ}{P_{\mathfrak{S}_{1}}}, \ldots , \stackrel{\circ}{P_{\mathfrak{S}_{l}}}$ such that the  areas
 from  $\stackrel{\circ}{P_{\mathfrak{S}_{i}}}$, $1\leq i\leq l$ give the polytopal decomposition for the area  $\stackrel{\circ}{P^k}$;
\item [4)]  $F_{\sigma }\cong  F_{\mathfrak{S}_{i}}$ for any $\sigma \in \mathfrak{S}_{i}$, where $1\leq i\leq l$.
\end{itemize}
If  $F_{\sigma _{0}}$ is a point for any $\sigma _{0}\in \mathfrak{S}$ such that $P_{\sigma _{0}}$ is a face of  the polytope $P^{k}$ then:
\[
M^{2n}/T^k \cong S^{k-1}\ast \bar{F} .
\]
\end{thm}
\begin{proof}
Recall that the joint $X\ast Y$ for  topological spaces $X$ and $Y$ is defined to be a quotient space of the product  $CX\times Y$ by the equivalence  relation $(x_1,0,y_1)\approx (x_2,0,y_2) \Leftrightarrow x_1=x_2$, where $CX$ is  a cone over $X$. In our case Theorem~\ref{main-main} gives that
\[
M^{2n}/T^k \cong P^{k} \times \overline{F}/\approx, \;\; (x,f_1)\approx (y,f_2) \Leftrightarrow x=y\in \partial P^{k}.
\]
We identify the polytope  $P^{k}$ with the closed disc $\overline{D^{k}}$ and further identify  $\overline{D^{k}}$  with the cone $CS^{k-1}$ over its boundary $S^{k-1}$. It implies that
\[
M^{2n}/T^k \cong CS^{k-1}\times \overline{F}/\approx,\;\; (x,0, f_1)\approx (x,0,f_2),
\]
which proves the statement.
\end{proof}




\section{Complex of  admissible polytopes}
Let $P_{\mathfrak{S}}$ denote  the  family of all admissible polytopes. Combining Example~\ref{adm_P}, Example~\ref{adm_v} and Proposition~\ref{face} we obtain:
\begin{itemize}
\item $P^{k}\in P_{\mathfrak{S}}$ and $v\in P_{\mathfrak{S}}$ for any vertex $v$.
\item  If $P_{\sigma} \in P_{\mathfrak{S}}$ and $P_{\bar{\sigma}}$ is a face of the polytope  $P_{\sigma}$, then  $P_{\bar{\sigma}} \in P_{\mathfrak{S}}$.
\end{itemize}

Let us consider the complex  $C(M^{2n}, P^k)$,  which is obtained a  formal disjoint union of  the interiors of all admissible polytopes:

\begin{equation}\label{complex}
 C(M^{2n}, P^k) = \cup _{P_{\sigma}\in P_{\mathfrak{S}}}\stackrel{\circ}{P_{\sigma}}.
\end{equation}


By the defintion~\eqref{complex}, there is a bijection between the set of cells  in  $C(M^{2n}, P^k)$  and the  set of admissible polytopes $P_{\mathfrak{S}}$.

\begin{lem}
There is the canonical map $\widehat{\pi} : C(M^{2n}, P^k)\to P^k$.
\end{lem}
\begin{proof}
For  any  polytope $P_{\sigma}^{'}$ in $C(M^{2n}, P^k)$  there is a unique admissible polytope  $P_{\sigma}\in P_{\mathfrak{S}}$ such that $\stackrel{\circ}{P}_{\sigma}$  corresponds to the polytope $P_{\sigma}^{'}$ and  we define $\widehat{\pi}(P_{\sigma}^{'})=\stackrel{\circ}{P}_{\sigma}$.
\end{proof}
\begin{cor}\label{homeom2}
For any $P_{\sigma}\in P_{\mathfrak{S}}$ there exists  a unique  polytope $P_{\sigma}^{'}$ in $C(M^{2n}, P^k)$  such that  the map $\widehat{\pi} : P_{\sigma}^{'}\to \stackrel{\circ}{P}_{\sigma}$ is a homeomorphism.
\end{cor}

\begin{cor}\label{quot}
The canonical map $\widehat{\pi} : C(M^{2n}, P^k)\to P^k$ is a quotient map.
\end{cor}
\begin{proof}
Define an equivalence relation $\approx$ on $C(M^{2n}, P^k)$ by $x\approx y$ if and only if $\widehat{\pi} (x) =\widehat{\pi} (y)$. It is obvious that $P^{k} = C(M^{2n}, P^k)/\approx$ and that  $\widehat{\pi} : C(M^{2n}, P^k)\to P^k$ is a quotient map.
\end{proof}

\begin{cor}\label{can-bij}
The canonical map $\widehat{\pi} : C(M^{2n}, P^k)\to P^k$ is a bijection if and only if the only admissible polytopes are the whole $P^{k}$ and its faces.
\end{cor}

\begin{ex}
For a quasitoric manifold $M^{2n}$  the canonical map  $\widehat{\pi} : C(M^{2n}, P^n)\to P^n$
is a bijection.
\end{ex}

For the  almost moment map $\mu : M^{2n}\to P^{k}$ we prove the following result:

\begin{lem}\label{canonic}
There is the canonical map $f : M^{2n}\to C(M^{2n}, P^k)$ such that $\mu = \widehat{\pi} \circ f$.
\end{lem}
\begin{proof}
For any  point $x\in M^{2n}$ there exists a unique  stratum  $W_{\sigma}$ such that  $x\in W_{\sigma}$. Then  $\mu(x)\in \stackrel{\circ}{P_{\sigma}}$ and by Corollary~\ref{homeom2}  there exists  a unique polytope $P_{\sigma}^{'}\subseteq C(M^{2n}, P^k)$ such that $\widehat{\pi} :  P_{\sigma}^{'}\to \stackrel{\circ}{P_{\sigma}}$ is a homeomorphism. It  implies that there exists a unique $y\in P_{\sigma}^{'}$ such that $\widehat{\pi} (y) = \mu (x)$. In this way, the map $f : M^{2n}\to C(M^{2n}, P^k)$ is defined by $f(x)=y$.
\end{proof}

For the  induced map $\widehat{\mu}: M^{2n}/T^k\to P^{k}$ we obtain:
\begin{cor}
There is the canonical map $\widehat{f} : M^{2n}/T^k\to C(M^{2n}, P^k)$ such that $\widehat{\mu} = \widehat{\pi} \circ \widehat{f}$.
\end{cor}

We  further assume that $P_{\mathfrak{S}}$ is  partially ordered by the inclusion of admissible polytopes. It is  defined the boundary operator $d$ which to each admissible polytope $P_{\sigma}$ assigns the disjoint union of of all its  faces. The operator $d$ induces the operator $d_{C}$ on the cells of the complex $C(M^{2n}, P^{k})$ by 
\[
d_{C}(P_{\sigma}^{'} )  = (dP_{\sigma})_{C},
\]
where $(dP_{\sigma})_{C}\subset C(M^{2n}, P^{k})$  corresponds  to $dP_{\sigma}$ by the  bijection between $P_{\mathfrak{S}}$ and  $C(M^{2n}, P^{k})$.

It is important to emphasize the following. One can try to ask if it is possible  to define similarly   the boundary operator $\tilde{d}$ on the set of all strata $W_{\mathfrak{S}}$  by    $\tilde {d}(W_{\sigma}) = \overline{W_{\sigma}} \setminus W_{\sigma}$. The answer is negative in general. Namely, as it will be shown in Section~\ref{Grassmann}, the Grassmann manifold $G_{7,3}$  is an example of a manifold that  belongs to our class,  but which contains  a stratum  $W_{\sigma}$ whose boundary is not the union of the strata. Therefore, the operator $\tilde{d}$ is  not  defined for $G_{7,3}$. This example is taken from~\cite{GS}.

We still want to note that there are important examples of $(2n,k)$-manifolds  for which the operator $\tilde{d}$ is well defined. It is proved in the paper~\cite{GS}  that the operator $\tilde{d}$ is well  defined  for the manifolds $G_{k+1,2}$ and $G_{6,3}$. In the paper~\cite{BT-1}, the cases of Grassmann manifold $G_{4,2}$ and   complex projective space $\C P^5$ with the canonical action of $T^4$  are studied in detail  . In these cases the operator $\tilde{d}$  is defined on  the complex of the  strata $W_{\mathfrak{S}}$.  Moreover, the canonical map $f : M^{2n}\to C(M^{2n}, P^k)$  gives a  mapping  from  the complex of strata  $W_{\mathfrak{S}}$ to  the  complex $C(M^{2n}, P^k)$.
More precisely:
\begin{lem}
If the operator $\tilde {d}(W_{\sigma}) = \overline{W_{\sigma}} - W_{\sigma}$ is defined on the set of strata  $W_{\mathfrak{S}}$ then the  map $f : M^{2n} \to C(M^{2n}, P^k)$ induces the map $f_{\mathfrak{S}} :  W_{\mathfrak{S}} \to C(M^{2n}, P^{k})$ which  is a bijective map  between these  complexes and commutes with the  boundary operators.
\end{lem}
\begin{proof}
Let $W_{\sigma}$ be a stratum. Then  $\mu (W_{\sigma}) = \stackrel{\circ}{P_{\sigma}}$ and Lemma~\ref{canonic} implies that $\widehat{\pi}(f(W_{\sigma})) =  \stackrel{\circ}{P_{\sigma}}$. By Corollary~\ref{homeom} we conclude that $f(W_{\sigma}) = P_{\sigma}^{'}$,  which is a cell of the complex $C(M^{2n}, P^k)$. Therefore, the map $f_{\mathfrak{S}} :  W_{\mathfrak{S}} \to C(M^{2n}, P^{k})$ is defined  by  $f_{\mathfrak{S}}(W_{\sigma}) = f(W_{\sigma})$.  The map $f_{\mathfrak{S}}$ is a bijection because of the above stated  bijection between the set of strata   $W_{\mathfrak{S}}$ and the set of admissible polytopes $P_{\mathfrak{S}}$. Moreover, it follows from Corollary~\ref{all} that $d\circ f_{\mathfrak{S}} = f_{\mathfrak{S}}\circ \tilde{d}$  which means that $f_{\mathfrak{S}}$ commutes with the boundary operators.
\end{proof}


\subsection{CW-topology on $C(M^{2n},P^k)$}\label{topCW}

The complex of all  admissible polytopes  $C(M^{2n},P^k)$  can be  naturally  endowed  with a   topology  such that $C(M^{2n},P^k)$ becomes CW-complex,  which we denote by $CW(M^{2n},P^k)$.
\begin{itemize}
\item The cells of these complex are the open polytopes $\stackrel{\circ}{P_{\sigma}}$ for $P_{\sigma}\in P_{\mathfrak{S}}$.
\item   The characteristic function on the boundary of the cells, which defines their attaching, is defined by the operator $d_{C}$. The skeletons  are defined inductively by the dimension of the cells.  The definition of the  operator $d_{C}$ verifies  that the cell axiom of CW-complex is satisfied.
\item According to the axioms of CW-complex,  it is defined on  $CW(M^{2n}, P^{k})$  the weak topology compatible  with  the cell decomposition.
\end{itemize}

Then the following is satisfied:
\begin{lem}
The canonical map $\widehat{\pi} : CW(M^{2n}, P^k)\to P^k$
\begin{itemize}
\item  is  a continuous map;
\item   is a  cell map  for the standard cell decomposition of the polytope $P^{k}$ if and only if $\hat{\pi}$  is a homeomorphism.
\end{itemize}
\end{lem}
\begin{proof}
To prove that the map $\widehat{\pi}$ is  continuous,  it is enough to notice that if $U\subseteq P^{k}$ is a closed set in $P^{k}$ then  $U\cap P$ is a closed set in $P$, for any polytope $P$ over some subsets of vertices of the polytope  $P^k$. This will be  true as well  for the  admissible polytopes, which  implies that $\widehat{\pi}^{-1}(U)$ is a closed set in $CW(M^{2n}, P^k)$.
The second statement  follows directly from Corollary~\ref{can-bij}.
\end{proof}

\begin{ex}
For a quasitoric manifold the canonical map $\widehat{\pi} : CW(M^{2n}, P^n)\to P^n$  is  a homeomorphism.
\end{ex}

\begin{prop}\label{canon-ncont}
If the set of admissible polytopes $P_{\mathfrak{S}}$ contains a   $k$-dimensional polytope different from $P^{k}$, the canonical map
$f : M^{2n}\to CW (M^{2n}, P^k)$ is not a continuous map.
\end{prop}

\begin{proof}
Let $P_{\sigma} \in  P_{\mathfrak{S}}$, $\dim P_{\sigma} =k$ and $P_{\sigma}\neq P^k$.  Then $P_{\sigma}^{'} \cong \stackrel{\circ}{P_{\sigma}}$ is an open set in $CW (M^{2n}, P^k)$.  Let us consider the corresponding stratum  $W_{\sigma}$, that is $\mu (W_{\sigma}) =\stackrel{\circ}{ P_{\sigma}}$. Since $\mu = \widehat{\pi}\circ f$, it follows that $f^{-1}(P_{\sigma}^{'}) = W_{\sigma}$, which is, by Remark~\ref{open},   not an open set in $M^{2n}$. 
\end{proof}

\begin{rem}
Note that  the same assumption as in  Proposition~\ref{canon-ncont} leads that  the canonical map  $\widehat{f} : M^{2n}/T^k\to CW (M^{2n}, P^k)$ is not continuous as well.
\end{rem}

\begin{ex}\label{G42}
We demonstrate Proposition~\ref{canon-ncont} in the case of  Grassmann manifold $G_{4,2}$ endowed with the canonical action of $T^4$.
Following~\cite{BT-1}, any four sided pyramid $P$  in $\Delta _{4,2}$ is an admissible polytope, which implies that $P^{'}\cong \stackrel{\circ}{P}$ is an open set in $CW(G_{4,2}, \Delta _{4,2})$. On the other hand $f^{-1}(P^{'})$ is an $(\C ^{*})^3$-orbit in the eight-dimensional manifold $G_{4,2}$, so it can not be an open set.
\end{ex}


\subsection{Quotient topology on $C(M^{2n},P^k)$}\label{topQ}
As  Proposition~\ref{canon-ncont} points, the CW-topology on  $C(M^{2n},P^k)$ is not compatible  with  the topology of $M^{2n}$.  Therefore,  the CW-topology  is not quite appropriate for the description of  a topology of the orbit space $M^{2n}/T^k$.

 We define  another topology on  the complex $C(M^{2n},P^k)$    such that the space $CQ(M^{2n},P^k)= C(M^{2n},P^k)$  becomes a   quotient space of $M^{2n}$ by the canonical map $f : M^{2n}\to C(M^{2n}, P^k)$.  More precisely,  we consider a  subset $U\subseteq C(M^{2n},P^k)$  to be  open if and only if the subset $f^{-1} (U)\subseteq M^{2n}$ is open. This is equivalent to say that  $CQ(M^{2n},P^k)$  is a  quotient space of $M^{2n}/T^k$ by the canonical 
map $\widehat{f} : M^{2n}/T^k\to C(M^{2n}, P^k)$.

\begin{lem}
The maps $\widehat{\pi} : CQ(M^{2n}, P^{k})\to P^{k}$,  $f: M^{2n}\to CQ(M^{2n}, P^{k})$ and  $\widehat{f} : M^{2n}/T^k \to CQ(M^{2n}, P^k)$ are continuous, canonical  maps.
\end{lem}

\begin{itemize}
\item Axiom~\ref{leaf} implies that a  face  of any polytope  $P^{'}\in CQ(M^{2n}, P^{k})$ is contained in the  boundary of $P^{'}$ regarded  to the quotient  topology.
\item Note also that $P^{k^{'}} = \stackrel{\circ}{P^{k}}$ is a dense set in  $CQ(M^{2n}, P^k)$ since $f^{-1}( P^{k^{'}}) = W$,   the main stratum  which is  a   dense set in $M^{2n}$.
\end{itemize}



\begin{lem}\label{HF}
If the set of admissible polytopes consists of $P^{k}$ and its faces, that is if the canonical map map $\widehat{\pi} : C(M^{2n}, P^{k})\to P^{k}$ is a bijection, then the space  $CQ(M^{2n}, P^k)$ is a Hausdorff topological space.
\end{lem}

\begin{lem}\label{T1}
If the set of admissible polytopes $P_{\mathfrak{S}}$ contains  a  polytope $P_{\sigma}$ such that $\stackrel{\circ}{P}_{\sigma}\subset \stackrel{\circ}{P^{k}}$, then the space
$CQ(M^{2n}, P^k)$ is not a Hausdorff topological space.
\end{lem}
\begin{proof}
Let $P_{\sigma}$ be   an  admissible polytope as stated in the formulation.  Then $\widehat{\pi}(P_{\sigma}^{'})\subset \stackrel{\circ}{P^{k}}$ and  for any point $x\in P_{\sigma}^{'}\subset CQ(M^{2n}, P^{k})$ there exists a point  $y\in P^{k^{'}}\subset CQ(M^{2n}, P^k)$   such that $\widehat{\pi}(x)=\widehat{\pi}(y)$.  Let further $V$ be an open set in $ CQ(M^{2n}, P^k)$ containing the point $x$. Then $ \widehat{f}^{-1}(V)$ is an open set in $M^{2n}/T^k$ and it contains  all points from $W_{\sigma}/T^k$  which   map to $\widehat{\pi}(x)$ by the map $\widehat{\mu}$.  On the other hand, by~\eqref{param-cl} there exists a  point $m\in \widehat{\mu}^{-1}(\widehat{\pi} (x))\cap W/T^k$ such that $m\in  \widehat{f}^{-1}(V)$. Note that $\widehat{f}(\widehat{\mu}^{-1}(\widehat{\pi}(x))\cap W/T^k) = y$ which  implies that $y\in V$. Thus every neighborhood of the point $x$ in $CQ(M^{2n}, P^k)$ contains a point $y\in P^{k^{'}}$.
\end{proof}


\begin{cor}
If the set of admissible polytopes $P_{\mathfrak{S}}$ is a pure set and contains  a  polytope different from $P^{k}$ and its faces then the space
$CQ(M^{2n}, P^k)$ is not a Hausdorff topological space.
\end{cor}

Let us discuss the  relation between the CW-topology   and the CQ-topology on   $C(M^{2n},P^k)$.

\begin{lem}
If the canonical map $\widehat{\pi} : C(M^{2n}, P^{k})\to P^{k}$ is a bijection then the CW-topology and the  CQ-topology on $C(M^{2n}, P^{k})$ coincide.
\end{lem}

\begin{ex}
For a quasitoric manifold $M^{2n}$ these two topologies on $ C(M^{2n}, P^{n})=P^n$ coincide.
\end{ex}

 As a  direct consequence of Proposition~\ref{canon-ncont}, we  also deduce the following:

\begin{lem}
If the set of admissible polytopes $P_{\mathfrak{S}}$ contains  a   $k$-dimensional polytope different from $P^{k}$, then there exists a set which  is open   in $CW (M^{2n}, P^k)$, but which is not open in $CQ(M^{2n}, P^k)$.
\end{lem}

Example~\ref{G42} demonstrates the situation  described in this Lemma.  The interior of any four-sided pyramid is an open set in  $CW(G_{4,2}, \Delta _{4,2})$, but it is not an open set in  $CQ(G_{4,2}, \Delta _{4,2})$.

Therefore, in general,  the sets in $C(M^{2n}, P^{k})$ which are open in the CW-topology are not necessarily open in the CQ-topology.

 The inverse inclusion does not hold as well, in general the sets in $C(M^{2n}, P^{k})$ which are open in the quotient topology are  not necessarily open in the CW-topology.
 We demonstrate this   in the case of  Grassmann manifold $G_{4,2}$.


\subsubsection{An example of a closed set  in  $CQ(G_{4,2}, \Delta_{4,2})$ that  is not closed in   $CW(G_{4,2}, \Delta_{4,2})$}

We follow the notation and the methods from~\cite{BT-1}.
Let us consider the set $\mathcal{C}$ in $G_{4,2}$  given by the matrices 
\[
\mathcal{C}=\left(\begin{array}{cccc}
1 & 0 & c & 1\\
0 & 1 & 1 &1
\end{array}\right) ,\;\; c\neq 0,1.
\]
This set  belongs to the main stratum since all its points  have all non-zero Pl\"ucker coordinates. We obtain the closure of $\mathcal{C}$   by attaching the limit points when  
$c\rightarrow 0,1, \infty.$
\begin{itemize}
\item When $c\rightarrow 0$, we obtain the point
\[
C_{0}=\left(\begin{array}{cccc}
1 & 0 & 0 & 1\\
0 & 1 & 1 & 1
\end{array}\right).
\]
\item When $c\rightarrow 1$, we obtain the point
\[
C_{1}=\left(\begin{array}{cccc}
1 & 0 & 1& 1\\
0 & 1 & 1 & 1
\end{array}\right) .
\]
\item When $c\rightarrow \infty$, in order  to see which   point  is obtaining,  we can proceed as follows. Since the set $\mathcal{C}$ belongs to all charts, it follows that  it can be  written down  in the local coordinates of the chart $M_{23}$. The points in the charts $M_{12}$ and $M_{23}$ are uniquely expressed   by the matrices:
\[
\left(\begin{array}{cccc}
1 & 0 & z_1 & z_2\\
0 & 1 & z_3 & z_4
\end{array}\right)  \;\; \text{and}\;\;
\left(\begin{array}{cccc}
w_1 & 1 & 0 & w_2 \\
w_3 & 0 & 1 & w_4
\end{array}\right) .
\]
Therefore, the transition map, on the intersection of these charts,   from the coordinates $(z_1,z_2,z_3,z_4)$ in the chart $M_{12}$ to the coordinates 
$(w_1, w_2,w_3, w_4)$ in the chart $M_{23}$ is given by the formulas
\[
w_1= -\frac{z_3}{z_1},\; w_2 = z_4-\frac{z_2z_3}{z_1},\; w_3 = \frac{1}{z_1},\; w_4=\frac{z_2}{z_1}.
\]
This  implies that the set $\mathcal{C}$ writes in the chart $M_{23}$ as
\[
\mathcal{C}=\left(\begin{array}{cccc}
-\frac{1}{c} & 1 & 0 & 1-\frac{1}{c}\\
 \frac{1}{c} & 0 & 1 & \frac{1}{c}
\end{array}\right) , \;\; c\neq 0,1, 
\]
and,   when $c\rightarrow \infty$, we obtain the point
\[
C_{\infty}=\left(\begin{array}{cccc}
0 &  1 & 0 & 1\\
0 & 0 & 1 & 0
\end{array}\right).
\]
\end{itemize}

Let us consider the closed set $\overline{\mathcal{C}} = \mathcal{C}\cup \{C_{0}, C_{1}, C_{\infty}\}.$  Its image  by the canonical map $f : G_{4,2}\to C(G_{4,2}, \Delta _{4,2})$ will be as follows:
\begin{itemize}
\item $f(\mathcal{C})\subset \stackrel{\circ}{\Delta} _{4,2}$, since $\mathcal{C}$ belongs to the main stratum.
\item $f(C_{0})\in \stackrel{\circ}{P_{23}}$, where $P_{23}$ is a four sided pyramid that  does not contain the vertex $\delta _{23}$.
\item $f(C_{1})\in \stackrel{\circ}{P_{34}}$, where $P_{34}$ is a four sided pyramid that  does not contain the vertex $\delta _{34}$.
\item $f(C_{\infty})\in (\delta _{23}, \delta _{34})$, where $(\delta _{23}, \delta _{34})$ is an edge with the vertices $\delta _{23}$ and $\delta _{34}$.
\end{itemize}

Therefore,  the set $f(\overline{\mathcal{C}})$ is not a closed set in $CW (G_{4,2}, \Delta _{4,2})$ since $f(\mathcal{C})\cap \Delta _{4,2} = \mu(\mathcal{C})\cup \mu(\{C_{\infty}\})$ which  it is not a closed set in $\Delta _{4,2}$, the points $\mu (C_{0}), \mu (C_{1})$ from the closure of $\mu (\mathcal{C})$ are missing. These points  belong to the open cells $\stackrel{\circ}{P}_{23}$ and $\stackrel{\circ}{P}_{34}$, respectively.

On the other hand, $f(\overline{\mathcal{C}})$   is obviously a  closed set in the quotient topology $CQ(G_{4,2}, \Delta _{4,2})$.

Therefore, the CW topology and the CQ  topology on $C(G_{4,2}, \Delta _{4,2})$ are essentially different: for each of these two topologies there  is   a set which is open   in  one topology, but  which is not open in the other.


Note that, in a general case when the operator $\tilde{d}$ is defined on the set of all  strata   $W_{\mathfrak{S}}$, it can be also defined  a  weak topology on  $W_{\mathfrak{S}}$ such that the boundary of any stratum  is given by the operator $\tilde{d}$. In this case $A\subseteq W_{\mathfrak{S}}$ is closed if and only if its intersection with the closure of any stratum  is closed. It is obvious that this topology coincide with the topology induced  from $M^{2n}$, since the topology on  strata  is induced from  the topology of a manifold  $M^{2n}$.


\subsection{Induced partial ordering on $CQ(M^{2n}, P^k)$}
There is a canonical way~\cite{E} to introduce a  preorder on  $CQ(M^{2n}, P^{k})$ using the quotient topology. More precisely
for $x,y\in CQ(M^{2n}, P^{k})$ one defines
\[
x\leq y \;\; \text{if and only if}\;\; x\in \overline{y},
\]
where $\overline{y}$ denotes the closure of $y$ in the quotient topology.
Note that on a general  Hausdorff topological space this preorder becomes trivial, which means that  $x\leq y$ implies $x=y$.
Since $\widehat{\pi} : CQ(M^{2n}, P^k)\to P^k$ is a continuous map, it follows:
\[
 x\leq y\;\;  \text{implies}\;\;   \widehat{\pi}(x) = \widehat{\pi}(y).
\]

 Obviously if the set of all admissible polytopes consists of $P^{k}$ and its faces this preorder will be trivial because in this case,  by Lemma~\ref{HF},  $CQ(M^{2n}, P^k)$ is a Hausdorff  space.  If this is not the case  from the  definition of the preorder, it directly follows:

\begin{lem}\
 If  $x\in P^{k^{'}}$  then there  is no $y\in CQ (M^{2n}, P^k)$, $y\neq x$  such that $x\leq y$.
  On the other hand,  if  $P_{\sigma}\neq P^{k}$ is an admissible polytope that  is not a face of $P^{k}$ then for any $x\in P_{\sigma}^{'}$ there exists a unique $y\in P^{k^{'}}$ such that $x\leq y$.
\end{lem}

Recall that, in a partially ordered set $(X, \leq)$,  an upper set is defined to be  a subset $U$ of $X$ such that if $x\in U$ and $x\leq y$ then $y\in U$ and,   accordingly, it  is defined a  lower set. It is a standard fact that, regarded to the specialization preorder on a topological space $X$, every open set in $X$  is an upper set  and every closed set in $X$  is a lower set.  Recall also that a topological space $X$  is said to be an P.~S.~Alexandrov space if the intersection of any family of open sets is an open set. Alexandrov topologies on $X$ are in one-to-one correspondence with preorders on $X$  meaning that $X$ is  an Alexandrov space if and only if every upper  set  regarded  to the specialization preorder is an open set. In particular, it implies that a Hausdorff topological space is an Alexandrov space if and only it is a discrete space.

As for the space $CQ(M^{2n}, P^k)$, for simplicity,  we consider   the case when the set of all admissible polytopes $P_{\mathfrak{S}}$ is a pure set. If  $P_{\mathfrak{S}}$ consists only of $P^{k}$ and its faces, it follows from  Lemma~\ref{HF} that $CQ(M^{2n}, P^k)$ is a Hausdorff topological space, which is obviously not discrete, so it is not an Alexandrov space. We prove that the same holds in general:

\begin{lem}
Assume that the set of admissible polytopes is a pure set and contains a polytope different from $P^k$ and its faces.  Then the  space $CQ(M^{2n}, P^k)$ is not a Hausdorff topological space and it  is not an    Alexandrov space as well.
\end{lem}

\begin{proof} Since $P_{\mathfrak{S}}$  contains a polytope $P_{\sigma}$ different from $P^k$ and its faces, it follows that  there exists a face $P_{\bar{\sigma}}$ of $P_{\sigma}$ such that $\stackrel{\circ}{P_{\bar{\sigma}}} \subset \stackrel{\circ}{P^k}$.  Let $V\subset CQ(M^{2n}, P^k)$  be the smallest,  under inclusion,  upper set that  contains $P^{k^{'}}$ and   $P_{\bar{\sigma}}^{'}$. Then $P_{\sigma}^{'}\not\subset V$ since there is no $x\in P_{\bar{\sigma}}^{'}$ such that $x\leq y$ for some $y\in P_{\sigma}^{'}$ as $\widehat{\pi}(x)\neq \widehat{\pi}(y)$. The set $V$ is not open in $ CQ(M^{2n}, P^k)$. Namely $f^{-1}(V)$ contains the stratum  over $\stackrel{\circ}{P_{\bar{\sigma}}}$  but it does not contain the stratum
over  $\stackrel{\circ}{P_{\sigma}}$, while  from Axiom~\ref{leaf} it follows  that there exist  points in the stratum  over $\stackrel{\circ}{P_{\bar{\sigma}}}$ which   are in the closure of the stratum  over $\stackrel{\circ}{P_{\sigma}}$.
\end{proof}

\begin{ex}
For the sake of clearness we provide  an   explicit  example of a set  which is an upper set in  $CQ(G_{4,2}, \Delta _{4,2})$, but which is not open. Let  $V\subset CQ(G_{4,2}, \Delta _{4,2})$ is given as the union of $\Delta_{4,2}^{'}$ and $P_{12, 34}^{'}$, where $P_{12,34}$ is a square in $\Delta _{4,2}$  which does not contain the vertices $\delta _{12}$ and $\delta _{34}$, see~\cite{BT-1}.   Then $V$ is not an open set in $CQ(G_{4,2}, \Delta _{4,2})$ since $f^{-1}(V)$ in not open in $G_{4,2}$. Namely, $f^{-1}(V)$  consists of the main stratum and of the  four-dimensional $(\C ^{*})^3$-orbit which maps to the square $\stackrel{\circ}{P}_{12,34}$. On the one hand, this orbit is   contained in the closures of the two six-dimensional $(\C ^{*})^3$-orbits,  which map to the pyramids $P_{12}$ and $P_{34}$ by the moment map. Therefore, the set $f^{-1}(V)$  is not  open in $G_{4,2}$, which  implies that the set $V$ is not open in $CQ(G_{4,2}, \Delta _{4,2})$. On the other hand,  it is obvious that $V$ in an upper set since for any $x\in V$ if $x\leq y$ it follows  that $x=y$ or $x\in P_{12,34}^{'}$, $y\in \Delta _{4,2}^{'}$ and $\widehat{\pi}(x) =\widehat{\pi}(y)$.
\end{ex}

\section{The  space $\mathfrak{E}(M^{2n}, P^k)$}

Using results from  previous sections, we can define the set  $\mathfrak{E}(M^{2n}, P^{k})$ over the complex $C(M^{2n}, P^{k})$ by:
\begin{equation}
\mathfrak{E}(M^{2n}, P^{k}) = \{(x,y)\in C(M^{2n}, P^{k})\times M^{2n}  : x\in P_{\sigma}^{'}, y\in W_{\sigma},  \widehat{\pi} (x)=\mu (y)\}.
\end{equation}

A topology on $\mathfrak{E}(M^{2n}, P^{k})$  is  defined  as the induced topology  by the embedding  $\mathfrak{E}(M^{2n}, P^k)\to CQ(M^{2n}, P^{k})\times M^{2n}$. In this  case the topology on  $CQ(M^{2n}, P^{k})$   can be obtained as a  quotient topology defined  by the map  $p: \mathfrak{E}(M^{2n}, P^{k}) \to CQ(M^{2n}, P^{k})$.


\begin{lem}\label{E-CH}
The space $\mathfrak{E} (M^{2n}, P^{k})$ is a compact  Hausdorff topological space.
\end{lem}

\begin{proof} The canonical  map $f: M^{2n}\to CQ(M^{2n}, P^{k})$ is a continuous, surjective map. The space $\mathfrak{E}(M^{2n}, P^{k})$ can be identified with the  graph of the map $f$.  Since the manifold $M^{2n}$ is a Hausdorff topological space, it follows that $\mathfrak{E}(M^{2n}, P^{k})$ is a Hausdorff space as well.
As for the compactness  let us consider a covering $E_{i}, i\in I$ of   the space $\mathfrak{E}(M^{2n}, P^{k})$ by an open sets. Without loss of generality, we may assume that  $E_i = (U_i\times V_i) \cap \mathfrak{E}(M^{2n}, P^{k})$ for $i\in I$, where $U_i$ and $V_i$ are open sets in $CQ(M^{2n}, P^k)$ and $M^{2n}$ respectively. Note that for any point $y\in M^{2n}$ there exists a point $x\in CQ(M^{2n}, P^{k})$ such that $(x,y)\in \mathfrak{E}(M^{2n}, P^{k})$, which  implies that $V_{i}, i\in I$ is an open covering for $M^{2n}$.  In addition,  for any point  $ x\in CQ(M^{2n}, P^{k})$,  there exists a point $y\in M^{2n}$  such that $(x,y) \in \mathfrak{E}(M^{2n}, P^{k})$,  which  implies that $U_{i}, i\in I$ is an open covering for $CQ (M^{2n}, P^k)$. The manifold   $M^{2n}$ is assumed to be a compact space,  which  implies that its  quotient space $CQ (M^{2n}, P^k)$  is a compact space as well. Therefore, there are finite sub-coverings $U_1,\ldots , U_s$ for $CQ(M^{2n}, P^k)$ and $V_1,\ldots, V_l$ for $M^{2n}/T^k$, which implies that $X_{ij} = (U_i\times V_j )\cap \mathfrak{E} (M^{2n}, P^{k})$ is a  finite sub-covering of $E_{i}$, $i\in I$ for $\mathfrak{E}(M^{2n}, P^{k})$.
\end{proof}

There are two natural  projections:
\begin{equation}\label{pr}
G_1:  \mathfrak{E} (M^{2n}, P^{k}) \to CQ(M^{2n}, P^k)\; \text{and} \; G_2 :  \mathfrak{E} (M^{2n}, P^{k})\to M^{2n}.
\end{equation}

\begin{itemize}
\item  The  maps  $G_1$ and $G_2$  are obviously surjective. For the map $G_1$, this  follows from the observation that   for any point $x\in P_{\sigma}^{'}\subset  CQ(M^{2n}, P^k)$ there exists
a point $y\in W_{\sigma}\subset M^{2n}$,  such that $\widehat{\pi}(x) = \mu (y)$. This is because   $\mu (W_{\sigma}) = \stackrel{\circ}{P_{\sigma}} = \widehat{\pi}(P_{\sigma}^{'})$, which implies that $(x,y)\in  \mathfrak{E}(M^{2n}, P^{k})$. As for the surjectivity of the map $G_2$, for any point $y\in W_{\sigma}\subset  M^{2n}$ and any  point $x\in P_{\sigma}^{'}$  such that $\widehat{\pi}(x) = \mu (y)$,  we see that $(x,y)\in  \mathfrak{E}(M^{2n}, P^{k})$. 
\item The  maps $G_1$ and $G_2$ are obviously continuous. It follows from the fact that $G_{1}^{-1}(U) =( U\times M^{2n})\cap  \mathfrak{E}(M^{2n}, P^{k})$  and $G_{2}^{-1}(V) = (CQ(M^{2n}, P^k) \times V) \cap  \mathfrak{E}(M^{2n}, P^{k})$  are open sets in  $\mathfrak{E} (M^{2n}, P^{k})$  for open sets $U\subseteq CQ(M^{2n}, P^k)$ and $V\subseteq M^{2n}$.
\end{itemize}

\begin{lem}
The map $G_2$ is injective, while the map $G_1$ is not injective.
\end{lem}
\begin{proof} The map  $G_2$ is injective since the condition $G_2(x_1,y)=G_2(x_2,y)$ implies that $\widehat{\pi} (x_1) = \widehat{\pi} (x_2) = \widehat{\mu}(y)$ and $x_1,x_2\in P_{\sigma}^{'}$. But, $\widehat{\pi} : P_{\sigma}^{'}\to \stackrel{\circ}{P_{\sigma}}$ is a homeomorphism,  which  implies $x_1=x_2$. The map  $G_1$ is not injective since  for any $y_1,y_2$ which belong to the same non-trivial $T^k$-orbit of  a stratum $W_{\sigma}$,  there exists a point $x\in P_{\sigma}^{'}$  such   that $\widehat{\pi}(x) = \mu (y_1) = \mu (y_2)$ and  $(x,y_1), (x,y_2)\in   \mathfrak{E}(M^{2n}, P^{k})$. It implies that $G_1(x, y_1)=G_1(x,y_2)=x$.
\end{proof}

Altogether this  leads to the following key result:

\begin{thm}\label{ME}
The  space $\mathfrak{E}(M^{2n}, P^{k})$ is homeomorphic to the space $M^{2n}$.
The homeomorphism  $G_2: \mathfrak{E}(M^{2n}, P^{k}) \to M^{2n}$ is given by the map $G_2(x,y)=y$.
\end{thm}

\begin{proof}
The map  $G_2: \mathfrak{E}(M^{2n}, P^{k}) \to M^{2n}$ is a continuous bijection. Since $\mathfrak{E}(M^{2n}, P^{k})$ is a  compact space   and $M^{2n}$ is Hausdorff, it follows that   elementary topology arguments lead that $G_2$ is a  homeomorphism.
\end{proof}

We define an action of the torus $T^{k}$ on $CQ(M^{n}, P^{k})\times M^{2n}$ using  the  given $T^k$-action on $M^{2n}$. Since the strata  as well as the almost moment map $\mu$ are invariant for this  torus  action, it follows that  this action  induces an action of the torus  $T^k$ on $\mathfrak{E}(M^{2n}, P^{k})$. We obtain 
\[
\mathfrak{E}(M^{2n}, P^{k})/T^k = \{(x,y)\in CQ(M^{n}, P^{k})\times M^{2n}/T^k  :  x\in P_{\sigma}^{'}, y\in W_{\sigma}/T^k,\widehat{\pi} (x)=  \widehat{\mu} (y)\}.
\]
\begin{rem}
Since $W_{\sigma}/T^k\cong \stackrel{\circ}{P}_{\sigma}\times F_{\sigma}$ and $P_{\sigma}^{'}\cong \stackrel{\circ}{P}_{\sigma}$, it follows that the points from the set $\mathfrak{E}(M^{2n}, P^k)/T^k$ can be represented  as the pairs $(x, c_{\sigma})$, where $x\in P_{\sigma}^{'}$ and $c_{\sigma}\in F_{\sigma}$.
\end{rem}

From  Lemma~\ref{E-CH} it follows

\begin{cor}
The space $\mathfrak{E}(M^{2n}, P^{k})/T^k$ is a compact  Hausdorff topological space.
\end{cor}

Note that the maps $G_1$ and $G_2$ defined by~\eqref{pr} are $T^k$-equivariant, where $CQ(M^{2n}, P^k)$ is considered to be  with the trivial $T^k$-action. Therefore,  they induce the maps of the corresponding orbit spaces $\widehat{G_1}: \mathfrak{E}(M^{2n}, P^{k})/T^k \to CQ(M^{n}, P^{k})$ and $\widehat{G_2}: \mathfrak{E}(M^{2n}, P^{k})/T^k \to M^{2n}/T^k$.

Combining the $T^k$-equaivariance of the map $\widehat{G_2}$ and Theorem~\ref{ME} we obtain :

\begin{thm}
The map $\widehat{G_2}: \mathfrak{E}(M^{2n}, P^{k})/T^k \to M^{2n}/T^k$  given by $\widehat{G_2}(x,y)=y$ is a homeomorphism.
\end{thm}

The map $\widehat{G_1}$ has the following important feature:

\begin{prop}\label{cfblocal}
The space $\widehat{G_1}^{-1}(P_{\sigma}^{'})$ is homeomorphic to the space   $P_{\sigma}^{'}\times F_{\sigma}$,  that is to the space $\stackrel{\circ}{P_{\sigma}}\times F_{\sigma}$ for any $\sigma \in \mathfrak{S}$.
\end{prop}
\begin{proof}
For any point  $x\in P_{\sigma}^{'}\in CQ(M^{n}, P^{k})$ we have that $\widehat{G_1}^{-1}(x) = \{ (x, y) | y \in W_{\sigma}/T^k, \widehat{\mu} (y) = \pi (x)\}$. It follows from Corollary~\ref{homeom} that the space $\widehat{G_1}^{-1}(x)$ is homeomorphic to the space $F_{\sigma}$. Then Corollary~\ref{homeom2} implies that the space  $\widehat{G_1}^{-1}(P_{\sigma}^{'})$ is homeomorphic to the space $\stackrel{\circ}{P_{\sigma}}\times F_{\sigma}$.
\end{proof}

\begin{rem}
We want to emphasize that, according to Proposition~\ref{cfblocal}, the space $\mathfrak{E}(M^{2n}, P^{k})/T^k$,  that is the orbit space $M^{2n}/T^k$, is   the  union  of the  trivial fiber bundles $\stackrel{\circ}{P_{\sigma}}\times F_{\sigma}$, where $\sigma$ runs through the set $\mathfrak{S}$. Then,  according to Axiom~\ref{leaf},   the gluing of fibers  $F_{\sigma}$ and $F_{\bar{\sigma}}$ for  $\stackrel{\circ}{P_{\sigma}}$ and $\stackrel{\circ}{P_{\bar{\sigma}}} \subset \partial \stackrel{\circ}{P_{\sigma}}$ respectively,  is given by the map $\eta _{\sigma, \bar{\sigma}} : F_{\sigma}\to F_{\bar{\sigma}}$.
\end{rem}


\section{The gluing of the  orbit spaces of strata  in $\mathfrak{E}(M^{2n}, P^k)/T^k$}
In the previous section we proved that the space  $\mathfrak{E}(M^{2n}, P^{k})/T^k$, that is the orbit space $M^{2n}/T^k$, is  the union of total spaces of the trivial fiber bundles $\stackrel{\circ}{P_{\sigma}}\times F_{\sigma}$, where  $\sigma \in \mathfrak{S}$.  In this section we want to describe how  the trivial bundles $\stackrel{\circ}{P_{\sigma}}\times F_{\sigma}$  are glued  together, or,  in other words,   to describe the  $\bar{\partial}$-boundary of $\stackrel{\circ}{P_{\sigma}}\times F_{\sigma}$. Recall that the $\bar{\partial}$- boundary of $\stackrel{\circ}{P_{\sigma}}\times F_{\sigma}$ is homeomorphic to the $\bar{\partial}$- boundary of the orbit space $W_{\sigma}/T^k$. Note that $\overline{W_{\sigma}/T^k} \cong \overline{W_{\sigma}}/T^k$, which implies that $\bar{\partial} (W_{\sigma}/T^k) \cong \bar{\partial}W_{\sigma}/T^k$.
Lemma~\ref{bound} immediately implies:
\begin{cor}
There is an embedding $\bar{\partial}(\stackrel{\circ}{P_{\sigma}}\times F_{\sigma})\subset \cup _{\tilde{\sigma}\subset \sigma}\stackrel{\circ}{P_{\tilde{\sigma}}}\times F_{\tilde{\sigma}}$, where $\tilde{\sigma}$ runs through all admissible subsets of the set $\sigma$.
\end{cor}

\begin{rem}
We want to point that for  a general $(2n,k)$-manifold, there might exist an  admissible subsets $\tilde{\sigma}$ of a admissible set $\sigma$ for which the polytope  $P_{\tilde{\sigma}}$ is not a face of the polytope $P_{\sigma}$. Furthermore, $P_{\tilde{\sigma}}$ might  not belong to  the boundary of the polytope  $P_{\sigma}$. Therefore, we {\it specially}  denote by $\bar{\sigma}$ those admissible subsets of a admissible set $\sigma$ for which $P_{\bar{\sigma}}$ is a face of the polytope  $P_{\sigma}$.
\end{rem}

Note that, as we will  demonstrate in Section~\ref{Grassmann},
 the $\bar{\partial}$-boundary of the orbit space  $W_{\sigma}/T^k$ can not be, in general, represented as a union of the orbit spaces of some other strata. This implies that, in general, the $\bar{\partial}$-boundary of  the space  $\stackrel{\circ}{P_{\sigma}}\times F_{\sigma}$ in  $\mathfrak{E}(M^{2n}, P^{k})/T^k$ can not be represented as a union of total spaces of some other trivial fiber bundles.

As for the boundary of the polytopes in $CQ(M^{2n}, P^k)$ we have the following:
\begin{lem}\label{subset}
If  $ P_{\tilde{\sigma}}^{'}\subset \partial P_{\sigma}^{'}$ for some any $P_{\sigma}^{'},  P_{\tilde{\sigma}}^{'} \subset CQ(M^{2n}, P^k)$  then  $\tilde{\sigma}\subset \sigma$.
\end{lem}
\begin{proof}
If  $P_{\tilde{\sigma}}^{'} \subset \partial P_{\sigma}^{'}$  it follows that  $f^{-1}( P_{\tilde{\sigma}}^{'})\subset f^{-1}(\partial P_{\sigma}^{'})$, where $f: M^{2n}\to CQ(M^{2n},P^k)$ is a quotient map. It implies that $f^{-1}(\partial P_{\sigma}^{'}) = \bar{\partial} f^{-1}( P_{\sigma}^{'}) = \bar{\partial} W_{\sigma}\subseteq \cup _{\tilde{\sigma}\subset \sigma}W_{\tilde{\sigma}}$ and, thus,  $f^{-1}( P_{\tilde{\sigma}}^{'}) = W_{\tilde{\sigma}}$ for some $\tilde{\sigma}\subset \sigma$.
\end{proof}

\begin{rem}
Recall that we already remarked that for any $P_{\bar{\sigma}}^{'}, P_{\sigma}^{'}\subset CQ(M^{2n}, P^{k})$,  such that $P_{\bar{\sigma}}$ is a face of the polytope $P_{\sigma}$, we have that  $P_{\bar{\sigma}}^{'} \subset \partial  P_{\sigma}^{'}$ in $CQ(M^{2n}, P^{k})$.
\end{rem}

At the end of this section we derive some results under the  additional assumption: if  $P_{\tilde{\sigma}}^{'}, P_{\sigma}^{'}\subset CQ(M^{2n}, P^k)$ and  $\partial P_{\sigma}^{'}\cap P_{\tilde{\sigma}}^{'}\neq \emptyset$ then  $P_{\tilde{\sigma}}^{'}\subset \partial P_{\sigma}^{'}$.

Then  Lemma~\ref{subset} implies:
\begin{cor}
The  boundary of any area $ P_{\sigma}^{'}\subset CQ(M^{2n}, P^k)$ is the union of some areas   $P_{\tilde{\sigma}}^{'}$ such that $\tilde{\sigma}\subset \sigma$.
\end{cor}

From  previous results we also have the following direct consequence:

\begin{cor}
For any point  $x\in P_{\sigma}^{'}$ it holds   $\overline{x} = \{y \in  P_{\tilde{\sigma}}^{'} \; | \:   P_{\tilde{\sigma}}^{'}\subset \partial P_{\sigma}^{'},\; \widehat{\pi}(y) =\widehat{\pi}(x)\}$.
\end{cor}

The  closure of the space of parameters $F_{\sigma}$  for  $P_{\sigma}^{'}$ can be described as follows:

\begin{lem}
For any point $x\in P_{\sigma}^{'}$ it holds: 
\[
\overline{x\times F_{\sigma}}   \subseteq \cup _{y\in \overline{x}}(y\times F_{\tilde{\sigma}}).
\]
\end{lem}
\begin{proof}
Let  $(y, c) \in \overline{x\times F_{\sigma}}$. . Then  $\widehat{G}_{2}(y,c) = a \in M^{2n}/T^k$ is  such a  point  that $a\in
\overline{W_{\sigma}}/T^k$ and $\widehat{f}(a) = y \in P_{\tilde{\sigma}}^{'} $, where  $P_{\tilde{\sigma}}^{'}\subset \partial P_{\sigma}^{'}$ and $y\in \overline{x}$.
\end{proof}

\begin{prop}\label{boundtr}
There is an embedding $\bar{\partial} (\stackrel{\circ}{P_{\sigma}}\times F_{\sigma})\subset \cup \stackrel{\circ}{P_{\tilde{\sigma}}}\times F_{\tilde{\sigma}}$,
where $\tilde{\sigma}$ runs through all admissible sets such that  $P_{\tilde{\sigma}}^{'}\subset \partial P_{\sigma}^{'}$
\end{prop}
\begin{proof}
Let, as before, $\widehat{f} : M^{2n}/T^k\to CQ(M^{2n}, P^k)$ be a quotient map and let a point $a\in  \mathfrak{E}(M^{2n}, P^{k})/T^k$ belongs to the $\bar{\partial}$-boundary of the space  $\stackrel{\circ}{P_{\sigma}}\times F_{\sigma}$. This means that the point $b=\widehat{F}^{-1}(a)$ belongs to the $\bar{\partial}$-boundary of the space $W_{\sigma}/T^k$. It implies that $\widehat{f}(b)$ belongs to the boundary of the polytope 
$P_{\sigma}^{'}$ in $CQ(M^{2n}, P^k)$.
\end{proof}

We can say more, that is which points form the union $\cup \stackrel{\circ}{P_{\tilde{\sigma}}}\times F_{\tilde{\sigma}}$ given by Proposition~\ref{boundtr} are for sure contained in the $\bar{\partial}$-boundary of the space  $\stackrel{\circ}{P_{\sigma}}\times F_{\sigma}$.

Let $\bar{\sigma}\subset \sigma$ be such a subset  that $P_{\bar{\sigma}}^{'}$ is a face of the polytope $P_{\sigma}^{'}$ and let $\eta _{\sigma, \bar{\sigma}} : F_{\sigma} \to F_{\bar{\sigma}}$ is a map introduced by~\eqref{eta}.  Put $F_{\sigma, \bar{\sigma}} = \eta _{\sigma, \bar{\sigma}}(F_{\sigma})$.

\begin{lem}
 For any $\bar{\sigma}\subset \sigma$, such that $P_{\bar{\sigma}}^{'}$ is a face of  the polytope $P_{\sigma}^{'}$, there is an embedding $\stackrel{\circ}{P_{\bar{\sigma}}}\times F_{\sigma, \bar{\sigma}}\subset \bar{\partial} (\stackrel{\circ}{P_{\sigma}}\times F_{\sigma})$.
\end{lem}
\begin{proof}
If  $P_{\bar{\sigma}}\times  c_{\bar{\sigma}}\subset  \stackrel{\circ}{P_{\bar{\sigma}}}\times F_{\sigma, \bar{\sigma}}$ then there exists a point $c_{\sigma}\in F_{\sigma}$ such that $\eta _{\sigma, \bar{\sigma}}(c_{\sigma}) = c_{\bar{\sigma}}$. It  means  that the $\bar{\partial}$-boundary of the leaf    $W_ {[\xi_{\sigma}, c_{\sigma}]}$  contains the leaf $W_ {[\xi_{\bar{\sigma}}, c_{\bar{\sigma}}]}$. It implies that  $W_ {[\xi_{\bar{\sigma}}, c_{\bar{\sigma}}]}/T^k \subset \bar{\partial} W_ {[\xi_{\sigma}, c_{\sigma}]}/T^k$, thus 
$\stackrel{\circ}{P_{\bar{\sigma}}}\times c_{\bar{\sigma}}\subset \bar{\partial} (\stackrel{\circ}{ P_{\sigma}}\times c_{\sigma})$.
\end{proof}

As for the points from $\stackrel{\circ}{P_{\tilde{\sigma}}}\times F_{\tilde{\sigma}}$, where $P_{\tilde{\sigma}}^{'}\subset \partial P_{\sigma}^{'}$,
we prove the following:
\begin{lem}
 Let $y\in P_{\tilde{\sigma}}^{'}$ where $P_{\tilde{\sigma}}^{'}\subset \partial P_{\sigma}^{'}$. Then there exists a point  $c_{\tilde{\sigma}}\in F_{\tilde{\sigma}}$ such that $(y, c_{\tilde{\sigma}})$ belongs to the $\bar{\partial}$-boundary of  the space $\stackrel{\circ}{P_{\sigma}}\times F_{\sigma}$.
\end{lem}
\begin{proof}
If  $P_{\tilde{\sigma}}^{'}\subset \partial P_{\sigma}^{'}$ then $\widehat{f}^{-1}(P_{\tilde{\sigma}}^{'}) \cap \bar{\partial} \widehat{f}^{-1}( P_{\sigma}^{'}) \neq \emptyset$.  It implies that  $W_{\tilde{\sigma}}/T^{\sigma} \cap \bar{\partial}  W_{\sigma}/T^k \neq \emptyset$ and, moreover, that $P_{\tilde{\sigma}}^{'}\subset \widehat{f} (W_{\tilde{\sigma}}/T^k \cap  \bar{\partial} W_{\sigma}/T^k)$, which means that for any point $y\in P_{\tilde{\sigma}}^{'}$ there exists a point $c_{\tilde{\sigma}}\in F_{\tilde{\sigma}}$ such that $(y, c_{\tilde{\sigma}})\in \partial (\stackrel{\circ}{P_{\sigma}}\times F_{\sigma})$.
\end{proof}


\section{A universal space of parameters}\label{universal}\label{Axiom-Univ}
In the theory of $(2n,k)$-manifolds there is an  effect, for  which we found an example in~\cite{GS}. It is about that there exists an $(2n,k)$-manifold $M^{2n}$ and  strata $W_{\sigma}, W_{\sigma ^{'}}\subset M^{2n}$ such that $W_{\sigma ^{'}}\cap \bar{\partial} W_{\sigma}\neq \emptyset$,  but $W_{\sigma^{'}}\not\subset \bar{\partial} W_{\sigma}$. The realization of this effect we elaborate      in Section~\ref{Grassmann} for the case  $M^{24} = G_{7,3}$. Note that according to~\cite{GS} this  effect does not appear in the case of $(4(n-2), n-1)$-manifolds $G_{n,2}$, although they are manifolds of the complexity  $n-3$. The manifolds $G_{n,2}$ and the orbits spaces $G_{n,2}/T^n$ are in the focus of attention due to the paper of Kapranov~\cite{Kap}. The considered  effect (we call it  Gel'fand-Serganova effect)  shows that  the description of the equivariant structure of $(2n,k)$-manifolds is a quite difficult problem. In~\cite{BTN} we proposed an approach for the   solution of this problem for the  Grassmann manifolds that  is   based on the  notion of a universal space  of parameters. This new notion we formalize    for $(2n,k)$-manifolds by the following axiom .

We recall that we use the following notation. For an admissible set $\sigma$ we denote by $\bar{\sigma}\subset \sigma$ such an    admissible set $\bar{\sigma}$ that $P_{\bar{\sigma}}$ is a face of $P_{\sigma}$. We denote by $F_{\sigma}$ the space of parameters of a stratum $W_{\sigma}$ and by $F_{i}$ the space of parameters of the stratum $W_{i}$ that  consists of $i$-th  fixed point.

\begin{ax}\label{universal}
There exists a topological space $\mathcal{F}$,   for any   $\sigma \in \mathfrak{S}$  there exist  topological spaces  $\tilde{F_{\sigma}}$ and  continuous  inclusions $I_{\sigma} : \tilde{F}_{\sigma}\to \mathcal{F}$ such that
\begin{itemize}
\item [a)] $\tilde{F} = F$  the space of parameters of the main stratum and   $\mathcal{F}$ is a compactification of $I(F)$,
\item [b)] $I_{\sigma}(\tilde{F_{\sigma}}) \subset I_{\bar{\sigma}}(\tilde{F}_{\bar{\sigma}})$ and $\mathcal{F} = \bigcup \limits_{i=1}^{m}I_{i}(\tilde{F}_{i})$,
\item [c)] for any $\sigma\in  \mathfrak{S}$ there exist continuous  projections $p_{\sigma} : \tilde{F}_{\sigma}\to F_{\sigma}$  such that  $p_{\bar{\sigma}} \circ \tilde{\eta} _{\sigma, \bar{\sigma}} = \eta _{\sigma, \bar{\sigma}} \circ p_{\sigma}$, where $\tilde{\eta}_{\sigma, \bar{\sigma}} : \tilde{F}_{\sigma} \to \tilde{F}_{\bar{\sigma}}$  is an inclusion given by the condition b).
\item [e)]  the map $H : \mathcal{E} = \cup _{\sigma}P^{'}_{\sigma}\times \tilde{F}_{\sigma} \to \mathfrak{E}(M^{2n}, P^k)/T^k$ defined by $H(x_{\sigma}, \tilde{c}_{\sigma}) = (x_{\sigma}, p_{\sigma}(\tilde{c}_{\sigma}))$ is a continuous map, where a  topology on $\mathcal{E}$ is  induced   by the embedding $\mathcal{E} \to CQ(M^{2n}, P^k)\times \mathcal{F}$.
\end{itemize}
\end{ax}

\begin{defn}
The space $\mathcal{F}$ is said to be the universal space of parameters and the spaces $\tilde{F}_{\sigma}$, $\sigma\in \mathfrak{S}$ are said to be the virtual spaces of parameters.
\end{defn}

It follows from Axiom~\ref{universal} that the orbit space $M^{2n}/T^k$ can be described in terms of the structural elements of $(2n,k)$-manifolds defined by our six  axioms.

\begin{thm}\label{mainm}
For any $(2n,k)$-manifold $M^{2n}$ the  orbit space $M^{2n}/T^k$  is homeomorphic to a quotient space of the space $\mathcal{E}$ by an equivalence relation $\approx$ such that  the map $H$ defines the homeomorphism $\hat{H} : \mathcal{E}/\approx \to \mathfrak{E}(M^{2n}, P^{k})/T^{k}$.
\end{thm}

\begin{proof}
The orbit space $M^{2n}/T^k$ is homeomorphic to $\mathfrak{E}(M^{2n},P^k)/T^k$ and the map $H : \mathcal{E} = \cup _{\sigma}P_{\sigma}^{'}\times \tilde{F}_{\sigma} \to \mathfrak{E}(M^{2n}, P^k)/T^k$ is surjective. Since $ \mathfrak{E}(M^{2n}, P^k)/T^k$  is a Hausdorff space,  the statement follows.
\end{proof}


\begin{rem}\label{inclusion}
It follows from the condition c) of Axiom~\ref{universal} that the inclusion  $\tilde{\eta}_{\sigma, \bar{\sigma}}: \tilde{F}_{\sigma}\to \tilde{F}_{\bar{\sigma}}$  is a lifting of the map $\eta _{\sigma, \bar{\sigma}}: F_{\sigma}\to F_{\bar{\sigma}}$. We elaborate this more closely.
Let $\tilde{c}_{\sigma}\in \tilde{F}_{\sigma}$ and let us consider the leaf $W_{[\xi _{\sigma}, p_{\sigma}(\tilde{c}_{\sigma}})]$, where $p_{\sigma}$ is given by the condition $c)$ of Axiom~\ref{universal}.  Then, by Axiom~\ref{leaf}, there exists a unique leaf $W_{[\xi _{\bar{\sigma}}, c_{\bar{\sigma}}]}$ that  belongs to the $\bar{\partial}$- boundary of   $W_{[\xi _{\sigma}, p_{\sigma}(\tilde{c})]}$.  Let  $y\in  W_{[\xi _{\bar{\sigma}}, c_{\bar{\sigma}}]}/T^{\bar{\sigma}}$ and $(y_n)$ a sequence of points from $W_{[\xi _{\sigma}, p_{\sigma}(\tilde{c})]}/T^{\sigma}$ that converges to the point $y$. By Theorem~\ref{mainm}, we have that
 $y = (x, [\tilde{c}_{\bar{\sigma}}]_{p_{\bar{\sigma}}})\in P^{'}_{\bar{\sigma}}\times \tilde{F}_{\bar{\sigma}}/p_{\bar{\sigma}}$
and  $y_n= (x_n, [\tilde{c}_{\sigma}]_{p_{\sigma}})\in P_{\sigma}^{'}\times \tilde{F}_{\sigma}/p_{\sigma}$, and $y_n$ converges to $y$ in the topology of $\mathcal{E}/H$.  Since $\eta _{\sigma, \bar{\sigma}}(p_{\sigma}(\tilde{c}_{\sigma})) = c_{\bar{\sigma}}$, it follows that the condition c) of Axiom~\ref{universal} implies that $\tilde{\eta}_{\sigma, \bar{\sigma}}(\tilde{c}_{\sigma})  \in p^{-1}_{\bar{\sigma}}(c_{\bar{\sigma}}) $
 \end{rem}

It immediately also  follows:
\begin{cor}
If a polytope  $P_{\sigma}$ is a face of the polytope $P^{k}$ then  $I(F)\subset I_{\sigma}(\tilde{F}_{\sigma})$. Moreover, 
\[
\partial P^{k}\times I(F) \subset \cup _{\sigma} \stackrel{\circ}{P}_{\sigma}\times I_{\sigma}(\tilde{F}_{\sigma}),
\]
 where $\sigma$ runs through {\bf all} admissible sets  such that $P_{\sigma}$ is a face of the polytope $P^{k}$.
\end{cor}
From  previous constructions it follows that,  for  $(2n,k)$-manifolds that  satisfy Theorem~\ref{main-main} or Theorem~\ref{join}, it holds  $\mathcal{F} = \bar{F}$,  where $\bar{F}$ is a notation used in these theorems. Moreover,  in the case of  such  manifolds $M^{2n}$,  for any admissible polytope $\stackrel{\circ}{P}_{\sigma}\subset \stackrel{\circ}{P^{k}}$  the  virtual space of parameters $\tilde{F}_{\sigma}$ coincides with its   space of parameters $F_{\sigma}$. 

 The manifolds $G_{4,2}$, $\C P^{5}$ and $F_3$ are   examples of  manifolds that  satisfy Theorems~\ref{main-main},~\ref{join}.
It is proved in~\cite{BT-1} that the universal spaces of parameters for $G_{4,2}$ and $\C P^{5}$ are the manifolds $\C P^1$ and $\C P^2$,  respectively. Proposition~\ref{F3} of the  current  paper  proves that the universal space of parameters for $F_3$ is $\C P^1$.  {\it The first non-trivial} example in this direction is  the Grassmann manifold  $G_{5,2}$.   

In~\cite{BTN} (Corollary 21)  it is proved that  the universal  space of parameters for $G_{5,2}$ can be taken to be to the blow up of $\C P^2$ at four  points. Moreover, the  manifold   $G_{5,2}$ provides  an example such  that $\mathcal{F} \neq \overline{F}_{x}$ for all points $x\in \stackrel{\circ}{\Delta}_{5,2}$ and that   virtual spaces of parameters $\tilde{F}_{\sigma}$  of strata are, in  general, wider then  spaces of parameters $F_{\sigma}$, that is the projections $p_{\sigma}: \tilde{F}_{\sigma}\to F_{\sigma}$ are not identity maps.  

As it is shown in the paper~\cite{BTN}, the spaces of parameters of the strata in $G_{5,2}$ over the pyramids $K_{ij}(7)\subset \Delta _{5,2}$,  $1\leq i<j\leq 5$ consist of a point (Corollary 12), while their virtual spaces of parameters are homeomorphic to $\C P^1$ (Theorem 11, Lemma 27), Recall that the pyramid  $K_{ij}(7)$ is a convex hull of the points $\delta _{kl}$, $kl\neq ij$ and $1\leq k<l\leq 5$. 

We will discuss  the case of  Grassmann manifold  $G_{5,2}$ in more details in Subsection~\ref{proofA5}


\section{Quasitoric manifolds $M^{2n}$ as $(2n,n)$-manifolds}
As  it is  presented in Subsection~\ref{qt}, a  quasitoric manifold $M^{2n}$ is  equipped with a smooth action of the torus $T^n$
and a smooth $T^n$-invariant map $\mu : M^{2n}\to P^{n}$ which  is induced by the projection $\pi : M^{2n}\to P^n$.

\begin{thm}
A quasitoric manifold has a structure of $(2n,n)$-manifold.
\end{thm}
\begin{proof}
Let $P_{v}\subset P^n$ denote the complement   to  the union of  those   faces of $P^{n}$ which   {\it  do not contain} the vertex $v$. The set $P_{v}$  is an open subset in $P^n$ and $M_{v} = \mu ^{-1}(P_{v})$ is  an open subset in $M^{2n}$. The set $M_{v}$ is $T^n$-invariant,  it contains exactly one fixed point $x_v$ and  $\mu (x_v) = v$. Moreover, the set $M_{v}$ is a  dense set in the manifold  $M^{2n}$. It follows  from the description of a model for a quasitoric manifold, see~\eqref{recov}, that $M_v$ is homeomorphic to the quotient space $(T^n\times P_{v})/\approx$,  which is further homeomorphic to the space  $(T^n\times \R ^{n}_{+}) \cong \C ^n$.   We take the sets $M_{i}=M_{v_i}$ as  charts for a quasitoric manifold $M^{2n}$ , where $i$ runs through the list of all vertices of the polytope $P^{n}$. In this way   we obtain  that Axiom~\ref{atlas} and  Axiom~\ref{bij} are satisfied. 

 Admissible polytopes $P_{\sigma}$ are the faces of the polytope  $P^n$ and $P^{n}$ itself. The strata   $W_{\sigma}$ are indexed by the sets   $\sigma$ that  run  through the set of vertices of all faces for the polytope   $P^n$. It directly follows that $\widehat{\mu} : W_{\sigma}\to \stackrel{\circ}{P_{\sigma}}$ and $W_{\sigma}/T^n \cong \stackrel{\circ}{P_{\sigma}}$, so Axiom 3 and Axiom~\ref{fiber} are also satisfied. 

We see that any stratum  $W_{\sigma}$ consists of one leaf and its $\bar{\partial}$-boundary is the union of the strata  over the faces of the corresponding admissible polytope $P_{\sigma}$, so Axiom~\ref{leaf} is satisfied as well. Axiom~\ref{universal} is obviously satisfied, the universal space of parameters can be taken to be  a point,  since for all strata the spaces of parameters are points.
\end{proof}


\section{$(2n,1)$-manifolds}

We first observe the following:
\begin{prop}
Any $(2n, 1)$-manifold is homeomorphic to the standard sphere $S^{2n}$
\end{prop}
\begin{proof}
By  definition, for an $(2n, 1)$-manifold $M^{2n}$there exists an almost moment map $\mu : M^{2n} \to [-1,1]$ which  is $S^1$-invariant. Then  Axiom 2 implies that the action of  $S^1$  on $M^{2n}$ has exactly two fixed points $A_1$, $A_2$ and Axiom 1 implies that $M^{2n}$ has an atlas consisting of two charts $(M_1,u_1), (M_{2},u_2)$ each of them containing exactly one fixed point. By Axiom~\ref{stab} and Axiom~ \ref{bij}, we have that $A_1 = W_1=M_1\cap Y_2$ and $A_2= W_2= Y_1\cap M_2$, where $Y_i=M^{2n}\setminus M_i$, $i=1,2$. It implies that  $M_1 = M^{2n}\setminus A_2$ and $M_2=M^{2n}\setminus A_1$. Since $u_i : M_{i}\to \R ^{2n}$, $i=1,2$ are homeomorphisms, it follows that $M^{2n}$ is the one-point compactification of $\R ^{2n}$ and hence,  it is homeomorphic to the sphere $S^{2n}$.
\end{proof}

The vice verse is true as well:
\begin{thm}
The standard sphere $S^{2n}$ has a structures of an $(2n, 1)$-manifold for any $n$.
\end{thm}
\begin{proof}
Represent  the sphere $S^{2n}$ as the hypersurface
\[
 |z_1|^2 +\cdots + |z_n|^2 +r^2= 1\; \text{in}\; \R ^{2n+1}\cong \C ^{n}\times \R
\]
and let us consider  the action of the circle $S^1$ on $S^{2n}$  defined by
\[
t(z_1,\ldots ,z_n, r) = (t^{\epsilon _1}z_1,\ldots ,t^{\epsilon _{n}}z_n, r), \;\; \text{where}\;\; \epsilon _{k}=\pm 1.
\]
The fixed point for this action are $A_{1} = (0,\ldots ,0,1)$ and $A_{-1}=(0,\ldots ,0,-1)$. An almost moment map we define   by
\[ \mu : S^{2n} \to P= [-1, 1],\;\;  \mu (z_1,\ldots ,z_n,r)=r. \]
It is straightforward to see that  Axiom~\ref{bij} is satisfied.

Let us consider the  atlas consisting of two charts  $(M_1,u_1), (M_{-1},u_{-1})$, where
\[
M_{1}=\{ (z,r), r\neq 1\}\; \text{and} \;  u_{1}(z,r)= \frac{1}{1-r}z,
\]
\[
M_{-1}\{ (z,r), r\neq -1\}\to \C ^{n}\; \text{and}\;  u_{-1}(z,r)=\frac{1}{1+r}z.
\]
The charts $M_{1}$ and $M_{-1}$ are $S^1$-equivariant, each of them contains exactly one fixed point and $\overline{M_1}=\overline{M_{-1}}= S^{2n}$, so Axiom~\ref{atlas} is satisfied. The only non point stratum  is the main stratum $W_{\{-1,1\}}=M_{1}\cap M_{-1} = \{(z,r)\in S^{2n} | r\neq -1,1\}$. The induced map $\widehat{\mu} : W_{\{-1,1\}}/S^1\to (-1,1)$ 
 is a fiber bundle with the   fiber $\C P^{n-1}$, so Axiom~\ref{fiber} is satisfied and the orbit space $W_{\{-1,1\}}/S^1$ is homeomorphic to the trivial bundle  $\C P^{n-1}\times (-1, 1)$.  The circle $S^1$ acts freely on $W_{\{-1,1\}}$,  so the projection $\pi : W_{\{-1,1\}}\to W_{\{-1,1\}}/S^1$ is a fiber bundle. Therefore, the leaf $W_{\{-1,1\}}[\xi, c]$    defined by~\eqref{leaf} is given as  $S^1\cdot c \times (-1,1)$, where $\xi : W_{\{-1,1\}}/S^1 \to \C P^{n-1}$ is a fixed projection  and $c\in \C P^{n-1}$. It  implies that  $\bar{\partial}$-boundary  of $W_{\{-1,1\}}[\xi, c]$ consists of the  two fixed points $A_{1}$ and $A_{-1}$, which  are  leafs over the vertices of  the interval $P$. In this way we see that Axiom~\ref{leaf} is satisfied as well.  Axiom~\ref{universal} is obviously satisfied. The universal space of parameters can be taken to be  $\C P^{n-1}$. It coincides with  the virtual spaces of parameters for all strata.
 \end{proof}
Since the only admissible polytopes for $P^{1}=[-1,1]$ are $P^{1}$ and its  faces $-1$, $1$ , it follows that Theorem~\ref{join} can be directly applied:
\begin{thm} 
It holds
\[ S^{2n}/T^1 = \partial P \ast \C P^{n-1},\; \text{where}\; P=[-1, 1].
\]
\end{thm}


 \section{Complex Grassmann manifolds $G_{k+1,q}$ as  $(2n,k)$-manifolds\\  with $n=q(k+1-q)$}

According to  Subsection~\ref{gr}, the complex Grassmann manifold $G_{k+1,q}$ is  canonically endowed  with an effective action of the torus $T^k$ and the smooth $T^k$-invariant moment map $\mu : G_{k+1, q}\to \R ^{k+1}$,  whose image is the hypersimplex $\Delta _{n, k}$.

Using the Pl\"ucker coordinates  and the corresponding atlas, as defined in Subsection~\ref{gr},   it is not difficult  to  prove:

\begin{prop}\label{gr-4}
The manifold $G_{k+1,q}$ has a structure that  satisfies the first five axioms of $(2n, k)$-manifolds, where  
$n=q(k+1-q)$.
\end{prop}

The detailed proof that  Axioms 1 - 4  are satisfied can be found in~\cite{BTN}, Section 2. We provide here the verification of Axiom~\ref{leaf}.


\subsection {Proof of Axiom~\ref{leaf}}\label{proofA5} (In order to avoid the confusion with indices ,  we use in this subsection the notation $z_{i,j}$  along with the common notation $z_{ij}$.)   It follows from the definition of a leaf and the description of  strata for $G_{k+1,q}$ (see~\cite{BTN}, Subsection 3.1), that any leaf $W_{\sigma}[\xi _{\sigma},c_{\sigma}]$ is a    $(\C ^{*})^{\sigma}$-orbit of a point from $W_{\sigma}$.
It immediately implies  that any  leaf  $W_{\sigma}[\xi _{\sigma},c_{\sigma}]$ is a smooth submanifold in $G_{k+1,q}$. On the other hand for any stratum $W_{\sigma}$  the closure of the $(\C ^{*})^{\sigma}$ -orbit of a point from $W_{\sigma}$   is a toric manifold.  By the result of~\cite{AT} the complement to this  $(\C ^{*})^{\sigma}$ -orbit in  this toric manifold  consists of  $(\C ^{*})^{\sigma}$-orbits of  smaller dimensions and the moment map gives a bijection between these orbits and the faces of  the polytope $P_{\sigma}$.  Moreover, the induced  moment map gives a diffeomorphism between $\overline{W_{\sigma}[\xi _{\sigma},c_{\sigma}]/T^{\sigma}}$ and $P_{\sigma}$ as  manifolds with corners.   In this  way we verify that the first  and the second conditions of Axiom~\ref{leaf} are satisfied for the Grassmann manifolds $G_{k+1,q}$. 

More precisely, we proved:
\begin{lem}
Let  $P_{\bar{\sigma}}$ be  a face of the polytope $P_{\sigma}$.  Then the map $\eta _{\sigma, \bar{\sigma}} : F_{\sigma}\to F_{\bar{\sigma}}$ is defined by  $\eta _{\sigma, \bar{\sigma}} (c_{\sigma}) = c_{\bar{\sigma}}$ such that
the leaf $W_{\bar{\sigma}}[\xi _{\bar{\sigma}}, c_{\bar{\sigma}}]$ is a unique
 $(\C ^{*})^{\bar{\sigma}}$-orbit over $P_{\bar{\sigma}}$ that  belongs to the $\bar{\partial}$-boundary of the  $(\C ^{*})^{\sigma}$-orbit $W_{\sigma}[\xi _{\sigma}, c_{\sigma}]$ over $P_{\sigma}$.
\end{lem}
We verify that the third condition of Axiom~\ref{leaf} is satisfied.  In order to do that, we prove the following result:
\begin{prop}\label{continuous}
 The map $\eta _{\sigma, \bar{\sigma}} : F_{\sigma}\to F_{\bar{\sigma}}$ is a continuous  map for any admissible sets $(\sigma, \bar{\sigma})$ such that $P_{\bar{\sigma}}$ is a face of the polytope $P_{\sigma}$.
\end{prop}
\begin{proof}

 Since the polytope $P_{\bar{\sigma}}$ is a face of the polytope $P_{\sigma}$, it follows that  they have a common vertex, so there always exists a chart $M_{I}$ such that both $W_{\sigma}$ and $W_{\bar{\sigma}}$ belong to this chart. Therefore, the proof that the map $\eta _{\sigma, \bar{\sigma}} : F_{\sigma}\to F_{\bar{\sigma}}$ is continuous can be realized  using the local coordinates of such a fixed chart.

We proceed with the proof  through the several   steps. We first show (according to the paper~\cite{BTN}),  that the space of parameters of a stratum $W_{\sigma}$ in $G_{k+1,q}$ can be embedded into  $(\C P_{B}^{1})^{N-l}$, where $\C P^{1}_{B} = \C P^{1}\setminus B$,  $B=\{(0:1), (1:0)\}$ and $N=(q-1)(k-q)$, $0 \leq l\leq N$.

Recall that we proved in~\cite{BTN} (Proposition 1)   that the space of parameters of the main stratum for the Grassmann manifolds $G_{k+1,2}$ can be embedded into  $(\C P^{1}_{A})^{k-2}$, where $\C P^{1}_{A} = \C P^{1}\setminus A$,  $A=\{(0:1), (1:0), (1:1)\}$. This embedding decomposes by  the  embedding $F\to (\C P^{1})^{n}/PGL(2, \C)$ (see also~\cite{Kap}) and the proof does not use the charts of the manifold $G_{k+1,2}$.   

We  construct now  an analogue embedding   for the space of parameters of  the main stratum $W$ of an arbitrary Grassmann manifold $G_{k+1,q}$, $q\geq 1$, using the coordinates  in a   chart.  The main stratum belongs to the intersection of all  charts $(M_I, u_I)$.  Fix the first chart $(M_{I}, u_{I})$, where $I =\{1,\ldots ,q\}$ with the local coordinates $z_{ij}$, $1\leq i \leq q$, $1\leq j\leq k+1-q$. The action of the algebraic torus $(\C ^{*})^{k+1}$ on $G_{k+1,q}$  induces  $(\C ^{*})^{k+1}$-action on   $\C ^{q(k+1-q)}$, which is,  according to~\cite{BTN} (Section 3), given by the representation
\[
(t_1,\ldots , t_{k+1}) \to (\frac{t_{q+1}}{t_1}, \ldots, \frac{t_{k+1}}{t_1}, \ldots , \frac{t_{q+1}}{t_q}, \ldots ,\frac{t_{k+1}}{t_{q}}).
\]
Let  $\tau _{j}= \frac{t_{q+j}}{t_1}$, $1\leq j\leq k+1-q$ and $\tau _{k+i-q} = \frac{t_{q+1}}{t_i}$, $2\leq i\leq q$.   We obtain  an
	effective action of the torus  $(\C ^{*})^{k}$ on $\C ^{q(k+1-q)}$, which is   given as the composition of the representation
\begin{equation}\label{rep}
(\tau _{1},\ldots ,\tau _{k}) \to (\tau _{1,1}, \tau_{1,2}, \ldots , \tau_{q,k+1-q}),\;\; \text{where}\;\;
\tau _{i,j} = \frac{\tau _{i}\tau _{ j}}{\tau _{1}},
\end{equation}
and the standard  action of the torus $(\C ^{*})^{q(k+1-q)}$ on
$\C ^{q(k+1-q)}$.

\begin{lem}\label{embmain}
The points of the $(\C ^{*})^{k}$-orbit  of a point ${\bf a} = (a_{1,1}, \ldots ,a_{q,k+1-q})$  of the main stratum
satisfy the equations
\begin{equation}\label{mainstr}
 c_{i,j}^{'}z_{1,1}z_{i,j}=c_{i,j}z_{i,1}z_{1,j}, \;\; 2\leq i\leq q, \; 2\leq j\leq k+1-q,
\end{equation}
where $c_{i,j}^{'} = a_{i,1}a_{1,j}$ and $c_{i,j} =a_{1,1}a_{i,j}$.
\end{lem}
\begin{proof}
If a point $(z_{1,1}, \ldots , z_{q,k+1-q})$  belongs to the $(\C ^{*})^{k}$-orbit  of a point ${\bf a} = (a_{1,1}, \ldots ,a_{q,k+1-q})$ then $z_{i,j} = \tau _{i,j}a_{i,j}$. It follows that $c_{i,j}^{'}z_{1,1}z_{i,j} = \tau _{1,1}\tau _{i,j}a_{i,1}a_{1,j}a_{1,1}a_{i,j}$, while $c_{i,j}z_{i,1}z_{1,j} = \tau_{i,1}\tau _{1,j}a_{1,1}a_{i,j}a_{i,1}a_{1,j}$. Since $\tau _{i,j} = \frac{\tau _{i}\tau _{ j}}{\tau _{1}}$, the statement follows from~\eqref{rep}.
\end{proof}

We find useful to note the following:
\begin{lem}\label{locpl}
Let $(a_{1,1}, \ldots , a_{q,k+1-q})$ be  the local coordinates of a point $L\in G_{k+1q}$  in a chart $M_{I}$. Then
\begin{equation}
a_{p,s}= P^{\hat{I}}(L), \;\; \hat{I} = (I\setminus \{s\})\cup \{p\}.
\end{equation}
 \end{lem}
\begin{proof}
Let  the matrix $A(L)$  represents an element $L$ in the chart $M_{I}$. Let us consider the submatrix $A_{\hat{I}}(L)$ of this matrix.   Since the submatrix $A_{I}(L)$  is an identity matrix and $\hat{I} = (I\setminus \{s\})\cup \{p\}$,  it follows that $a_{p,s}= P^{\hat{I}}(L)$.
\end{proof}

The Pl\"ucker coordinates of a point $L\in G_{k+1,q}$, up to common multiply, do not depend on the choice of a chart. All   Pl\"ucker coordinates of all points from the main stratum $W$ are non-zero, so  it  follows from  Lemma~\ref{locpl}  that all coordinates  $a_{i,j}$ are non-zero. Moreover, all    $(2\times 2)$-minors of the matrix $A(L) = (a_{i,j})_{1\leq i\leq q, 1\leq j\leq k+1-q}$ are non-zero.  It implies that $(c_{i,j}:c_{i,j}^{'})\notin A=\{ (1:0), (0:1), (1:1)\}$.
 Therefore, this  shows that  the  main stratum $W$,  written in the local coordinates of $M_{I}$, belongs to the  family of   algebraic manifolds  which are  given by  the system~\eqref{mainstr},  where $(c_{i,j}: c_{i,j}^{'})\in \C P^{1}_{A}$.

\begin{lem}
Let the coordinates of  a point ${\bf b} = (b_{1,1}, \ldots ,b_{q,k+1-q})$  of the main stratum  satisfy the  equations~\eqref{mainstr} of a point ${\bf a}$.  Then the point  ${\bf b}$ belongs to the $(\C ^{*})^{k}$-orbit of the point ${\bf a}$.
\end{lem}

\begin{proof}
Let $b_{i,1} = \tau _{i,1}a_{i,1}$, $1\leq i\leq q$  and $b_{1,j} = \tau _{1,j}a_{1,j}$, $2\leq j\leq k+1-q$, for some point $(\tau _{1,1}, \ldots \tau _{q,1}, \tau _{1,2}, \ldots \tau _{1,k+1-q})\in (\C ^{*})^{k}$. From~\eqref{mainstr} it follows that
\[
b_{i,2} = \frac{c_{i,2}b_{i,1}b_{1,2}}{c_{i,2}^{'}b_{1,1}}  = \frac{a_{1,1}a_{i,2}\tau _{i,1}a_{i,1}\tau_{1,2}a_{1.2}}{a_{i,1}a_{1,2}\tau _{1,1}a_{1,1}}= = \frac{\tau _{1,2}\tau_{i,1}}{\tau _{1,1}}a_{i,2} = \tau _{i,2}a_{,i2},
\]
where $\tau _{i.2}$ is as in~\eqref{rep}.

In the same way, we obtain  $b_{i,j}=\tau _{i,j}a_{i,j}$, which proves the statement.
\end{proof}

Altogether,  we obtain

\begin{prop}\label{mainembed}
The map $f : W \to (\C P^{1}_{A})^{N}$, where $N=(q-1)(k-q)$ given in the local coordinates of a chart $M_{I}$ by
\[
f(a_{1,1}, \ldots ,a_{q,k+1-q}) =
\]
\[  ((c_{2,2}:c_{2,2}^{'}), \ldots ,(c_{2,k+1-q}:c_{2,k+1-q}^{'}), \ldots, (c_{q,2}:c_{q,2}^{'}), \ldots, (c_{q,k+1-q}:c_{q,k+1-q}^{'})),
\]
\[
c_{i,j} =  a_{1,1}a_{i,j}, \;\; c_{i,j}^{'} = a_{i,1}a_{1,j},
\]
is $(\C ^{*})^{k}$-invariant, where $(\C P^{1}_{A})^{N}$  is considered with the trivial $(\C ^{*})^{k}$-action. Moreover,  it induces an embedding of the  space  of parameters $F = W/(\C ^{*})^{k}$ of the main stratum  into   $(\C P^{1}_{A})^{N}$.
\end{prop}

Applying  the same argument as for the main stratum we   obtain as well analogue result in the following case:
\begin{cor}
Let $W_{\sigma}$ be a stratum such that $W_{\sigma}\subset M_{I}$ and $u_{I}(W_{\sigma})\subset \{ {\bf z} = (z_{i,j})\in \
C ^{q(k+1-q)} \; |\; z_{i,j}\neq 0 \; \text{for all} \; (i,j)\}$. Then the space of parameters $F_{\sigma}$ of the stratum $W_{\sigma}$ can be embedded into $(\C P^{1}_{B})^{N}$, $N=(q-1)(k-q)$ as above.
\end{cor}

\begin{rem}
The condition that all  local coordinates  in a fixed chart $M_I$  of all points of a stratum $W_{\sigma}$ are  non-zero is the property  of  a fixed chart. Precisely, if we consider the Grassmann manifold $G_{4,2}$ then the stratum  $W_{34}$ defined by the condition that "$P^{ij}(W_{34})=0$ if and only if $(i,j)=34$, belongs to the intersections of the  charts $M_{12}$ and $M_{13}$". It is easy to check~\cite{BT-1} that all local coordinates for all points of the stratum $W_{34}$ in the chart $M_{12}$ are non-zero, while in the chart $M_{13}$ all points of this stratum  have  one  zero coordinate.
\end{rem}

The result similar to that in Proposition~\ref{mainembed} holds for any stratum whose space of parameters is not a point. We first recall that the definition of the strata~  as well as  results of~\cite{BTN} (Subsection 3.1)  imply:

\begin{lem} Let $W_{\sigma}$ be a stratum and let $M_{I}$ be a chart on the Grassmann manifold $G_{k+1,q}$. Then:
\begin{itemize}
\item The stratum $W_{\sigma}$ belongs to the chart $M_{I}$ if and only if  $I\in \sigma$, that is if and only if  $P^{I}(W_{\sigma})\neq 0$,
\item  Let  $W_{\sigma}\subset M_{I}$, $L_{0}\in W_{\sigma}$,  $u_{I}(L_0) = (z_{1,1}(L_0), \ldots ,z_{q,k+1-q}(L_0))$ and $z_{i,j}(L_0)=0$ for some $(i,j)$.  Then $z_{i,j}(L)=0$ for any   point  $L\in W_{\sigma}$,
\item Let  $J\subset \{(1,1), \ldots, (q,k+1-q)\}$ be such a subset of the set of indices for  the coordinates in a chart $M_{I}$, that $z_{i,j}(L_0)\neq 0$ if and only if  $(i,j)\in J$ for some point $L_{0}\in W_{\sigma}\subset M_{I}$. Then
 \[
u_{I}(W_{\sigma})\subset \C ^{J}=\{(z_{1,1}, \ldots , z_{q,k+1-q})\in \C ^{q(k+1-q)} \; | \: z_{i,j}\neq 0 \Longleftrightarrow  (i,j) \in J\}.
\]
\end{itemize}
\end{lem}
Here and further by the symbol $\C ^{J}$ we denote the  linear space of  maps $\{ J\to \C\}$.

In an analogous way as for the main stratum we prove:
\begin{lem}\label{l-28}
The space of parameters $F_{\sigma}\neq {\text pt}$ of a stratum $W_{\sigma}$ can be embedded into $(\C P ^{1}_{B})^{g}$ for some $1\leq g\leq N$.
\end{lem}
\begin{proof}
Let us consider a stratum  $W_{\sigma}\subset M_{I}$ and   assume that  
\[
u_{I}(W_{\sigma}) \subset \C ^{J} = \{ (z_{1,1},\ldots ,z_{q,k+1-q}) | z_{i,j}\neq 0,\; (i,j)\in J,\;\; z_{i,j} =0,\;(i,j)\notin J \},
\]
 where $J =\{(i_1,j_1),\ldots ,(i_l,j_l)\}\subseteq \{(1,1),\ldots ,(q,k+1-q)\}$. 

The algebraic torus $(\C ^{*})^{\sigma} = (\C ^{*})^{s}\subset (\C ^{*})^{l}$, $s\leq l$ is defined to be   a torus of the  maximal dimension  that  acts freely on the stratum $W_{\sigma}$.  There is a representation  $(\C ^{*})^{s}\to  (\C ^{*})^{l}$ obtained in an analogous way  as the  representation~\eqref{rep}. Namely, we take   the representation~\eqref{rep} composed with the projection $(\C^{*})^{q(k+1-q)} \to (\C ^{*})^{l}$ on the coordinate subspace $\C ^{J}$ and the canonical action of $(\C ^{*})^{l}$ on $\C ^{l}$. 

   Let  $\tau _{i_1,j_1}, \ldots ,\tau_{i_l,j_l}$  be  coordinates on the torus $(\C ^{*})^{s}$. Without loss of generality,   the coordinates  $\tau _{i_1,j_1}, \ldots ,\tau_{i_s,j_s}$ can be taken as coordinates for the torus $(\C ^{*})^{s}$. Then the   representation of $(\C ^{*})^{s}\to (\C ^{*})^{l}$ writes in these coordinates as an identity on the coordinates  $\tau _{i_1,j_1}, \ldots ,\tau_{i_s,j_s}$ and it is given by  $\tau _{ip,j_p} = \tau _{i_1,j_1}^{\varepsilon _{1p}}\cdots \tau _{i_s,j_s}^{\varepsilon _{sp}}$ for $s+1\leq p\leq l$,  where $\varepsilon _{rp}= 0,1,-1$  for any $1\leq r\leq s$. Therefore, as in the proof of Lemma~\ref{embmain}, we conclude that  the  points from the stratum $W_{\sigma}$,  written   in the local coordinates of the chart $M_{I}$, satisfy the following system of equations:
\begin{equation}\label{stratum-coord}
\left \{
\begin{array}{c}
c_{i_p,j_p}^{'}z_{i_p,j_p}\prod\limits_{\varepsilon _{rp} = -1} z_{i_r,j_r} =  c_{i_p,j_p}\prod\limits_{\varepsilon_{rp}=1}z_{i_r,j_r}, \;s+1\leq p\leq l,\; 1\leq r\leq s,\\
z_{i,j}=0,\; (i,j)\notin \{(i_1,j_1),\ldots ,(i_l,j_l)\},
\end{array}\right .
\end{equation}
where   $c_{i,j}, c_{i,j}^{'} \neq 0$.  In this way, an embedding  of the  space of parameters $F_{\sigma}$  into $(\C P^{1}_{B})^{l-s}$ is defined .
\end{proof}

\begin{rem}
The embedding described in Lemma~\ref{l-28} can be presented more explicitly in an  analogous way  as it is done in  Proposition~\ref{mainembed}. Namely according to the system of equations~\eqref{stratum-coord}  there is  the map $f_{\sigma} : W_{\sigma}\to (\C P^{1}_{B})^{l-s}$, which is in the local coordinates of the chart
$M_{I}$ given by
\[
f_{\sigma}(b_{i_1,j_1}, \ldots ,b_{i_l,j_l}) = ((c_{i_p,j_p}:c_{i_p,j_p}^{'}))_{s+1\leq p\leq l},
\]
\[
c_{i_p,j_p} = b_{i_p,j_p}\prod\limits_{\varepsilon _{rp}=-1}b_{i_r,j_r}, \;\; c_{i_p,j_p}^{'} = \prod\limits_{\varepsilon _{rp}=1}b_{i_r,j_r}, \; 1\leq r\leq s.
\]
This map is invariant for the considered  $(\C ^{*})^{s}$-action, so it induces  an embedding of the space of parameters $F_{\sigma} = W_{\sigma}/(\C ^{*})^{\sigma}$ into $(\C P^{1}_{B})^{l-s}$.
\end{rem}
We finally describe the map $\eta _{\sigma, {\bar \sigma}} : F_{\sigma} \to F_{\bar{\sigma}}$.
\begin{lem}
Let a polytope $P_{\bar{\sigma}}$ is a face of the polytope $P_{\sigma}$. Let us consider the above constructed embeddings of the spaces of parameters $F_{\bar{\sigma}}\subset (\C P^{1}_{B})^{r}$ and  $F_{\sigma}\subset (\C P^{1}_{B})^{b}$. Then $r<b$ and the space  $(\C P^{1}_{B})^{r}$ is a factor in the product $(\C P^{1}_{B})^{b}$. The  map $\eta _{\sigma, {\bar \sigma}} : F_{\sigma} \to F_{\bar{\sigma}}$  is given  by the projection $(\C P^{1}_{B})^{b}\to (\C P^{1}_{B})^{r}$.
\end{lem}
\begin{proof}
 Let us consider the leaf $W_{\sigma}[\xi_{\sigma}, c_{\sigma}]$ of the stratum $W_{\sigma}$. Let $W_{\bar{\sigma}}[\xi _{\bar{\sigma}}, c_{\bar{\sigma}}]$ be a unique leaf of  $W_{\bar{\sigma}}$ which is contained in the $\bar{\partial}$-boundary of $W[\xi_{\sigma}, c_{\sigma}]$. The coordinates of the points from the leaf  $W_{\sigma}[\xi _{\sigma}, c_{\sigma}]$ satisfy equations~\eqref{stratum-coord} for  a fixed point of the space  $(\C P^{1}_{A})^{l-s}$. It follows that the coordinates of the points from  the leaf   $W_{\bar{\sigma}}[\xi _{\bar{\sigma}}, c_{\bar{\sigma}}]$  in a fixed chart satisfy the following system: 
\begin{equation}
\left \{
\begin{array}{c}
c_{i_p,j_p}^{'}z_{i_p,j_p}\prod\limits_{\varepsilon _{rp} = -1} z_{i_r,j_r} =  c_{i_p,j_p}\prod\limits_{\varepsilon_{rp}=1}z_{i_r,j_r}, \; p\in V, \; V\subset \{s+1,\ldots,  l\},\\
z_{i_{p},j_{p}}=0 \; \text{for}\; p\in \{s+1, \ldots l\}\setminus V,\\
z_{i,j}=0,\; (i,j)\notin\{(i_1,j_1),\ldots ,(i_l,j_l)\}.
\end{array}\right .
\end{equation}
In this way, we see that  the space of parameters $F_{\bar{\sigma}}$ of the stratum $W_{\bar{\sigma}}$  can be embedded into $(\C P^{1}_{B})^{r}$, where $r={|V|}<b$, while   the space of parameters $F_{\sigma}$ of the stratum $W_{\sigma}$ is embedded into $(\C P^{1}_{B})^{b}$, $b=l-s$.  Moreover, the space $(\C P^{1}_{B})^{r}$ contains the space  $(\C P^{1}_{B})^{b}$ as its coordinate subspace,  with   the coordinates  indexed by $ i_{p}j_{p}$, where $p\in V$.  Therefore, the coordinate projection $(\C P^{1}_{B})^{b}\to (\C P^{1}_{B})^{r}$ gives the map  $\eta _{\sigma, \bar{\sigma}} : F_{\sigma} \to F _{\bar{\sigma}}$ 
\end{proof}

\begin{cor}
The map $\eta _{\sigma, \bar{\sigma}} : F_{\sigma} \to F _{\bar{\sigma}}$  is a continuous map.
\end{cor}
In this way we completed the proof of Proposition~\ref{continuous}.
\end{proof}

\begin{rem} The strata $W_{\sigma}$ coincide with the matroid strata defined in~\cite{GGMS}.  Moreover,  in the papers~\cite{GS} and~\cite{GGMS}, see also~\cite{BT-1} for the summary, the three more  stratifications of the Grassmann manifolds $G_{k+1,q}$ are defined: by "the soft Schubert cells", by the moment map and  by the arrangements
of planes, and it is showed that all these four stratifications are  equivalent.
\end{rem}
\begin{rem}
The closure of the  strata  $W_{\sigma}$ on the Grassmann manifolds are known in the literature as matroid varieties. These varieties do not behave nicely regarded  to the matroids assigned to them. In that context these varieties were recently studied in~\cite{FNR},~\cite{F}.
\end{rem}


\subsection{Axiom~\ref{universal} for the manifolds $G_{k+1, q}$} In order to verify Axiom~\ref{universal} we need to introduce  an universal space of parameters $\mathcal{F}$,  virtual spaces of parameters $\tilde{F}_{\sigma}$, continuous embeddings
 $I_{\sigma} : \tilde{F}_{\sigma}\to \mathcal{F}$ and projections $\tilde{F}_{\sigma}\to F_{\sigma}$, where $\sigma$ runs through the admissible sets.   We also need to prove that:
\begin{itemize}
\item[1)]  $\mathcal{F}$ is a compactification of the space of parameters $F=\tilde{F}$ of the main stratum and  $\mathcal{F} = \cup _{\sigma}I_{\sigma}(\tilde{F}_{\sigma})$,
\item[2)] The map $H : \mathcal{E} = \cup _{\sigma}P_{\sigma}^{'}\times \tilde{F}_{\sigma} \to \mathfrak{E}(G_{k+1, q}, \Delta_{k+1,q})/T^k$ is a continuous map. 
\end{itemize}
At this moment we have proved    Axiom~\ref{universal}   for    the Grassmannians $G_{4,2}$ and $G_{5,2}$. These results are described in detail   in~\cite{BT-1} and~\cite{BTN} according to the theory presented above. They turned out to be non trivial and already have been further developed in the papers of several authors. We formulate these results precisely.

\underline{ In the case of  Grassmann manifold $G_{4,2}$:}

\begin{itemize}
\item [(1)] $F= \tilde{F}=\C \setminus\{0,1\}$;
\item[(2)] $\mathcal{F}= \C P^1$;
\item [(3)] $F_{\sigma} = \tilde{F}_{\sigma}=pt$  for  the  strata $W_{\sigma}$ such that $\stackrel{\circ}{P}_{\sigma}\subset \stackrel{\circ}{\Delta}_{4,2}$;\\
 $F_{\sigma}=pt$, $\tilde{F}_{\sigma}=\C P^1$,   for the strata $W_{\sigma}$ such that   $P_{\sigma}\subset \partial \Delta_{4,2}$.
\end{itemize}
 Following the approach described above (see the details in~\cite{BT-1}), the embeddings $I_{\sigma} : \tilde{F}_{\sigma} \to \mathcal{F} $ can be described using the notion of  coordinates in a chart. We do  it in the chart  $M_{12}$.  The strata $W_{\sigma}$ such that $W_{\sigma}\subset M_{12}$ and $\stackrel{\circ}{P}_{\sigma}\subset \stackrel{\circ}{\Delta}_{4,2}$ are indexed by the admissible sets $\sigma_{ij}  =\{12, 13, 14, 23, 24, 34\}\setminus \{ij\}$, $ij\neq 12$,  and $\sigma_{13, 24}= \{12, 23, 14, 34\}$, $\sigma _{14, 23} = \{12, 23, 24, 34\}$.  For them it holds $I_{\sigma _{34}}(\tilde{F}_{\sigma_{34}})=(1:1)$,  $I_{\sigma _{14}}(\tilde{F}_{\sigma_{14}})= I_{\sigma _{23}}(\tilde{F}_{\sigma_{23}})=(0:1)$,  $I_{\sigma _{13}}(\tilde{F}_{\sigma_{13}})= I_{\sigma _{24}}(\tilde{F}_{\sigma_{24})}=(1:0)$ and  $I_{\sigma _{14, 24}}(\tilde{F}_{\sigma_{14,23}})= (0:1)$,  $I_{\sigma _{13, 24}}(\tilde{F}_{\sigma_{13, 24}})=(1:0)$. There are also  the strata $W_{\sigma}$ that do not belong to the chart $M_{12}$ but for which $\stackrel{\circ}{P}_{\sigma}\subset \stackrel{\circ}{\Delta}_{4,2}$. These are   the strata  $W_{\sigma_{1}}$, $\sigma _{1} = \{ 13, 14, 23, 24, 34\}$  and $W_{\sigma_2}$, $\sigma_{2} = \{13, 14, 23, 24\}$.
The strata $W_{\sigma _1}$ and $W_{\sigma _2}$   belong to the chart $M_{13}$.  In the paper~\cite{BT-1}, see proof of Proposition 4,  it is showed that the transition map from the coordinates of the   chart $M_{13}$ to the coordinates of the chart  $M_{12}$  induces a homeomorphism of the space of parameters of the main stratum $F=\C \setminus\{0,1\}$ that is given  by the involution $c\to \frac{c}{c-1}$.  This homeomorphism extends to a  homeomorphism of the universal space of parameters $\C P^1$.  The strata $W_{\sigma _1}$ and $W_{\sigma_2}$  in the chart $M_{13}$ are  limits of the main stratum $W$, when the parameter $c$ of the main stratum, written in the coordinates of the chart $M_{13}$, tends to   $\infty$. It  follows    that in the chart $M_{12}$ we have that $I_{\sigma_1}(\tilde{F}_{\sigma_1}) = I_{\sigma_2}(\tilde{F}_{\sigma_2}) = (1:1)$. The admissible polytopes $P_{\sigma}$  for the  other strata such that $W_{\sigma}\subset M_{13}\setminus (M_{12}\cap M_{13})$ belong to the boundary $\partial \Delta _{4,2}$. For them  $I_{\sigma}: \tilde{F_{\sigma}}\cong \C P^1 \to \C P^1$ is given by the map $c\to \frac{c}{c-1}$. In an analogous  way we describe these  embeddings in the coordinates of  the chart $M_{12}$ for the  strata which does not belong to the union  of these two charts.

\underline{In the case of    Grassmann manifold $G_{5,2}$:}
\begin{itemize}
\item[(1)] $F=\tilde{F}=\{((c_1:c_1^{'}), (c_{2}:c_{2}^{'}), (c_3:c_{3}^{'})) \in (\C P_{A}^1)^{3} | c_1c_{2}^{'}c_3 = c_{1}^{'}c_2c_{3}^{'}\}$;
\item[(2)] $\mathcal{F}$ is given as the blowup at the point $((1:1), (1:1), (1:1))$ of the compact non singular surface 
$\{((c_1:c_1^{'}), (c_{2}:c_{2}^{'}), (c_3:c_{3}^{'})) \in (\C P^1)^3 | c_1c_{2}^{'}c_3 = c_{1}^{'}c_2c_{3}^{'}\}$;
\item[(3)] An  explicit description of the virtual spaces of parameters  $\tilde{F}_{\sigma}$ for all admissible polytopes and the spaces of parameters $F_{\sigma}$ for all strata $W_{\sigma}$, as well as  the corresponding  embeddings $I_{\sigma} : \tilde{F}_{\sigma} \to \mathcal{F}$ and  projections $p_{\sigma} : \tilde{F}_{\sigma}\to F_{\sigma}$ is given in~\cite{BT-1}.
\end{itemize}

Together   with Proposition~\ref{gr-4}, we get
\begin{thm}
The complex Grassmann manifolds $G_{4,2}$ and $G_{5,2}$ have a canonical structure of $(8, 3)$ and $(10,4)$-manifolds respectively.
\end{thm}
\begin{rem}\label{gr-2nk}
We believe that the approach developed in~\cite{BTN}, in the case $G_{5,2}$,  for finding   virtual spaces of parameters $\tilde{F}_{\sigma}$ and an universal space of parameters  $\mathcal{F}=\cup _{\sigma}I_{\sigma}(\tilde{F}_{\sigma})$ together with the results of the current paper  brings  the proof of Axiom~\ref{universal}  for  all Grassmann manifolds $G_{k+1, q}$.
\end{rem}


\section{The complex manifold of complete flags}

The complete complex flag manifold $F_{k+1}$ consists of flags of $k$ complex subspaces $L =(L_1\subset L_{2}\subset \ldots \subset L_{k})$ in $\C ^{k+1}$. It is a homogeneous space, which can be  represented by $F_{k+1} = U(k+1)/T^{k+1}$.  As in the case of   complex Grassmann manifolds the canonical action    of the torus $T^{k+1}$ on $\C ^{k+1}$ induces an effective action  of the torus $T^{k}$ on the manifold $F_{k+1}$. This action extends to an action of the corresponding algebraic torus. The moment map $\mu : F_{k+1} \to \R ^{k+1}$ is defined by
\begin{equation}\label{momflag}
\mu (L) = \mu _{1}(L_1) + \mu _{2}(L_2)+\ldots +\mu _{k}(L_k),
\end{equation}
where $L=(L_1\subset L_2\subset \ldots \subset L_{k})$ and
\[
\mu _{i}(L_i) = \frac{\sum _{J}|P^{J}(L_i)|^2\delta _{J}}{\sum _{J}|P^{J}(L_i)|^2},
\]
where $J \subset \{1,\ldots k+1\}$, $|J|=i$.

The image $\mu _{i}(L_i)$ is the hypersimplex $\Delta _{k+1,i}$  since  it is the image by the moment map of the Grassmann manifold $G_{k+1,i}$. Therefore,  the image of the map $\mu$ is the Minkowski sum of the hypersimplices $\Delta _{k+1,i}$, $1\leq i\leq k$,  which  is known to be the   standard permutahedron $Pe^{k}$.  Recall that the standard permutahedron $Pe^{k}$  is a convex hull over the set of $(k+1)!$ points given by $\sigma(0,1,2,\ldots k)$, where $\sigma$ runs through the symmetric group $S_{k+1}$.

The algebraic torus $(\C ^{*})^{k+1}$ acts on $F_{k+1}$.  In an analogous way as for the complex Grassmann manifolds we prove:

\begin{thm}
The complete complex flag manifold $F_{k+1}$  satisfies the first five axioms   of an  $(2n, k)$-manifold, where $2n=k(k+1)$.
\end{thm}

We  provide the  proof in detail  for the case $k=2$ and, furthermore, we show that $F_3$ satisfies the sixth axiom as well.  For the manifold  $F_{k+1}$ the atlas required by  Axiom~\ref{atlas} can be constructed as follows: 

The charts are the sets $M_{i_1, i_1i_2,\ldots ,i_1\ldots i_k}=\{ L=(L_1\subset L_2\subset\ldots \subset L_{k})\in F_{k+1}\; |\; P^{i_1}(L_1)\neq 0, P^{i_1i_2}(L_2)\neq 0, \ldots , P^{i_1\ldots i_k}(L_k)\neq 0, \; 1\leq i_1,\ldots ,i_k\leq k+1, \; i_p\neq i_q, \; p\neq q, 1\leq p,q\leq k\}$. Then any point $L\in M_{i_1, i_1i_2, \ldots, i_1\ldots i_k}$ can be represented by the matrix $A_{L}$ such that $a_{i_jj}=1$ and $a_{i_jp}=0$, $j+1\leq p\leq k$. The coordinate map
$u_{i_1,i_1i_2, \ldots , i_1\ldots i_k} : M_{i_1, i_1i_2, \ldots , i_1\ldots i_k}\to \C ^{\frac{k(k-1)}{2}}$ is  given by
\[
u_{i_1,i_1i_2, \ldots , i_1\ldots i_k}(L) = (a_{ij}),\;  i\neq i_p\; \text{and}\; j<p, \; 1\leq p\leq k.
\]

\begin{prop}\label{F3}
 The manifold $F_3$ has a structure of  an $(6,2)$-manifold and $F_{3}/T^2 \cong S^1 \ast \C P^1\cong S^4$
\end{prop}
\begin{proof}
 We consider the atlas for  $F_{3}$ whose charts are given by
\[
M_{i,ij} = \{ L =(L_1\subset L_2) \in F_{3}\; |\: P^{i}(L_1)\neq 0, P^{ij}(L_2)\neq 0\}, \; 1\leq i,j\leq 3, \; i\neq j,
\]
and the homeomorphisms $u_{i,ij} : M_{i,ij}\to \C ^3$  are  defined as above.   Any point $L\in M_{i,ij}$ can be represented by the matrix $A_{L} = (a_{sl})$, $1\leq s\leq 3$, $1\leq l\leq 2$  such that $a_{i1}=1$, $a_{i2}=0$ and  $a_{j2}=1$. Then
$u_{i,ij}(L) = (a_{sl})$, where $sl\neq i1, i2,j2$.  This atlas is invariant under the canonical action of the algebraic torus $(\C ^{*})^{3}$.

Let us consider the chart $M_{1,12}$. Any point $L\in M_{1,12}$ represents by the matrix
\[
\left(\begin{array}{cc}
1 & 0\\
a_1 & 1\\
a_2 & a_3
\end{array}\right)
\;\; \text{and}\;\; u_{1,12}(L) = (a_1,a_2,a_3).
\]
The moment map~\eqref{momflag} in this chart writes as
\[
\mu (L) = \frac{1}{1+|a_1|^2 + |a_2|^2}((1,0,0)  + |a_1|^2(0,1,0)+|a_2|^2(0,0,1)) +\]
\[ + \frac{1}{1 + |a_3|^2+|a_1a_3-a_2|^2}((1,1,0)+|a_3|^2(1,0,1)+|a_1a_3-a_2|^2(0,1,1)).
\]
The action of the algebraic torus $(\C ^{*})^{2}$ on $\C ^{3}=u_{1,12}(M_{1,12})$,  induced by the action of $(\C ^{*})^{3}$ on $F_3$, is, in the chart $M_{1,12}$,  given by
\[
(t_1,t_2)\cdot (a_1,a_2,a_3) = (t_1a_1, t_2a_2, \frac{t_2}{t_1}a_3).
\]
Thus, the orbits  for this action   are as follows:
\begin{enumerate}
\item[(1)] if $a_i\neq 0$, $1\leq i\leq 3$ and $a_1a_3-a_2\neq 0$  it is the  hypersurface $\frac{z_2z_3}{z_1}=c$, where $c\in \C \setminus \{0,1\}$;
\item[(2)]if all $a_i\neq 0$ and $a_1a_3-a_2=0$ it is the hypersurface $\frac{z_2z_3}{z_1}=1$;
\item[(3)] if $a_i=0$ and $a_s, a_l\neq 0$  it is the coordinate subspace $\C _{ls}$;
\item[(4)] if $a_i, a_s=0$ and $a_l\neq 0$ it is the coordinates axis $\C_{l}$;
\item[(5)] if all $a_i=0$ it is the point $(0,0,0)$.
\end{enumerate}

Note that the orbits given by (1) form the main stratum and they are parametrized by $c\in \C \setminus \{0,1\}$. All other strata consist of one orbit. Since the main stratum is an everywhere dense  set in $F_{3}$, it follows that  all other orbits can be parametrized using this  parametrization of the main stratum.

The admissible polytopes for the chart $M_{1,12}$ are:
\begin{enumerate}
\item[(1)] the  hexagon $Pe^{2}$;
\item[(2)] the quadrilateral $Q$  over the vertices $(2,1,0), (2,0,1), (1,2,0), (1,0,2)$;
\item[(3)] the   three quadrilaterals $Q_i$, $1\leq i\leq 3$  defined by the vertices:
$Q_1 -  (2,1,0), (2,0,1), (1,0,2), (0,1,2)$; $Q_2 - (2,1,0), (2,0,1), (1,2,0), (0,2,1)$; $Q_3 - (2,1,0), (1,2,0), (0,2,1), (0,1,2)$;
\item[(4)] the  three  segments $I_{i}$, $1\leq i\leq 3$   defined by the vertices: $I_1 - (2,1,0), (0,1,2)$; $I_2 - (2,1,0), (2,0,1)$, $I_3  - (2,1,0), (1,2,0)$;
\item[(5)] the vertex $(2,1,0)$.
\end{enumerate}
We see that the orbit over the square $Q$ can be parametrized by the point $c=1$, the orbits over $Q_1$ by the point $c=\infty$, while the orbits over $Q_2$ and $Q_3$ by the point $c=0$. Also  the orbits over the intervals $I_1$ and $I_2$ can be parametrized by any point $c\in \C \cup \{\infty\}$, while the orbit over $I_3$ by the point $c=0$ and every  fixed point by any point  $c\in \C \cup \{\infty\}$. Note that $Q_2$, $Q_3$ and $I_3$ glue together to give the interior of the polytope  $Pe^{2}$ and the orbits over them  are all parametrized by the  point  $c=0$.

By considering the other charts it is easy to generalize this case and conclude  that all admissible polytopes are given by  the hexagon and its faces and  six quadrilaterals in the hexagon and their faces. The quadrilateral complementary to $Q$, as well as its edge that belongs to the interior of the  hexagon, can be  parametrized by the point $c=1$, while  the quadrilateral complementary to $Q_1$ as well as its edge that  belongs to the interior of hexagon can be parametrized  by the point $c=\infty$. In this way we prove that $\widehat{\mu}^{-1}(x) \cong \C P^1$ for all points $x$ from the interior of the hexagon. It follows from Theorem~\ref{join}  that $F_{3}/T^3\cong (\partial Pe^{2})\ast \C P^1$.

\end{proof}

\begin{rem} We believe that the approach developed for finding an universal space of parameters and  virtual spaces of parameters  in the case of Grassmann manifolds, can be in an analogous way  applied to the complete flag manifolds as well.
\end{rem}


\section{The orbit spaces of some  key examples}

Let us consider the Grassmann manifold $G_{4,2}$  of the complex two-dimensional subspaces in $\C ^{4}$ as an example of $(8,3)$-manifold over the hypersimplex $\Delta _{4,2}$.
 As it is remarked in Section 14.2,  all points $x\in \partial \Delta _{4,2}$ are simple and $\widehat{\mu}^{-1}(x)$ is homeomorphic to $\C P^1$ for all points  $x\in \stackrel{\circ}{\Delta} _{4,2}$. Theorem~\ref{join} implies that the orbit space $G_{4,2}/T^4$ is homeomorphic to the space $(\partial \Delta_{4,2})\ast \C P^1$. This statement is one of the key result of  the paper~\cite{BT-1}.

The action of  $T^4$ on the complex projective space $\C P^5$  is also studied in detail in~\cite{BT-1}. This action is  given as the composition of the second symmetric power representation $\Lambda ^{2} : T^4\to T^6$   and the canonical action of the torus $T^6$ on $\C P^6$. This action is not effective, but it induces an effective action of the torus $T^3$, for which the Pl\"ucker map $\mathcal{P} : G_{4,2}\to \C P^{5}$ is an equivariant map. As a result, we obtain on $\C P^5$ a structure of $(10,3)$-manifold with the almost moment map $\mu : \C P^5 \to \Delta _{4,2}$. The space of parameters $F$ of the main stratum can be identified with 
 $\{(c_1,c_2)\in \C ^{2} | c_1c_2\neq 0\}$. All points $x\in \partial \Delta _{4,2}$  are simple and  $\widehat{\mu}^{-1}(x)$ is homeomorphic to the complex projective plane $\C P^2$ for any point $x\in \stackrel{\circ}{\Delta} _{4,2}$.  Applying again  Theorem~\ref{join} we obtain that $\C P^5/T^4$ is homeomorphic to the space $(\partial \Delta_{4,2})\ast \C P^2$. Moreover, the Pl\"ucker embedding induces the embedding  $\hat{\mathcal{P}} : G_{4,2}/T^3\to \C P^{5}/T^3$, which,  by the stated homeomorphisms, produces   the embedding $(\partial \Delta_{4,2})\ast \C P^1\to (\partial \Delta_{4,2})\ast \C P^2$,  whose restriction to $\partial \Delta _{4,2}$ is given by an identity map, while its restriction  to $\C P^1$  is given by an inclusion in $\C P^2$. This  is one more key result of the paper~\cite{BT-1}.

Let us consider now the Grassmann manifold $G_{5,2}$ of the complex two-dimensional subspaces in $\C ^{5}$ as an example of $(12,4)$-manifold over the hypersimplex $\Delta _{5,2}$. In this case Theorem~\ref{join} does not apply, since not all  points from $\partial \Delta _{5,2}$ are  simple. More precisely, the points which belong to the interiors of the octahedra are singular. Nevertheless, in this  case, we can describe the orbit space $G_{5,2}/T^5$ by providing an explicit constructions of the universal space of parameters, virtual spaces of parameters as well as the construction of the corresponding  projection and embedding maps. This is realized in the paper~\cite{BTN} and  it is proved, as a  key result,  that $G_{5,2}/T^5$ is homotopy equivalent to the wedge $(S^{2}\ast \C P^1)\vee (S^3\ast \C P^2)$.     Note that $S^2\cong \partial\Delta _{4,2}$, $S^3\cong \partial\Delta _{5,2}$ and $S^2\ast \C P^1$ is homeomorphic to the orbit space  $G_{4,2}/T^4$.


\section{Examples for the construction of    virtual spaces of  parameters for $G_{k+1, q}$.}
The construction of  virtual spaces of parameters should use    the fact that the main stratum is an open,  dense set in  an  $(2n,k)$-manifold $M^{2n}$. The idea for introducing these spaces as well to call  them as virtual can be illustrated by  our work on the  description of the orbit space $G_{5,2}/T^5$. see \cite{BTN}.

We start by considering the fixed points in an $(2n,k)$-manifold $M^{2n}$ . For a  vertex $v$ of the polytope $P^{k}$ let $P^{v}_{\mathfrak{\sigma}} = \{ P_{\sigma}\in P_{\mathfrak{\sigma}} | v \in  P_{\sigma}\}$. We introduce a partial ordering on the set $P^{v}_{\mathfrak{\sigma}}$ by : $P_{\sigma _{1}} < P_{\sigma _{2}}$  if and only if  $P_{\sigma _{1}}$ is a face of the polytope $P_{\sigma _{2}}$. In particular $v < P_{\sigma}$ for any $P_{\sigma} \in P_{\mathfrak{\sigma}}^{v}$. Proposition~\ref{inclusion} implies that if  $P_{\sigma _{1}}<P_{\sigma _{2}}$ then $I_{\sigma _{2}}(\tilde{F}_{\sigma _{2}})\subset I_{\sigma _{1}}(\tilde{F}_{\sigma _{1}})$ and , in particular,  $I_{\sigma}(\tilde{F}_{\sigma})\subset I_{i}(\tilde{F}_{i})$, for any $P_{\sigma}\in P^{v}_{\mathfrak{\sigma}}$, where by $\tilde{F}_{i}$ is denoted the virtual space of parameters of the $i$-th vertex $v_i$.


In the case of  Grassmann manifold $G_{k+1, q}$ due to an action of the symmetric group $S_{k+1}$, we obtain
\begin{lem}
 The spaces $\tilde{F}_{i}$ and $\tilde{F}_{j}$  are  homeomorphic  for any two vertices $v_i, v_j$ of $\Delta _{k+1,q}$.
\end{lem}

In order to illustrate more closely  the idea for introducing  virtual spaces of parameters as well an   universal space of parameters,  let us consider the stratum $W_{\sigma}$, $\sigma =\{ 12, 13, 14, 15, 24, 34, 45\}$ in $G_{5,2}$.
  In the local coordinates $z_{ij}^{12}$ of the chart $M_{12}$ this stratum is given by the equations: $z^{12}_{11}=z^{12}_{31}=0$. Following~\cite{BT-1}, let  $(c_{i, 12}:c_{i,12}^{'})$, $1\leq i\leq 3$ be such  coordinates in $(\C P^{1})^{3}$  that the main stratum in the local coordinates of the chart $M_{12}$ is given by the system of equations $c_{1,12}^{'}z_{11}^{12}z_{22}^{12} = c_{1,12}z_{21}^{12}z_{12}^{12}$, $c_{2,12}^{'}z_{11}^{12}z_{32}^{12} = c_{2,12}z_{31}^{12}z_{12}^{12}$, $c_{3, 12}^{'}z_{21}^{12}z_{32}^{12}=c_{3,12}z_{31}^{12}z_{22}^{12}$.   The condition that $z^{12}_{11}, z^{12}_{31}\to 0$ implies that in $\bar{F}_{12}$ we have that  $(c_{1,12}:c_{1,12}^{'}) = (0:1)$, $(c_{3.12}:c_{3,12}^{'})=(1:0)$, while the limit of the points , $(c_{2,12}:c_{2,12}^{'})$ in $\bar{F}_{12}$ is {\it not defined}. Here $\bar{F}_{12}$ is the closure of the space of parameters $F_{12}$ of the main stratum in $\C P^1\times \C P^1\times \C P^1$  considered  in the chart $ M_{12}$.  We define  the virtual space of parameters  to be $\tilde{F}_{\sigma, 12} = \{((0:1), (c_{2, 12}:c_{2,12}^{'}), (1:0)), \; (c_{2,12}:c_{2,12}^{'})\in \C P^1\} \subset \bar{F}_{12}$. It contains all non defined points $(c_{2,12}:c_{2,12}^{'})\in \C P^{1}$.  Thus, the  virtual space of parameters $\tilde{F}_{\sigma, 12}$ resolves the singularities corresponding to uncertainties   when the points from the main stratum converge  to the points of  the stratum $W_{\sigma}$.

The situation is the same if consider the charts $M_{13}, M_{15}, M_{24}, M_{34}, M_{45}$ which  contain this stratum. Precisely, in the local coordinates of the chart $M_{13}$ this stratum is given by the equations $z_{11}^{13}=z_{31}^{13}=0$ and the virtual space of parameters is given by $\tilde{F}_{\sigma, 13} = \{((0:1), (c_{2,13}:c_{2,13}^{'}), (1:0)), \; (c_{2,13}:c_{2,13}^{'})\in \C P^1\} \subset \bar{F}_{13}$. In the local coordinate of the chart $M_{15}$, this stratum is defined by $z_{11}^{15}=z_{21}^{15}=0$ and as the virtual space of parameters we obtain $\tilde{F}_{\sigma, 15}=\{((c_{1,15}:c_{1,15}^{'}), (0:1), (0:1)), \; (c_{1,15}:c_{1,15}^{'})\in \C P^1\}\subset \bar{F}_{15}$.  In the chart $M_{24}$ the stratum $W_{\sigma}$ is given by $z_{22}^{24}=z_{32}^{24}= 0$ and the virtual space of parameters is given by $ \tilde{F}_{\sigma, 24} = \{((0:1), (0:1), (c_{3,24}:c_{3,24}^{'})), \; (c_{3,24}:c_{3,24}^{'})\in \C P^1\}\subset  \bar{F}_{24}$. Further, in the local coordinates of the chart $M_{34}$ this stratum is defined by $z_{22}^{34}=z_{32}^{34}=0$ and the virtual space of parameters is given by $ \tilde{F}_{\sigma, 34}= \{ ((0:1), (0:1), (c_{1,34}:c_{3,34}^{'})), (c_{3, 34}:c_{3,34}^{'})\in \C P^1\}\subset \bar{F}_{34}$.  In the local coordinates of the chart $M_{45}$ this stratum is defined by $z_{21}^{45}=z_{31}^{45}=0$ and the virtual space of parameters is given by $\tilde{F}_{\sigma, 45}= \{ ((1:0), (1:0), (c_{3,45}:c_{3, 45}^{'})), \; (c_{3, 45}:c_{3, 45}^{'})\in \C P^1\}\subset \bar{F}_{45}$.

This stratum belongs to the chart $M_{14}$ as well, in this chart all  coordinates of its points are non-zero and it is given by the equations:  $z_{11}^{14}z_{22}^{14}=z_{21}^{14}z_{12}^{14}, \; z_{11}^{14}z_{32}^{14}=z_{31}^{14}z_{12}^{14}, \; z_{21}^{14}z_{32}^{14} = z_{31}^{14}z_{22}^{14}$. These equations  imply that the virtual space of parameters $\tilde{F}_{\sigma, 14}$ , considering it as a subset in the closure $\bar{F}_{14}$, is given by the point $((1:1), (1:1), (1:1))\in\bar{F}_{14}$. According to Axiom~\ref{universal}, a  virtual space of parameters $\tilde{F}_{\sigma}$ is  defined by   the stratum $W_{\sigma}$  and its  construction  should not depend on the choice of the charts  that  contains $W_{\sigma}$.  The virtual space of parameters,   in the local coordinates of the  charts $M_{12}, M_{13}, M_{15}, M_{24}, M_{34}, M_{45}$, for the stratum we consider here, is  homeomorphic to $\C P^1$,  while in the chart $M_{14}$, when approaching    the limit  point  $((1:1), (1:1), (1:1)))\in \bar{F}_{14}$, a singularity  does not appear. Therefore, in order to obtain the space $\tilde{F}_{\sigma, 14}$  we need to consider  the blow up of the space $\bar{F}_{14}$  at the point $((1:1), (1:1), (1:1)))$.


\section{Gel'fand-Serganova example}~\label{Grassmann}
According to  Lemma~\ref{bound}, for any stratum $W_{\sigma}$  there is an inclusion  $\bar{\partial}W_{\sigma} \subseteq  \cup W_{\overline{\sigma}}$, where ${\overline{\sigma}}$ runs  through all  admissible  subsets of  $\sigma$ given. We show here that this inclusion is strict in general, meaning   that there exist such  strata $W_{\sigma}$ and $W_{\bar{\sigma}}$ for which 
\[
\bar{\sigma}\subset \sigma;\;\;      \bar{\partial} W_{\sigma}\cap W_{\bar{\sigma}}\neq \emptyset,\;\;    \text{but}\;\; W_{\bar{\sigma}}\not\subset \bar{\partial} W_{\sigma}.
\]
 Such  a pair  of strata $W_{\sigma}$ and $W_{\bar{\sigma}}$ is given in  the paper of Gel'fand-Serganova~\cite{GS}.  The space of parameters of the strata  $W_{\sigma}$ in $G_{7,3}$  is a point, while its $\bar{\partial}$-boundary has non empty intersection with the stratum  $W_{\bar{\sigma}}$ whose space of parameters has non-zero dimension.
\

Consider  a point $L\in G_{7,3}$ given by the  matrix
\[
\left(\begin{array}{ccccccc}
1 & 0 & 0 & 0 & b_1 & c_1 & d_1\\
0 & 1 & 0 & a_2 & 0 & c_2 & d_2\\
0 & 0 & 1 & a_3 & b_3 & 0 & d_3
\end{array}\right),
\]
such that
\begin{equation}\label{cond}
d_1d_2d_3\neq 0, \;\; d_1c_2=d_2c_1,\;\; d_1b_3=d_3b_1, \;\; d_2a_3=a_2d_3, \;\; a_i, b_i, c_i\neq 0.
\end{equation}
Form the condition~\eqref{cond} it follows that
\[ a_3b_1c_2=a_2b_3c_1.
\]
The charts  on the Grassmann manifolds are  indexed  by the non-zero Pl\"ucker coordinates (see  Subsection~\ref{gr}).  It implies  that the set of points  which   satisfy~\eqref{cond} form  the stratum  $W$ obtained as the intersection of the sets $X_{ijk}$, $1\leq i<j<k\leq 7$, where $X_{ijk} = Y_{ijk}$ for  $ijk = 126, 135, 234, 147, 257, 367$ and $X_{ijk}=M_{ijk}$ for the other indices.

On the other hand, according to Subsection~\ref{gr}, the $(\C ^{*})^{6}$-orbit of a  point\\  $(0, b_1, c_1, d_1, a_2, 0, c_2, d_2, a_3, b_3, 0, d_3)$,  in the local coordinates of the chart $M_{123}$, is given by
\[
(0,\frac{t_5}{t_1}b_1, \frac{t_6}{t_1}c_1, \frac{t_7}{t_1}d_1, \frac{t_4}{t_2}a_2, 0,  \frac{t_6}{t_2}c_2, \frac{t_7}{t_2}d_2, \frac{t_4}{t_3}a_3, \frac{t_5}{t_3}b_3, 0, \frac{t_7}{t_3}d_3),
\]
which can be  written as
\begin{equation}\label{orb}
(0, \tau _{1}b_1, \tau _{2}c_1, \tau _{3}d_1, \tau _{4}a_2, 0, \tau _{5}c_2, \frac{\tau _{3}\tau _{5}}{\tau _{2}}d_2, \tau _{6}a_3, \frac{\tau _{1}\tau_{5}\tau _{6}}{\tau _{2}\tau _{4}}b_3, 0, \frac{\tau _{3}\tau _{5}\tau _{6}}{\tau _{2}\tau _{4}}d_3).
\end{equation}
Therefore, all points of the form~\eqref{orb}  belong to the one $(\C ^{*})^{6}$-orbit,  that is the stratum $W$ consists of this one  $(\C ^{*})^{6}$-orbit.

Consider a point $L_1\in G_{7,3}$ given by the matrix
\begin{equation}\label{pointm}
\left(\begin{array}{ccccccc}
1 & 0 & 0 & 0 & b_1 & c_1 & 0\\
0 & 1 & 0 & a_2 & 0 & c_2 & 0\\
0 & 0 & 1 & a_3 & b_3 & 0 & 0
\end{array}\right),
\end{equation}
where  $a_3b_1c_2=a_2b_3c_1$.  Let   $W^{'}$ be a stratum  that  contains this point. It is the intersection of the sets $X_{ijk}$ where $X_{ijk}= M_{ijk}$ for $ijk =123, 124, 125$, $135, 136, 235,236, 456$ and $X_{ijk}=Y_{ijk}$ for the other indices. It implies that the stratum  $W^{'}$ consists of the points that can be represented by the matrices
\[
\left(\begin{array}{ccccccc}
1 & 0 & 0 & 0 & b_1 & c_1 & 0\\
0 & 1 & 0 & a_2 & 0 & c_2 & 0\\
0 & 0 & 1 & a_3 & b_3 & 0 & 0
\end{array}\right),
\]
such that $a_3b_1c_2\neq - a_2b_3c_1$.

On the one hand,  it is obvious that the point $L_1$ from the stratum $W^{'}$ as well as its $(\C ^{*})^{6}$-orbit,  which is given by
$(0, \tau _{1}b_1, \tau _{2}c_1, 0, \tau _{3}a_2, 0, \tau _{4}c_2, 0, \tau _{5}a_3, \frac{\tau _{1}\tau_{4}\tau _{5}}{\tau _{2}\tau _{3}}b_3, 0, 0)$ such that  $a_3b_1c_2=a_2b_3c_1$,   belong to $\bar{\partial}$- boundary of the stratum  $W$.

On the other hand,  the stratum  $W^{'}$ does not coincide with the  $(\C ^{*})^{6}$-orbit of a point $L_1$. Namely, this  $(\C ^{*})^{6}$-orbit is contained in $W^{'}$,  but its points  satisfy relation $a_3b_1c_2= a_2b_3c_1$, while the points from $W^{'}$ are defined by the  weaker relation  $a_3b_1c_2\neq - a_2b_3c_1$.

It implies that
\[
\bar{\partial} W \cap W^{'}\neq \emptyset \;\; \text{and}\;\; W^{'} \not\subset \bar{\partial}W .
\]


Gel'fand-Serganova leads to the following important comment:
\begin{cor}
The map $\eta _{\sigma, \bar{\sigma}} : F_{\sigma}\to F_{\bar{\sigma}}$, whose existence is stated by Axiom~\ref{leaf}, is not a surjection in general.
\end{cor}

\begin{rem}
It follows from the results of the paper  of Gel'fand-Serganova~\cite{GS} that the map $\eta _{\sigma, \bar{\sigma}} : F_{\sigma}\to F_{\bar{\sigma}}$ is a surjection for the Grassmann manifolds $G_{k+1,2}$ and $G_{6,3}$ .
\end{rem}

\bibliographystyle{amsplain}

 Victor M.~Buchstaber\\
Steklov Mathematical Institute, Russian Academy of Sciences\\
Gubkina Street 8, 119991 Moscow, Russia\\
E-mail: buchstab@mi.ras.ru
\\ \\ \\

Svjetlana Terzi\'c \\
Faculty of Science and Mathematics, University of Montenegro\\
Dzordza Vasingtona bb, 81000 Podgorica, Montenegro\\
E-mail: sterzic@ucg.ac.me
\end{document}